\documentclass[reqno]{amsart}

\usepackage[a4paper,margin=1.5in]{geometry}

\usepackage{mathrsfs}
\usepackage[absolute, overlay]{textpos}

\usepackage{amsfonts,  amsmath}
\usepackage{graphicx}
\usepackage{enumerate,  enumitem,  amscd}
\usepackage{amsthm,  amssymb,  epsfig,  hyperref,  amsmath}  
\usepackage{mathtools}
\usepackage[table,xcdraw]{xcolor}
\usepackage{blkarray}

\usepackage{tikz}
\usepackage{tikz-cd}
\usetikzlibrary{arrows,automata}
\usepackage{quiver}
\usepackage{spectralsequences}
\usepackage{algorithm}
\usepackage{algpseudocode}

\newcommand{\pres}[3]{\textnormal{#1} \langle #2 \mid #3 \rangle}

\newcommand{\kk}{\mathbf{k}}
\newcommand{\C}{\mathbb{C}}

\newcommand{\Z}{\mathbb{Z}}
\newcommand{\ZM}{\mathbb{Z}M}
\newcommand{\Q}{\mathbb{Q}}
\newcommand{\ab}{\mathrm{ab}}
\newcommand{\cR}{\mathcal{R}}
\newcommand{\fg}{\mathfrak{g}}
\newcommand{\orb}{\mathrm{orb}}
\newcommand{\cE}{\mathcal{E}}

\newcommand{\xra}[1]{\xrightarrow{}^{\ast}_{#1}}
\newcommand{\lra}[1]{\xleftrightarrow{}^{\ast}_{#1}}
\newcommand{\xr}[1]{\xrightarrow{}^{}_{#1}}

\DeclareMathOperator{\Fib}{Fib}
\DeclareMathOperator{\Irr}{Irr}

\DeclareMathOperator{\PSL}{PSL}
\DeclareMathOperator{\SL}{SL}
\DeclareMathOperator{\GL}{GL}

\DeclareMathOperator{\FP}{FP}

\DeclareMathOperator{\im}{im}

\DeclareMathOperator{\Hom}{Hom}
\DeclareMathOperator{\Ext}{Ext}
\DeclareMathOperator{\Tor}{Tor}
\DeclareMathOperator{\Isom}{Isom}

\newtheorem{theorem}{Theorem} 
\newtheorem*{theorem*}{Theorem} 
 
\numberwithin{theorem}{section}
\newtheorem{lemma}[theorem]{Lemma}     
\newtheorem{corollary}[theorem]{Corollary}
\newtheorem*{corollary*}{Corollary}
\newtheorem{proposition}[theorem]{Proposition}
\newtheorem*{proposition*}{Proposition}
\numberwithin{equation}{section}

\theoremstyle{definition}

\newtheorem*{question*}{Question}
\newtheorem{remark}[theorem]{Remark}
\newtheorem{remark*}[theorem]{Remark}
\newtheorem{example}{Example}
\numberwithin{example}{section}

\tikzcdset{scale cd/.style={every label/.append style={scale=#1},
    cells={nodes={scale=#1}}}}

\begin{document}

\title[Profinite rigidity and homology of $3$-orbifold groups]{On profinite rigidity, Grothendieck pairs, and the second homology of some $3$-orbifold groups}

\author{Carl-Fredrik Nyberg-Brodda}
\address{June E Huh Center for Mathematical Challenges, Korea Institute for Advanced Study (KIAS), Seoul 02455, Korea}
\email{cfnb@kias.re.kr}

\thanks{\newline The author is currently supported by KIAS Individual Grant HP094701 at Korea Institute for Advanced Study, and by the Mid-Career Researcher Program (RS-2023-00278510) through the National Research Foundation funded by the government of Korea.}

\date{\today}


\begin{abstract} The second homology group is of central importance in the study of profinite rigidity of $3$-manifold groups. Although general and deep results imply that the integral homology of cocompact hyperbolic $3$-orbifold groups is computable in principle, the resulting algorithm is not practical. We develop an effective method for computing $H_2$ in the case of orbifold groups arising as finite extensions of the fundamental group of hyperbolic rational homology $3$-spheres. As a special case, this yields explicit computations of the second homology groups of all cocompact lattices between $\pi_1(\mathcal W)$ and its normalizer in $\PSL_2(\mathbb C)$, where $\mathcal W$ is the Weeks manifold. We also show that these lattices are absolutely profinitely rigid, completing work by Bridson, McReynolds, Reid \& Spitler in this setting. As a special case, we determine that $H_2(\Gamma_{\mathcal O}, \mathbb Z) \cong \Z / 2\Z$, where $\Gamma_{\mathcal O}$ is the normalizer of the group of units $\Gamma_{\mathcal{O}}^1$ in a choice of maximal order $\mathcal O$ of the quaternion algebra associated to $\mathcal W$, thereby answering a question of Bridson \& Reid. Although this non-vanishing obstructs one possible construction of Grothendieck pairs in $\Gamma_{\mathcal O}^1 \times \Gamma_{\mathcal O}^1$, we use our computations to show the vanishing of the second homology of another lattice whose derived subgroup is $\Gamma_{\mathcal{O}}^1$, which then yields Grothendieck pairs in this direct product by a theorem of Bridson \& Reid. Finally, to showcase the generality of the techniques, we also compute the second homology of some finite extensions by orientable isometries of the fundamental group of some Fibonacci manifolds $M_n$. 
\end{abstract}

\maketitle 

\noindent The history of group theory is permeated with questions concerning how much can be deduced about an infinite group by considering only finitary data. While the mid-1950s flood of general undecidability for finitely presented groups swept away any hope of a uniform approach for understanding this wild class of groups, sifting through the flotsam and jetsam left in its wake uncovered many new specialised approaches. In particular, the natural question arises of when, and how much of, a group is determined by its finite quotients. This question is most naturally made formal in the question of \textit{profinite rigidity}: given two finitely generated residually finite groups $G$ and $H$ with the same finite quotients, is it necessarily the case that $G \cong H$? The present article is concerned with several interrelated questions, all ultimately tracing back to this central problem. 

Profinite rigidity has been at the focal point of much recent research in combinatorial and geometric group theory. Some positive results are easy to come by; indeed, some groups are obviously profinitely rigid. For example, any finitely generated residually finite group $G$ with the same finite quotients as a finitely generated abelian group $A$ is easily seen to be abelian, and hence by considering rank and torsion, we necessarily find $G \cong A$. Thus all finitely generated abelian groups are profinitely rigid. However, beyond this class, general results are scarce and often difficult to prove. For example, Baumslag \cite{Baumslag1974} constructed pairs of non-isomorphic nilpotent (even virtually cyclic) groups with the same finite quotients. Pickel \cite{Pickel1971, Pickel1973}, relying on deep results by Borel \& Serre \cite{Borel1964}, showed that for any finitely generated virtually nilpotent group $G$, there are at most finitely many other isomorphism classes of finitely generated residually finite groups with the same finite quotients as $G$ (i.e.\ the \textit{genus} of $G$ is finite). However, beyond this class the situation quickly turns wild: for metabelian groups, the genus can be infinite \cite{Pickel1974}, and for solvable groups the genus can even be of cardinality $2^{\aleph_0}$, as recently proved by Nikolov \& Segal \cite{Nikolov2021}. Remeslennikov (see \cite[Question~12]{Noskov1979}) famously asked whether finitely generated free groups are profinitely rigid; this remains a tantalizing open problem, as does the corresponding problem for surface groups. 

Thus the landscape of understanding which general finitely presented groups are profinitely rigid appears at present as rather difficult to navigate, to say the least. On the other hand, in the cases that the groups are of geometric origin, more tools have recently become available. One such class that has proved particularly fruitful is that of cocompact lattices in $\PSL_2(\C)$, i.e.\ fundamental groups of closed hyperbolic $3$-orbifolds. The first profinite rigidity results for full-sized lattices (i.e.\ those containing non-abelian free subgroups) were obtained by Bridson, McReynolds, Reid \& Spitler \cite{BMRS}, who proved profinite rigidity for, among other groups, the Bianchi group $\PSL_2(\Z[\omega])$ with $\omega^2 + \omega + 1 = 0$, and the fundamental group $\Gamma_\mathcal{W}$ of the Weeks manifold $\mathcal{W}$, the closed hyperbolic $3$-manifold of minimal volume. The same authors subsequently applied the same strategy to other arithmetic groups in \cite{BMRSTriangle}, proving profinite rigidity also for many triangle groups (which are lattices in $\PSL_2(\mathbb{R})$), cf.\ also \cite{Conder2021}.

Any hyperbolic $3$-manifold group $G = \pi_1(M)$ has, as a consequence of Mostow--Prasad Rigidity \cite{Mostow1968, Prasad1973}, a finite $\Isom(M)$, which is isomorphic to $\operatorname{Out}(\pi_1(M))$. Thus, every conjugacy class of subgroups in $\Isom(M)$ gives rise to a finite extension of $G$ that is also a lattice in $\Isom(\mathbb{H}^3)$. Hence, the question of the relation between the profinite rigidity of these lattices and that of $G$ naturally arises. As a guiding example, consider the Weeks manifold $\mathcal{W} = \mathbb{H}^3 / \Gamma_{\mathcal{W}}$, the unique closed hyperbolic $3$-manifold of minimal volume. Let $k$ be the unique cubic field of discriminant $-23$, and let $B_{\mathcal{W}}$ be the quaternion algebra over $k$ and ramified at the real place and at the prime of norm $5$ in $k$. Letting $\mathcal{O}$ denote any maximal order in $B_{\mathcal{W}}$, the lattice $\Gamma_{\mathcal{W}}$ arises as the unique subgroup of index $12$ in the normalizer $\Gamma_{\mathcal{O}} < \PSL_2(\mathbb{C})$ of the group of units $\Gamma_{\mathcal{O}}^1$ of $\mathcal{O}$. Then $\Gamma_{\mathcal{W}} <_3 \Gamma_{\mathcal{O}}^1 <_4 \Gamma_{\mathcal{O}}$, and $\Gamma_{\mathcal{O}} / \Gamma_{\mathcal{W}} = \Isom(\mathcal{W}) \cong D_6$, the dihedral group of order $12$. 

There are, in total, ten lattices between $\Gamma_{\mathcal{W}}$ and $\Gamma_{\mathcal{O}}$ (inclusive) in $\PSL_2(\C)$, each corresponding to a conjugacy class of subgroups in $D_6$. Bridson, McReynolds, Reid \& Spitler \cite{BMRS} established that $\Gamma_{\mathcal{O}}^1$ and $\Gamma_{\mathcal{O}}$ are profinitely rigid, along with two other lattices (the orbifold groups of $9_{49}(2)$ and $7_1^2(2,3)$ in the notation of \cite{Mednykh1998}, see Figure~\ref{Fig:Weeks-lattice}). Our first result is to extend their analysis to all intermediate lattices. 

\begin{theorem*}[See Theorem~\ref{Thm:Weeks-are-profinitely-rigid}]
Let $\Gamma$ be a lattice $\Gamma_{\mathcal{W}} \leq \Gamma \leq \Gamma_{\mathcal{O}}$, i.e.\ a finite extension of the Weeks manifold group $\Gamma_{\mathcal{W}}$ by some group of isometries. Then $\Gamma$ is (absolutely) profinitely rigid.
\end{theorem*}

A closely related line of inquiry comes from Grothendieck pairs. Let $\phi \colon \Lambda \to \Gamma$ be an embedding between two finitely generated, residually finite groups with $\Lambda \neq \Gamma$. Then $(\Gamma, \Lambda)$ is a \textit{Grothendieck pair} if $\phi$ lifts to an isomorphism between the profinite completions of $\Lambda$ and $\Gamma$ (see \S\ref{Subsec:profinite-rigidity}). Grothendieck \cite{Grothendieck1970} asked about the existence of such pairs where both $\Lambda$ and $\Gamma$ are finitely presented; this was answered positively by Bridson \& Grunewald \cite{Bridson2004}, cf.\ also \cite{Platonov1986} for the finitely generated case.

Bridson \& Reid \cite[Theorem~A]{BridsonReidC} recently showed that the group $\Gamma_{\mathcal{W}} \times \Gamma_{\mathcal{W}}$ is not profinitely rigid, and that furthermore all of the failures of this rigidity appear as Grothendieck pairs: if a finitely generated residually finite group $\Lambda$ has the same set of finite quotients as $\Gamma_{\mathcal{W}} \times \Gamma_{\mathcal{W}}$, then there is a Grothendieck pair $(\Gamma_{\mathcal{W}} \times \Gamma_{\mathcal{W}}, \Lambda)$. Their technique for constructing such Grothendieck pairs comes from their ``Theorem~C'' (see Theorem~\ref{ThmC-BridsonReid}). This theorem guarantees for certain groups $G$ the existence of uncountably many pairwise non-isomorphic finitely generated groups $P_\lambda$ and Grothendieck pairs $(G \times G, P_\lambda)$. The difficult part of the condition for $G$ is, roughly speaking, to exhibit a group $\Gamma$ with $H_2(\Gamma, \Z) = 0$ such that $G = [\Gamma, \Gamma]$. Their application of Theorem~C to $G = \Gamma_{\mathcal{W}}$ in \cite[\S7.4]{BridsonReidC} depends on two properties of $\Gamma_{\mathcal{O}}^1$: first, that $[\Gamma_{\mathcal{O}}^1, \Gamma_{\mathcal{O}}^1] = \Gamma_{\mathcal{W}}$, and second that $H_2(\Gamma_{\mathcal{O}}^1, \Z) = 0$. However, when they subsequently wished to apply their Theorem~C to produce Grothendieck pairs in $\Gamma_{\mathcal{O}}^1 \times \Gamma_{\mathcal{O}}^1$, they were unable to do so, as although $[\Gamma_{\mathcal{O}}, \Gamma_{\mathcal{O}}] = \Gamma_{\mathcal{O}}^1$, they stated that ``unfortunately, we do not know that $H_2(\Gamma_{\mathcal{O}}, \Z) = 0$'' \cite[Remark~7.3]{BridsonReidC}. Answering this question was the initial seed for the entire present article. Unfortunately, our answer turns out to be negative: 

\begin{proposition*}[See Proposition~\ref{Prop:sad-news}]
We have $H_2(\Gamma_{\mathcal{O}}, \Z) \cong \Z / 2\Z$. More generally, if $Q \leq \Isom(\mathcal{W})$ and $\Gamma$ is an extension of $\Gamma_{\mathcal{W}}$ by $Q$, then $H_2(\Gamma, \Z) \cong H_2(Q, \Z)$.
\end{proposition*}

Thus we answer Bridson \& Reid's question negatively, closing their suggested path for applying Theorem~C to produce Grothendieck pairs in $\Gamma_{\mathcal{O}}^1 \times\Gamma_{\mathcal{O}}^1$. However, our next result is that another path is possible. Namely, we will consider either the orbifold fundamental group $\Gamma_{1}^\Theta$ resp.\ $\Gamma_2^\Theta$ of the orbifold $\Theta_1(2,2,3)$ resp.\ $\Theta_2(2,2,3)$, both groups of which are extensions of $\Gamma_{\mathcal{W}}$ by the symmetric group $S_3$. We will show that they satisfy $[\Gamma_i^\Theta, \Gamma_i^\Theta] = \Gamma_{\mathcal{O}}^1$ for $i=1, 2$ and, by the above proposition, since $H_2(S_3, \Z) = 0$, we also have $H_2(\Gamma_i^\Theta, \Z) = 0$. Theorem~C thus yields the following, successfully constructing the Grothendieck pairs Bridson \& Reid sought in $\Gamma_{\mathcal{O}}^1 \times \Gamma_{\mathcal{O}}^1$. 

\begin{corollary*}[See Corollary~\ref{Cor:Groth-pairs-inGamma1}]
There exist uncountably many non-isomorphic groups $P_\lambda$ with embeddings $P_\lambda \hookrightarrow \Gamma_{\mathcal{O}}^1 \times \Gamma_{\mathcal{O}}^1$ that induce an isomorphism of profinite completions, and infinitely many of these $P_\lambda$ are finitely generated. 
\end{corollary*}

The above proposition, and the question it resolved, point to a gap in the library of techniques for computing the second homology group of lattices in $\PSL_2(\C)$. Seeking to fill this gap, we will devote the remainder of the article to discussing such techniques. In principle, for any $n \geq 1$ the group $H_n(\Gamma, \Z)$ is computable for any closed hyperbolic $3$-orbifold group $\Gamma$. We will deduce this in \S\ref{Subsec:rewriting-and-3-manifolds} as follows: Hermiller \& Shapiro \cite{Hermiller1999b} proved that the Geometrization Theorem and the Virtual Haken Theorem combined yield that every such $\Gamma$ admits a \textit{finite complete rewriting system}. Any group admitting such a system is effectively $\FP_\infty$, i.e.\ there exists an effectively computable finitely generated free resolution of $\Z$, yielding computable integral homology groups. This approach is, however, hopelessly impractical even in simple cases. 

The approach of computing homology groups via complete rewriting systems is not, however, a dead end, and we will use this to produce a practical algorithm for computing $H_2(\Gamma, \Z)$ of certain lattices. To do this, we consider groups $\Gamma$ that are extensions of $H$ by $Q$, where $H_2(H, \Z) = 0$ and $Q$ admits a finite complete rewriting system. We identify and correct a gap in the proof of a theorem due to Hermiller \& Meier \cite[Theorem~2.2]{Hermiller1999} on rewriting systems for extensions of groups. By placing a natural filtration on the aforementioned free resolution (the \textit{Kobayashi resolution}) arising from such rewriting systems, and by analysing the differentials in detail, we show that the resulting upper right-quadrant spectral sequence has its bottom two rows entirely computable under reasonable assumptions, along with the transgression maps. The assumption that $H_2(H, \Z) = 0$ ensures that this is all the information we need, and our approach amounts to proving the following theorem: 

\begin{theorem*}[See Theorem~\ref{Thm:main-thm-h2-G}]
Let $H$ be a finitely presented group with $H_2(H, \Z) = 0$. Suppose that $1 \to H \to G \to Q \to 1$ is an extension such that 
\begin{itemize}
\item $Q$ admits a finite complete rewriting system, 
\item a section map $s \colon Q \to G$ is effectively computable, and 
\item the conjugation action of $Q$ on $H^\ab$ is effectively computable.
\end{itemize}
Then there is a (practical) algorithm for computing $H_2(G, \Z)$. 
\end{theorem*}

The advantage of this approach is that it requires no direct manipulations of the group $H$ beyond abelianization computations. In particular, in our orbifold setting, we see that for any hyperbolic rational homology $3$-sphere $M$ and any $Q \leq \Isom(M)$, the group $H_2(\pi_1^\orb(M / Q), \Z)$ can be computed. As a special case, we can algorithmically -- by an algorithm that can be carried out by hand! -- verify our above answer to the question of Bridson \& Reid that $H_2(\Gamma_{\mathcal{O}}, \Z) \cong \Z / 2\Z$, since the Weeks manifold $\mathcal{W}$ is a rational homology $3$-sphere. As a showcase of the generality of our methods, we will also apply it in the case of the Fibonacci manifolds $M_n$, which are $n$-fold cyclic covers of $S^3$ branched over the figure-eight knot, and which are hyperbolic for $n \geq 4$. The question of Galois rigidity for $\pi_1(M_n)$ and some related groups has seen some study recently \cite[\S6]{Bridson2022}, and Bridson \& Reid \cite[\S7.5]{BridsonReidC} observe that certain cyclic extension $\Delta_n$ of $\pi_1(M_n)$ have $H_2(\Delta_n, \Z) = 0$. We recover these results, and more, by our algorithmic results. In the case of orientation-preserving isometries, we prove the following result:

\begin{theorem*}[See Theorem~\ref{Thm:prime-fibonacci}]
Let $p>3$ be prime, let $M_p$ be the Fibonacci manifold. Let $Q \leq \Isom^+(M_p)$, and let $G = \pi_1^\orb(M_p / Q)$. Then $H_2(G, \Z) \cong H_2(Q, \Z)$.
\end{theorem*}

In particular, this implies that if $Q$ is a proper subgroup of $\Isom^+(M_p)$ which is not isomorphic to $C_2\times C_2$, then $H_2(G, \Z) = 0$. Finally, we prove that the isomorphism $H_2(G, \Z) \cong H_2(Q, \Z)$ holds also for $G = \pi_1^\orb(M_4 / Q)$, i.e.\ the case $n=4$. 

\

The outline of the article is as follows. By necessity of presentation, the results will be presented in a somewhat different order than the exposition given in the above introduction. The first half of the article will be focussed on developing a practical algorithm from the theory of rewriting systems for the computation of the second homology group for lattices $\Gamma < \PSL_2(\mathbb{C})$ arising as finite extensions of rational homology $3$-spheres. The second half will use this technology in specific classes of examples, along with a range of other related results. Concretely, we have the following:
\begin{itemize}
\item In \S\ref{Sec:Background}, we will present the necessary background results on profinite rigidity, Grothendieck pairs, Bridson \& Reid's Theorem~C, complete rewriting systems, the Kobayashi resolution, and complete rewriting systems for $3$-manifold groups. 
\item In \S\ref{Sec:new-tech}, we first present a practical lemma (Lemma~\ref{Lem:easy-SES}) for computing the second homology group of orbifold groups, and next we present and correct the Hermiller--Meier complete rewriting system for groups $G$ arising as extensions of groups admitting complete rewriting systems. In \S\ref{Subsec:algorithm-description} we use such rewriting systems to find an algorithm for computing terms and maps in a spectral sequence for $H_n(G, \Z)$, allowing us to compute $H_2(G, \Z)$ under reasonable assumptions. The lengthy, but straightforward, details of the correctness of this algorithm and some shortcuts therein are relegated to an Appendix.
\item In \S\ref{Sec:Weeks-manifold}, we use our new techniques on the Weeks manifold and its finite extensions by isometries, and compute the second homology group of all such extensions. We also prove their profinite rigidity. 
\item In \S\ref{Sec:Fibonacci}, we analyse various cases involving the Fibonacci manifolds $M_n$ for $n \geq 4$, and prove various results about the second homology groups of extensions of $\pi_1(M_n)$ by groups of orientation-preserving isometries $Q \leq \Isom^+(M_n)$.
\end{itemize}

\section*{Acknowledgements} 

\noindent I would like to thank both Martin R.\ Bridson (Oxford) and Alan Reid (Rice / KIAS) for many helpful discussions and  warm encouragement pertaining to this note and beyond. I also thank David Xu (KIAS) for useful discussions and the long list of typos which he will inevitably spot. Finally, I wish to thank Susan Hermiller (Nebraska-Lincoln) for discussions pertaining to Proposition~\ref{Prop:HermillerTheorem}, ultimately due to her and John Meier.

\clearpage

\section{Background}\label{Sec:Background}

\noindent Throughout this article, unless otherwise specified, all groups will be finitely generated. Furthermore, by a hyperbolic $3$-orbifold group we mean a discrete $\Gamma \leq \Isom(\mathbb{H}^3)$ acting properly discontinuously with finite volume quotient $\mathbb{H}^3/\Gamma$. We say that such a group is cocompact if $\mathbb{H}^3/\Gamma$ is compact; equivalently, the corresponding hyperbolic $3$-orbifold is closed. If $\Gamma$ is torsion-free, then we say it is a hyperbolic $3$-manifold group. We will not make any arguments requiring any in-depth knowledge of $3$-manifold theory. We refer the reader to the standard book by Maclachlan \& Reid \cite{MaclachlanReidBook} for a more thorough treatment, especially of the case of arithmetic hyperbolic $3$-manifolds.

\subsection{Profinite rigidity}\label{Subsec:profinite-rigidity}

Let $G$ be a finitely generated group, and denote by $\mathcal{F}(G)$ the set of all finite quotients of $G$. The \textit{genus} $\mathcal{C}(G)$ of $G$ is defined as the set of isomorphism classes of finitely generated, residually finite groups $G'$ such that $\mathcal{F}(G) = \mathcal{F}(G')$. The set $\mathcal{F}(G)$ forms an inverse system, with indexing given by the finite index normal subgroup $N \trianglelefteq G$ as follows: if $N_1 < N_2$, then $G/N_1 \twoheadrightarrow G / N_2$. The \textit{profinite completion} $\widehat{G}$ of $G$ is the inverse limit of this system. Somewhat surprisingly, for finitely generated groups $G_1, G_2$, we have $\mathcal{F}(G_1) = \mathcal{F}(G_2)$ if and only if $\widehat{G}_1 \cong \widehat{G}_2$ (either as topological \cite{Dixon1982} or as abstract \cite{Nikolov2007} groups).
We say that $G$ is \textit{profinitely rigid} if $\mathcal{C}(G) = \{ G \}$, i.e.\ if any finitely generated residually finite group with the same finite quotients as $G$ is necessarily isomorphic to $G$. The above implies that this is the same thing as saying that for all finitely generated residually finite $\Gamma$, we have $\widehat{\Gamma} \cong \widehat{G}$ implies $\Gamma \cong G$. This notion of profinite rigidity, where $\Gamma$ is arbitrary, is sometimes called \textit{absolute} profinite rigidity to distinguish it from the corresponding \textit{relative} notion, i.e.\ when $\Gamma$ is taken to belong to a certain class. 

As mentioned in the introduction, it is not difficult to prove that abelian groups are (absolutely) profinitely rigid, and genera in the nilpotent case, while not generally yielding profinite rigidity, can be understood with some more work. A significantly more difficult question is to find examples of \textit{full-sized} profinitely rigid groups, i.e.\ such groups containing subgroups isomorphic to $F_2$, the free group of rank two. In a landmark paper, Bridson, McReynolds, Reid, and Spitler \cite{BMRS} relatively recently produced the first such examples, arising as hyperbolic $3$-orbifold groups. The setting of hyperbolic geometry for producing profinite rigidity has seen much progress since then, and while we certainly cannot hope to survey all developments, we mention that Liu \cite{Liu2023} has shown that the genus of a hyperbolic $3$-manifold group $\Gamma$ relative to the class of hyperbolic $3$-manifold groups is finite (i.e.\ $\Gamma$ is almost relatively profinitely rigid), cf.\ also the result by Xu \cite{Xu2025} in the general $3$-manifold setting. Bridson, McReynolds, Reid, and Spitler \cite{BMRSTriangle} have also given several examples of (absolutely) profinitely rigid triangle groups. 

Many of the profinite rigidity results in the above setting can be extended to other lattices in $\PSL_2(\C)$, particularly those arising as finite extensions of lattices that are known to be profinitely rigid. In particular, one of the central results for deducing further results on profinite rigidity comes in the following form. 

\begin{theorem}[{\cite[Corollary~4.3]{Bridson2022}}]\label{Thm:genus-is-same-level}
Let $\Gamma < \Delta < \PSL_2(\C)$ be lattices such that $\Gamma$ is profinitely rigid. Then the profinite genus of $\Delta$ is a subset of the lattices in $\PSL_2(\C)$ that contain $\Gamma$ with index $[\Delta \colon \Gamma]$. 
\end{theorem}

In \S\ref{Sec:Weeks-manifold}, and particularly in \S\ref{Subsec:Profinite-rigidity-of-Weeks}, we will use this theorem to deduce profinite rigidity for all finite extensions of the Weeks manifold group arising from isometries of the corresponding manifold.

\subsection{Grothendieck Pairs and Theorem C}\label{Subsec:GrothPairsAndTheoremC}

In \cite{BridsonReidC}, the authors establish a powerful technique for constructing Grothendieck pairs inside direct products of groups whose second homology is trivial in the relatively hyperbolic setting. Specifically, they obtained the following result: 

\begin{theorem}[Bridson \& Reid \cite{BridsonReidC}]\label{ThmC-BridsonReid}
Let $\Gamma$ be a finitely presented, non-elementary, relatively hyperbolic group and suppose that for infinitely many primes $p$ there is not an element of order $p$ in $\Gamma$. Let $G$ be either a hyperbolic group in which centralizers of non-trivial elements are virtually cyclic; or a central extension of such a group. If $H_2(\Gamma, \Z)=0$ and $G$ maps onto a subgroup of finite index in $[\Gamma, \Gamma]$, then:
\begin{enumerate}
\item there exist uncountably many pairwise non-isomorphic groups $P_\lambda$ and embeddings $P_\lambda \hookrightarrow G \times G$ that induce an isomorphism of profinite completions, and
\item infinitely many of these $P_\lambda$ are finitely generated.
\end{enumerate}
\end{theorem}

As observed by Bridson \& Reid \cite[p. 9]{BridsonReidC}, the above theorem becomes particularly pleasant in the setting of $G$ and $\Gamma$ being cocompact lattices in $\PSL_2(\C)$: any such groups will have virtually cyclic centralizer, and are hyperbolic (see \cite[p.9]{BridsonReidC}). Thus we have the following:

\begin{corollary}[Bridson \& Reid \cite{BridsonReidC}]\label{Cor:ThmC-for-PSL2C}
Let $G$ and $\Gamma$ be cocompact lattices in $\PSL_2(\C)$ such that $G = [\Gamma, \Gamma]$. If $H_2(\Gamma, \Z) = 0$, then: 
\begin{enumerate}
\item there exist uncountably many pairwise non-isomorphic groups $P_\lambda$ and embeddings $P_\lambda \hookrightarrow G \times G$ that induce an isomorphism of profinite completions, and
\item infinitely many of these $P_\lambda$ are finitely generated.
\end{enumerate}
\end{corollary}

Thus, given a cocompact lattice $G < \PSL_2(\C)$, if we wish to construct infinitely many Grothendieck pairs $P_\lambda \hookrightarrow G \times G$, a natural strategy becomes as follows: pick some (finite) group $Q < \Isom(\mathbb{H}^3 / G)$ and let $Q$ act on $G$ by outer automorphisms by Mostow--Prasad rigidity, such that in the resulting extension $\Gamma$ we have both (a) $[\Gamma, \Gamma] = G$, and (b) $H_2(\Gamma, \Z) = 0$. In a few sporadic cases, this has been accomplished for certain $3$-orbifold groups $G$, e.g.\ when $G$ is the fundamental group of the Weeks manifold $\mathcal{W}$ (see \S\ref{Sec:Weeks-manifold}). The objective of our article is to provide a range of tools for investigating when the second condition can be guaranteed. We will do this in three different ways: 
\begin{enumerate}
\item We prove that for \textit{any} cocompact hyperbolic $3$-orbifold group $\Gamma$, the groups $H_n(\Gamma, \Z)$ can all be effectively computed in theory; in practice, this algorithm is completely impractical (see \S\ref{Subsec:rewriting-and-3-manifolds}). 
\item We will give a simple but surprisingly useful sufficient criterion for the vanishing of $H_2(\Gamma, \Z)$, when $\Gamma$ is an extension of $G$ by a finite group $Q$, coming only from the action of $Q$ on $G^\ab$ (see \S\ref{Subsec:Easy-criterion}).
\item We will give a practical algorithm for computing $H_2(\Gamma, \Z)$ from the extension data and a finite complete rewriting system for the finite group $Q$ (see \S\ref{Subsec:algorithm-description}). 
\end{enumerate}

The latter of these three points evidently requires a brief incursion into the theory of \textit{complete rewriting systems} for groups, towards which we now turn.

\subsection{Rewriting systems}\label{Subsec:background-rewriting}

One of the most practical ways of computing explicit free resolutions for groups, and in particular also for computing integral homology groups, is by using complete rewriting systems. We give a brief summary of the necessary background, and refer the reader to \cite{Book1993} for a thorough account. 

An \textit{alphabet} $A$ is any finite set of symbols. A \textit{rewriting system} $\cR \subseteq A^\ast \times A^\ast$ is any set of pairs of words; elements $(\ell, r) \in \cR$ are called \textit{rules}, and are sometimes written $(\ell \to r)$. A system $\cR$ induces several relations on $A^\ast$. We will write $u \xr{\cR} v$ if there exist $x, y \in A^\ast$ and some rule $(\ell, r) \in \cR$ such that $u \equiv x\ell y$ and $v \equiv xry$. We say that $u$ is \textit{irreducible} (mod $\cR$) if there is no $v \in A^\ast$ such that $u \xr{\cR} v$. The set of irreducible words mod $\cR$ is denoted $\Irr(\cR)$. Let $\cR_\ell$ denote the set of left-hand sides of rules, i.e.\ $\ell \in A^\ast$ such that there is some $r \in A^\ast$ with $(\ell, r) \in \cR$. The system $\cR$ is \textit{reduced} if for every $(\ell, r) \in \cR$, $\cR_\ell \cap A^\ast \ell A^\ast = \{ \ell \}$ and $r$ is irreducible mod $\cR$. We let $\xra{\cR}$ denote the reflexive and transitive closure of $\xr{\cR}$. We denote by $\lra{\cR}$ the symmetric, reflexive, and transitive closure of $\cR$. The relation $\lra{\cR}$ defines the least congruence on $A^\ast$ containing $\cR$, and the monoid defined by the presentation $\pres{Mon}{A}{\cR}$ is identified with $A^\ast / \lra{\cR}$. The system $\cR$ is \textit{terminating} if there is no infinite chain 
\[
u_0 \xr{\cR} u_1 \xr{\cR} \cdots \xr{\cR} u_k \xr{\cR} \cdots 
\]
of rewritings. Clearly, if $\cR$ is terminating, then every word can be rewritten to an irreducible descendant (this need not be unique, however). If $\cR$ is a finite and terminating rewriting system, and $w \in A^\ast$ is a word, then we define the \textit{disorder} $\delta_\cR(w)$ of $w$ with respect to $\cR$ as the maximum of the lengths of all possible sequences of rewritings $w \xr{\cR} w_1 \xr{\cR} \cdots \xr{\cR} w_n$ (where such a sequence is defined to have length $n$). Similarly, the \textit{stretch} $\sigma_\cR(w)$ of $w$ with respect to $\cR$ is defined as the maximum of the lengths of all the words $w_i$ appearing in any such rewriting sequence. It is easy to see that these functions are well-defined; see e.g.\ \cite[Lemma~2.3]{Hermiller1999} for a proof. Furthermore, if $w \xr{\cR} w'$, then $\delta_\cR(w') < \delta_\cR(w)$ and $\sigma_\cR(w') \leq \sigma_\cR(w)$. 

Two words $v, w \in A^\ast$ are \textit{unifiable} if there exists $z \in A^\ast$ such that $v \xra{\cR} z$ and $w \xra{\cR} z$.  The system $\cR$ is \textit{locally confluent} if for all $u, v, w \in A^\ast$, we have that $u \xr{\cR} v$ and $u \xr{\cR} w$ together imply that $v$ and $w$ are unifiable. The system $\cR$ is \textit{confluent} if, with the same notation, $u \xra{\cR} v$ and $u \xra{\cR} w$ together imply that $v$ and $w$ are unifiable. Finally, $\cR$ is \textit{complete} if it is terminating and confluent. Any terminating and locally confluent system is, by Newman's Lemma \cite{Newman1942}, also confluent, and hence complete. If $\cR$ is complete, then any word rewrites in finitely many steps to a unique irreducible descendant, called its \textit{normal form}.

If $G$ is a group, then we say that $G$ \textit{admits} the rewriting system $\cR \subseteq A^\ast \times A^\ast$ if $G \cong \pres{Mon}{A}{\cR}$. Obviously, any finite group admits a finite complete rewriting system (e.g.\ its multiplication table), and many other classes of groups have been shown to admit such systems, including surface groups \cite{Chenadec1986, Hermiller1994}, some Coxeter groups \cite{Hermiller1994}, right-angled Artin groups \cite{Hermiller1995}, and some $3$-manifold groups (see \S\ref{Subsec:rewriting-and-3-manifolds}). If $G$ admits a finite complete rewriting system, then clearly $G$ has decidable word problem; to check if a word is equal to the identity element, it suffices to rewrite it to its normal form, which is the empty word if and only if the word is equal to $1$. On the other hand, as proved by Squier \cite{Squier1987}, there exists a finitely presented group $G$ with decidable word problem but such that $G$ admits no finite complete rewriting system. The proof of this uses homological methods, on which we will expand in \S\ref{Subsec:Kobayashi-back} and which play a key role in this paper. 

\subsection{Some homological algebra}\label{Subsec:homological-algebra-back}

Let $M$ be a monoid, and let $\Z M$ denote the \textit{monoid ring} of $M$, i.e.\ the ring of finitely supported linear combinations of elements of $M$. If $Q$ is a right $\Z M$-module, then we denote the action of $m \in M$ on $q \in Q$ by $q \circ m$, where the $\circ$ notation will often be useful in the sequel. We will often use square bracket notation to group together monoid ring elements acting on an element, and e.g.\ write $q \circ [m_1 - 2m_2]$ for the action of $m_1 - 2m_2 \in \Z M$ on $q \in Q$. Throughout this article, we will only in practice deal with the homology of groups, but we will phrase it in terms of monoid homology. A monoid $M$ is said to be (right-)$\FP_n$ if the trivial right $\Z M$-module $\Z$ admits some projective resolution $P_\bullet \twoheadrightarrow \Z \to 0$ such that $P_0, \dots, P_n$ are all finitely generated. If $M$ is $\FP_n$ for all $n \geq 0$, then we say that $M$ is $\FP_\infty$. Any finitely generated monoid is $\FP_1$, and any finitely presented monoid is $\FP_2$. If $M$ is a group, then being $\FP_1$ is equivalent to being finitely generated, but this fails in general when $M$ is not a group; see \cite{Kobayashi2007}.

For a fixed left $M$-module $Q$, we define the homology of $M$ with coefficients in $Q$ by $H_n(M, Q) = \Tor_n^{\Z M}(\Z, Q)$. That is, we will resolve $\Z$ as a trivial \textit{right} $\Z M$-module.\footnote{For general monoids $M$, the left-right issue is a serious one, cf.\ e.g.\ work by Guba \cite{Guba1998}, but when $M$ is a group, left and right homology are identical.} When $M$ is a group, there is no distinction between homology defined in this manner and of that defined, say, via classifying spaces. For a primer on homological algebra, the reader is referred to \cite{Brown1982, Rotman1979}.

\subsection{The Kobayashi resolution}\label{Subsec:Kobayashi-back}

We will now show how to compute the homology of a group using complete rewriting systems. We will phrase everything in terms of monoids, as this is the natural setting of rewriting systems. Fix a monoid $M$ generated by a finite set $A$, and let $\Z M$ be the monoid ring of $M$. Let $\cR \subseteq A^\ast \times A^\ast$ be a complete and reduced rewriting system defining $M$. From a complete rewriting system $\mathcal{R}$ defining a monoid $M$, Kobayashi \cite{Kobayashi1990} constructed a free (right) $\Z M$-resolution $(P_\bullet, \partial_\bullet)$ of $\Z$ with the property that if $\mathcal{R}$ is finite, then every $P_i$ is finitely generated; thereby proving the remarkable result that any monoid admitting a finite complete rewriting system is $\FP_\infty$ (a result also obtained by Anick \cite{Anick1986} and Groves \cite{Groves1990}). We will follow the notation of Guba \& Pride \cite{Guba1996} for this resolution.  Let $\Irr(\cR)$ be the set of all $\cR$-irreducible words -- we will simply call these irreducible -- and let $\Irr^+(\cR) = \Irr(\cR) \setminus \{ 1 \}$. The normal form, i.e.\ the unique irreducible descendant, of a word $w \in A^\ast$ modulo $\cR$ will in this section be denoted by $\overline{w}$. Since $\cR$ is complete, we identify $\Irr(\cR)$ with $M$, and speak of the elements of $M$ as irreducible words. Let $\cE$ be the set of all pairs of words $(u, v)$ with $u, v \in \Irr^+(\cR)$, such that $uv$ is \textbf{not} irreducible but all of its proper prefixes are irreducible. For $n \geq 1$, let $V^{(n)}$ be the set of all sequences of the form $(v_1, v_2, \dots, v_n)$ such that $v_1 \in A$, and for all $1 \leq i < n$ we have $(v_i, v_{i+1}) \in \cE$. Let $P_n$ be the free right $\ZM$-module on the basis $V^{(n)}$. Setting $P_{-1} = \Z$ and $P_0 = \ZM$, we will occasionally consider $P_0$ as the free right $\ZM$-module with basis a formal symbol $\square$. We remark that if $\cR$ is finite then each $P_i$ is clearly finitely generated. 

As a running example, let $A_1 = \{ x, y \}$ and let $\cR_1 = \{ x^2 \to 1, y^2 \to 1, yx \to xy \}$. Then $\cR_1$ is complete, and defines the Klein four-group. We then have 
\begin{align*}
V^{(1)} &= \{ (x), (y) \}, \\
V^{(2)} &= \{ (x,x), \: (y,y), \: (y,x) \}, \\
V^{(3)} &= \{ (x,x,x), (y,y,y), (y,y,x), (y,x,x) \}, \quad \dots
\end{align*}
and in general $|V^{(n)}| = n + 1$, and thus $P_n$, in this example, is generated as a free $\Z M$-module by $n+1$ elements. 

We now define chain maps $\partial_\bullet$ such that $(P_\bullet, \partial_\bullet)$ is a free resolution. In some contexts in this article we will denote the maps by $\partial_\bullet^\cR$ to clarify which rewriting system they originate with, but throughout this section we will maintain the notation without a superscript. Since we will primarily be using the values of $\partial_n$ on $n$-tuples, we will often suppress the additional redundant parentheses, i.e.\ instead of e.g.\ $\partial_3((x, y, z))$ we will simply write $\partial_3(x,y,z)$. We construct the chain maps $\partial_\bullet$ inductively  alongside $\Z$-homomorphisms $i_\bullet$. These maps will yield a chain homotopy, i.e.\ $\partial i  + i \partial = \operatorname{id}$, proving that $(P_\bullet, \partial_\bullet)$ is indeed a resolution. We set $\partial_0 = \varepsilon$, the augmentation map, and let $i_0 \colon \Z \to \ZM$ be defined by $i_0(1) = 1$. Let $\partial_1 \colon P_1 \to \ZM$ be defined by setting $\partial_1(a) = a-1$ for $(a) \in V^{(1)}$, and extend $\ZM$-linearly. We define $i_1 \colon \ZM \to P_1$ as follows. Let $x \in \Irr(\cR)$, and write $x = x_1 x_2 \cdots x_k$ with $x_i \in A$. Then we set $i_1(x) = \sum_{i=1}^k (x_i) \circ x_{i+1} \cdots x_k$ where the notation $\circ$ is used to indicate the right action of $\ZM$ on $P_1$ (and, more generally, $P_i$). In particular, we set $i_1(1) = 0$. For the reader familiar with free differential calculus (see e.g.\ \cite[Chapter V]{MKS}), we may note that $i_1(w)$ is essentially a sum over all Fox derivatives of $w$. 

We begin the inductive step of the remaining $i_\bullet$ and $\partial_\bullet$. For $n \geq 1$, we set
\begin{equation}\label{Eq:partial_n-definition}
\partial_{n+1}(v_1, \dots, v_n, v_{n+1}) = (v_1, \dots, v_n) \circ v_{n+1} - i_n \left( \partial_n (v_1, \dots, v_n) \circ v_{n+1} \right).
\end{equation}
To define $i_{n+1}$, notice that since $P_n$ is free as a $\Z$-module on the set of elements $(v_1, \dots, v_n) \circ x$, where $(v_1, \dots, v_n) \in V^{(n)}$ and $x \in \Irr(\cR)$, we only need to define $i_{n+1}$ on this set. Fix such an element. If $v_nx$ is irreducible, then we set $i_{n+1}( (v_1, \dots, v_n) \circ x) = 0$. Otherwise, there is some minimal prefix $v_n v_{n+1}$ of $v_n x$ which is reducible, where $v_{n+1}$ is non-empty; write $x = v_{n+1} y$. Then $(v_n, v_{n+1}) \in \cE$, and by \cite[p.~266]{Kobayashi1990} the element $i_n \left( \partial_n( v_1, \dots, v_n) \circ v_{n+1} \right)$ can be written as a sum of basis elements of the form $(w_1, \dots, w_n) \circ z$, for all of which $w_1 \cdots w_n z$ is shorter than $v_1 \cdots v_n v_{n+1}$. In particular, $i_{n+1}$ is defined on such basis elements by induction (on word length), and hence also on $i_n \left( \partial_n ( (v_1, \dots, v_n) \circ v_{n+1} )  \circ y \right)$. We then define 
\begin{equation}\label{Eq:i-n-definition}
i_{n+1}( (v_1, \dots, v_n) \circ x) = (v_1, \dots, v_n, v_{n+1}) \circ y + i_{n+1} \left( i_n \left( \partial_{n}(v_1, \dots, v_n) \circ v_{n+1} \right) \circ y \right)
\end{equation}
which completes the inductive definition of $i_\bullet$ and $\partial_\bullet$. It can be shown that $i_\bullet$ is a chain homotopy, and hence the Kobayashi resolution is acyclic. 

\

We continue our example of the Klein four-group from earlier. We have e.g.\ that 
\begin{align*}
\partial_2( y, x) &= (y) \circ x - i_1 \partial_1 \left( (y) \circ x \right) 
= (y) \circ x - i_1 \left( (y-1)x \right) \\ 
&= (y) \circ x - i_1 \left( \overline{yx} - x \right) = (y) \circ x - i_1 \left( xy - x \right)  \\ 
&= (y) \circ x - (x)\circ y - (y) \circ 1 + (x) = (y)\circ [x - 1] - (x) \circ [y - 1] \in P_1,
\end{align*}
since $\overline{yx} \equiv xy$, i.e.\ the normal form of $yx$ is $xy$. Note that upon tensoring $(P_\bullet, \partial_\bullet)$ with $\Z$ (as a \textit{left} $\Z M$-module), we thus see that $\widetilde{\partial}_2(y, x) = 0$, where $\widetilde{\partial}_2 = \partial_2 \otimes_{\Z M} \operatorname{id}_\Z$. In particular, the element $(y,x)$ is a generator of $H_2(M, \Z)$, where $M$ is our Klein four-group. One can easily compute $\partial_2(x,x) = (x) \circ [x+1]$ and $\partial_2(y,y) = (y) \circ [y+1]$, and hence that $\widetilde{\partial}_2(x,x) = 2(x)$ and $\widetilde{\partial}_2(y,y) = 2(y)$. Thus neither of them contribute to the kernel of $\widetilde{\partial}_2$, and in particular $H_2(M, \Z)$ is generated only by $(y,x)$. We can also compute that
\begin{align*}
\partial_3(y,y,x) &= (y,y) \circ [x-1] - (y,x) \circ [y + 1], \quad \text{and} \\
\partial_3(y,x,x) &= (x,x)\circ [y -1] + (y,x) \circ [x + 1]
\end{align*}
and hence $-\widetilde{\partial}_3(y,y,x) = \widetilde{\partial}_3(y,x,x) = 2(y,x)$. Thus $(y,x)$ has order at most $2$ in $H_2(M, \Z)$. One can now either compute the images of the remaining two basis elements of $V^{(3)}$ under $\partial_3$, or use Lemma~\ref{Lem:Kobayashi-c} to conclude that no other element in $V^{(3)}$ can map to $(y,x)$, to find that $(y,x)$ has order $2$ in $H_2(M, \Z)$. Thus we have found that $H_2(M, \Z) \cong \Z / 2\Z$ when $M$ is the Klein four-group, which may seem a modest achievement; however, the above maps are sufficiently easy to define that an inductive formula for the values of $\partial_{2n}$ and $\partial_{2n+1}$ is easy to write down for all $n$, and in particular it is easy to recover the usual formula for $H_n(M, \Z)$. This is left as an exercise.

\subsection{Properties of the Kobayashi resolution}\label{Subsec:props-of-Kobayashi}

We retain the notation from the previous subsection \S\ref{Subsec:Kobayashi-back}. If the complete rewriting system $\cR$ with which we begin is finite, then all basis elements are effectively computable, and in every dimension there are, of course, only finitely many, i.e.\ $V^{(n)}$ is finite for all $n \geq 0$. Furthermore, all chain maps are also clearly effectively (though not always \textit{efficiently}; the runtime can be exponential in the number of rules) computable. In particular, we conclude the following classical result which was, as mentioned before, independently discovered by a number of researchers.

\begin{theorem}[Anick \cite{Anick1986}, Groves \cite{Groves1990}, Kobayashi \cite{Kobayashi1990}]\label{Thm:AnickGrovesKobayashi}
Let $M$ be a monoid admitting a finite complete rewriting system. Then $M$ is $\FP_\infty$. In particular, there is an algorithm which takes as input $n \geq 0$ and outputs the isomorphism type of the finitely generated abelian group $H_n(M, \Z)$. 
\end{theorem}

One of the key properties of the Kobayashi resolution that we will make central use of in this article is the following, which can be found on \cite[p. 266]{Kobayashi1990}. To state it, we first need a particular well-founded relation $\preceq$ on $\bigcup_{n=0}^\infty P_n$, which is defined by Kobayashi as follows. For $X = (u_1, \dots, u_m) \times \{ x \} \in V^{(m)} \times \Irr(\cR)$ and $Y = (v_1, \dots, v_n) \times \{ y \}\in V^{(n)} \times \Irr(\cR)$, we write $X \succeq Y$ if $u_1 \cdots u_m x \xra{\cR} v_1 \cdots v_n y$, and $X \succ Y$ if $X \succeq Y$ but $u_1 \cdots u_m x \not\equiv v_1 \cdots v_n y$. This induces an ordering, denoted with the same symbol, on elements of the form $\nu \circ \alpha$ where $\nu \in V^{(n)}$ and $\alpha \in \Irr(\cR)$, which we thus extend to all of $\bigcup_{n=0}^\infty P_n$ as follows: given two elements, written as finite sums, $W_1 = \sum_i \mathbf{u}_i \circ (z_i \alpha_i) \in P_m$ and $W_2 = \sum_j \mathbf{v}_j \circ (z'_j \beta_j) \in P_n$, where $\nu_i \in P_m$, $\mu_j \in P_n$, $\alpha_i, \beta_j \in \Irr(\cR)$, and $z_i, z_j' \in \Z$ non-zero for all $i, j$, we set $W_1 \succ W_2$ if for every term $\nu_i \circ z_i \alpha_i$ in $W_1$, there exists some term $\mu_j \circ z_j' \beta_j$ in $W_2$ such that $\nu_i \circ \alpha_i \succ \mu_j \circ \beta_j$. The key proposition can now be stated in the following form:

\begin{lemma}[{Kobayashi \cite[p. 266]{Kobayashi1990}}]\label{Lem:Kobayashi-c}
For all $n \geq 0$, all $\mathbf{v} = (v_1, \dots, v_n) \in V^{(n)}$, and all $x \in \Irr(\cR)$, we have:
\begin{enumerate}
\item $\partial_n( \mathbf{v} \circ x) \preceq \mathbf{v} \circ x$ and $i_{n+1} (\mathbf{v} \circ x) \preceq \mathbf{v} \circ x$. 
\item If $\mathbf{v}_n x \in \Irr(\cR)$, then $i_n \partial_n (\mathbf{v} \circ x) = \mathbf{v} \circ x$. 
\item If $\mathbf{v}_n x \not\in \Irr(\cR)$, then $i_n \partial_n (\mathbf{v} \circ x) \prec \mathbf{v} \circ x$. 
\end{enumerate}
\end{lemma}

We will in \S\ref{Subsec:rewriting-for-complete-extensions} put a filtration on $V^{(1)}, V^{(2)}$, and $V^{(3)}$ for a certain type of rewriting system, which will extend to a filtration on $P_1, P_2$, and $P_3$, and Lemma~\ref{Lem:Kobayashi-c} will be instrumental in ensuring that this filtration is preserved by the chain map $\partial$. In this way, we will obtain a computable spectral sequence for the homology of certain group extensions, being the core part of our computation of $H_2(\Gamma, \Z)$ for lattices $\Gamma < \Isom(\mathbb{H}^3)$. Before passing to this, however, we first make a remark about how rewriting systems allow us to \textit{in principle} compute the homology groups of any $3$-orbifold group. 

\subsection{Rewriting systems for $3$-manifold groups}\label{Subsec:rewriting-and-3-manifolds}

Let $M$ be a closed $3$-manifold bearing one of Thurston's eight geometries (for more details, see \cite{Morgan2014}). Suppose further that if $M$ is hyperbolic, then $M$ virtually fibres over $S^1$. In 1999, Hermiller \& Shapiro \cite{Hermiller1999b} proved that under these assumptions, the group $\pi_1(M)$ admits a finite complete rewriting system. Today, we know by work of Agol \cite{Agol2013} and Wise \cite{Wise2021} on the Virtual Haken Theorem that all hyperbolic $3$-manifolds virtually fibre. In particular, the aforementioned theorem can now be stated in the following form.

\begin{theorem}[{Hermiller \& Shapiro \cite{Hermiller1999b}}\footnote{The reader is advised to note that the proof by Hermiller \& Shapiro of this theorem relies on the fact that an extension of a group admitting a finite complete rewriting system by another such group also admits a finite complete rewriting system, claimed in \cite{Hermiller1999}. As we shall discuss in \S\ref{Subsec:rewriting-for-complete-extensions}, the proof in \cite{Hermiller1999} is incomplete, although Hermiller has provided us with a fix, which we present in the same section. Furthermore, the proof of Theorem~\ref{Thm:Hermiller1999b} only relies on the claim in the case of a split extension, which \textit{is} covered by the proof in \cite{Hermiller1999}.}]\label{Thm:Hermiller1999b}
Let $M$ be a closed geometric $3$-manifold. Then $\pi_1(M)$ admits a finite complete rewriting system. 
\end{theorem}

Thus, by Theorem~\ref{Thm:AnickGrovesKobayashi} we have the following immediate corollary as a special case.

\begin{corollary}\label{Cor:algo-exists-technically}
Let $\Gamma$ be a cocompact hyperbolic $3$-orbifold group. Then there is an algorithm for computing $H_n(\Gamma, \Z)$ for every $n \geq 1$. 
\end{corollary}
\begin{proof}
By Selberg's Lemma, we have that there exists some torsion-free finite index subgroup $\Gamma_0 \leq \Gamma$. By taking the normal core, we may without loss of generality assume $\Gamma_0$ is normal in $\Gamma$. Thus fixing such a $\Gamma_0$, we have that $\mathbb{H}^3 / \Gamma_0$ is a closed hyperbolic $3$-manifold, and hence $\Gamma_0$ admits a finite complete rewriting system by Theorem~\ref{Thm:Hermiller1999b}. Thus, since $\Gamma$ is a finite extension of $\Gamma_0$, it follows easily, see e.g.\ \cite[p.~285]{Groves1993}, that $\Gamma$ also admits a finite complete rewriting system. Thus $\Gamma$ is of type $\FP_\infty$, and indeed an explicit free $\Z\Gamma$-resolution of the trivial module $\Z$, which is finitely generated in each dimension, can be constructed effectively by Theorem~\ref{Thm:AnickGrovesKobayashi}.
\end{proof}

The group $\Gamma_{\mathcal{O}}$ mentioned in the introduction of this article is an example of a hyperbolic $3$-orbifold group. Thus, Bridson \& Reid's question of whether $H_2(\Gamma_{\mathcal{O}}, \Z)=0$ or not can in principle be resolved algorithmically. However, we must temporarily shelve our enthusiasm. Indeed, notice that in the proof of Corollary~\ref{Cor:algo-exists-technically} we have relied on two particularly deep theorems: the Geometrization Theorem and the Virtual Haken Theorem. Furthermore, Hermiller \& Shapiro construct their rewriting systems based on having an explicit witness to the fact that the $3$-manifold (say, $M$) virtually fibres, including explicitly constructing a finite index fibred subgroup $G' \leq \pi_1(M)$ and an extension
\[
1 \to \pi_1(\Sigma_g) \to G' \to \Z \to 1
\]
with the associated conjugacy action of $\Z$ on the surface group $\pi_1(\Sigma_g)$.

We are thus led to ask how difficult the algorithm of Corollary~\ref{Cor:algo-exists-technically} is to implement in practice. That is, if we are given a finite extension of some $3$-manifold group ``out of the gutter'', i.e.\ with no associated information beyond its presentation, and are asked to compute (say) a finite complete rewriting system for its fundamental group, along with its (say) seventh homology group, how difficult is this? As far as the author is aware, this is prohibitively impractical. Indeed, even for the Weeks manifold $\mathcal{W}$ (defined by the group presentation \eqref{Eq:weeks-presentation} in \S\ref{Sec:Weeks-manifold}), the author is not aware of any explicit finite complete rewriting system for $\pi_1(\mathcal{W})$, even though this group is only given by two generators and two relators. There is a known index $24$ subgroup of $\pi_1(\mathcal{W})$, described in \cite[\S9.2]{BMRS} (see also \cite{Button2005}), corresponding to a genus $2$ surface bundle covering $\mathcal{W}$. This, at least in principle, gives a computationally feasible way of finding a finite complete rewriting system for $\pi_1(\mathcal{W})$, and thus also for any of its finite extensions, but the resulting system would likely become unmanageably large for homological computations. 

Nevertheless, the impracticality of applying the algorithm in Corollary~\ref{Cor:algo-exists-technically} can be bypassed in some cases. In the next section (\S\ref{Sec:new-tech}), one of the main results will be a practical algorithm for computing the \textit{second} homology group of any $3$-orbifold group $\Gamma$ sitting in an extension
\[
1 \to \pi_1(M) \longrightarrow \Gamma \longrightarrow Q \to 1
\]
for a finite group $Q \leq \Isom(M)$ and a \textit{rational homology sphere} $\pi_1(M)$. This rather modest-sounding theorem (when compared to Corollary~\ref{Cor:algo-exists-technically}) has the decisive advantage of being both practical and easy to carry out, even by hand.
 
\section{Tools for computing the second homology group}\label{Sec:new-tech}

\noindent In this section, we will develop a toolbox for computing $H_2(\Gamma,\Z)$ for certain lattices $\Gamma < \Isom(\mathbb{H}^3)$. We will first in \S\ref{Subsec:Easy-criterion} give a simple criterion, which turns out to be surprisingly powerful, for ensuring that $H_2(\Gamma,\Z) = 0$. Next, we turn to a more general tool, and in \S\ref{Subsec:rewriting-for-complete-extensions} we will describe a (corrected) construction due to Hermiller \& Meier of a complete rewriting system for an extension $G$ of a group $H$ defined by a complete rewriting system by another such group $Q$. In \S\ref{Subsec:algorithm-description}, we will show how the Kobayashi resolution for this complete rewriting system, which shadows the Lyndon--Hochschild-Serre resolution for an extension, can be used to compute the second homology of $G$, assuming $H_2(H, \Z) = 0$, the group $H$ is finitely presented, and $Q$ is defined by a finite complete rewriting system. 

\subsection{A criterion for vanishing second homology}\label{Subsec:Easy-criterion}

The following basic lemma provides a surprisingly powerful tool to ensure that $H_2(\Gamma,\Z) = 0$ for certain lattices $\Gamma < \Isom(\mathbb{H}^3)$. Indeed, as we shall see in \S\ref{Sec:Weeks-manifold}, in the case of the Weeks manifold $\mathcal{W}$, the lemma will immediately imply that $H_2(\Gamma, \Z) = 0$ for all lattices $\Gamma$ with $\pi_1(\mathcal{W}) \leq \Gamma < \PSL_2(\mathbb{C})$ with two exceptions, including the maximal case, i.e.\ when $\Gamma$ is the normalizer of $\pi_1(\mathcal{W})$ in $\PSL_2(\C)$. 

\begin{lemma}\label{Lem:easy-SES}
Let $\Gamma = \pi^\orb_1(X)$ where $X = M / Q$ for $M$ a hyperbolic rational homology $3$-sphere, and $Q \leq \Isom(M)$. If $\gcd(|Q|, |\pi_1(M)^{\ab}|) = 1$, then $H_2(\Gamma, \Z) \cong H_2(Q, \Z)$.
\end{lemma}
\begin{proof}
Consider the homological Lyndon--Hochschild-Serre spectral sequence $E^2_{p,q}$ for the extension 
\[
1 \to \pi_1(M) \to \Gamma \to Q \to 1.
\]
Since $M$ is a rational homology $3$-sphere, we have $H_2(\pi_1(M), \Z) = 0$ by Poincaré duality, and hence $E^2_{p,2}=0$ for all $p \geq 0$. Furthermore, note that since $M$ is a hyperbolic $3$-manifold we have $Q \leq \Isom(M)$ finite, and since $M$ is a rational homology $3$-sphere we have that $\pi_1(M)^\ab$ is finite. Thus, $\gcd(|Q|, |\pi_1(M)^\ab|)$ is well-defined, and since it is equal to $1$ by assumption, it thus follows by using the transfer (see \cite[III.10.4]{Brown1982}) that $H_p(Q,\pi_1(M)^\ab)=0$ for $p>0$, i.e.\ $E^{2}_{p,1}=0$ for $p>0$. Thus the only possible non-zero term on the diagonal $p+q=2$ is $E^2_{2,0}=H_2(Q,\Z)$, and since $E^\infty_{2,0}=E^3_{2,0} = \ker(E^2_{2,0} \to E^{2}_{0,1})$ we immediately deduce that there exists an exact sequence
\begin{equation}\label{Eq:easy-SES-2}
0 \to H_2(\Gamma, \Z) \to H_2(Q, \Z) \xrightarrow{\tau} \pi_1(M)^{\ab}_{Q},
\end{equation}
where $\tau$ is the transgression. But $H_2(Q, \Z)$ is annihilated by $|Q|$ by using the transfer, and hence its order must also be co-prime with $\pi_1(M)^\ab$, and hence also with the quotient group $\pi_1(M)^\ab_Q$. Thus $\Hom(H_2(Q, \Z), \pi_1(M)^\ab_Q) = 0$, so $\tau=0$. Thus \eqref{Eq:easy-SES-2} gives $H_2(\Gamma, \Z) \cong H_2(Q, \Z)$, as desired. 
\end{proof}

More generally, the exact sequence \eqref{Eq:easy-SES-2} also arises whenever $H_1(Q, H_1(H, \Z))=0$ and $M$ is a rational homology $3$-sphere. When coprimality of the orders is not available, however, we need a more general tool; we now provide such a tool. 

\subsection{Complete rewriting systems for group extensions}\label{Subsec:rewriting-for-complete-extensions}

Throughout this section, we will let 
\begin{equation}\label{Eq:our-extension-for-Kobayashi}
1 \to H \to G \xrightarrow{\pi} Q \to 1
\end{equation}
be an extension of groups with $H$ a finitely generated group, and $Q$ a group admitting a finite complete rewriting system. We say that $G$ is a \textit{complete extension} of $H$ (by $Q$). We will make no assumption of finite presentability of $H$, but we will assume without loss of generality that $\cR_H$ is some complete, reduced rewriting system and that $A$ is a finite \textit{monoid} generating set for $H$, i.e.\ that $H \cong \pres{Mon}{A}{\cR_H}$. We similarly assume $Q$ is given by a complete, reduced monoid presentation $\pres{Mon}{B}{\cR_Q}$. 

The corresponding surjective morphisms will be denoted $\pi_H \colon A \to H$ resp.\ $\pi_Q \colon B \to Q$. We fix some map $s \colon \cR_Q \to  \Irr(\cR_H)$, calling it the \textit{defect map}, such that $s(u, v) \in \Irr(\cR_H)$ is a $\cR_H$-irreducible word representing $uv^{-1}$ in the extension $G$. For example, the extension \eqref{Eq:our-extension-for-Kobayashi} being split is equivalent to saying that $s(u,v) = 1$ in $H$ (and therefore also $G$) for all rules $(u,v) \in \cR_Q$. We fix a map
\[
\varphi \colon B \times A \to \Irr(\cR_H)
\]
such that $\varphi(b, a)$ represents the element $\overline{\pi_Q(b)} \pi_H(a)  \overline{\pi_Q(b)}^{-1}$ in $H$ (in other words, $\varphi$ encodes the conjugation action of $Q$ on $H$). To simplify notation, we will often write $\varphi_b(a)$ for this map, and thus we have $\overline{b}a = \varphi_b(a)\overline{b}$ in $G$ for all $a \in A, b \in B$. The map $\varphi$ extends naturally to a map $\varphi^\ast \colon  B^\ast \times A^\ast \to A^\ast$.

We will present a proof due to Hermiller \& Meier \cite{Hermiller1999}\footnote{In fact, there is a gap in the proof of \cite[Theorem~2.2]{Hermiller1999}. Indeed, the proof therein only covers the case of a split extension, and does not take the extension class into account. Susan Hermiller has, in private communication with the author, provided the below corrected proof.} for constructing a complete rewriting system $\mathcal{R}_G$ on the alphabet $A \cup B$ for $G$. The rules of the system will be the union of the following sets:
\begin{enumerate}
\item all rules from $\mathcal{R}_H$,
\item all rules of the form $(ba \to \varphi_b(a)b)$ for all $a \in A, b \in B$. 
\item all rules of the form $(u \to s(u,v)v)$ whenever $(u \to v) \in \cR_Q$. 
\end{enumerate}
Note that there are generally infinitely many rules of the first type (since $\cR_H$ is not assumed to be finite), but if $\cR_Q$ is finite and $H$ is finitely generated by the set $A$, then there are only finitely many of the second and third type. We will use this system $\cR_G$ to compute transgression maps in a spectral sequence for the homology of $G$, and thus also $H_2(G, \Z)$ in some special cases. Before doing this, however, we present the aforementioned proof of completeness of $\cR_G$. 

First, we mention two useful facts about words modulo $\cR_G$. The first is that since for every pair $(b,a) \in B \times A$ there exists a rule $ba \to \varphi_b(a)b$ in $\cR_G$, and consequently that $ba$ is reducible, we see that every irreducible word modulo $\cR_G$ is of the form $hw$ for some $h \in A^\ast$ and $w \in B^\ast$, where $h$ is irreducible modulo $\cR_H$ and $w$ is irreducible modulo $\cR_Q$. Second, we note that if $h \in A^\ast$ and $w \in B^\ast$, then for all $q \in B^\ast$ we have
\begin{equation}\label{Eq:rewrite-respects-projection}
q \longrightarrow^\ast_{\cR_G} hw \quad \iff \quad q \longrightarrow^\ast_{\cR_Q} w.
\end{equation}
Indeed, this follows directly by induction on the number of rewriting steps, and the fact that the rules of type (2) in $\cR_G$ do not alter the projection of a word onto $B^\ast$. Thus, $\cR_Q$ can be used to compute the projection map $\pi \colon G \to Q$. 

The proof of the following proposition is similar to the proof of \cite[Theorem~2.2]{Hermiller1999}, but contains a fix of a gap therein; the fix was kindly provided to the author by Hermiller. 

\begin{proposition}\label{Prop:HermillerTheorem}
The rewriting system $\cR_G$ is complete and reduced. 
\end{proposition}
\begin{proof}
By the above observation of the set of irreducible words of $\cR_G$, it follows that there is a bijection between the set of irreducible words modulo $\cR_G$ and the elements of $G$. In particular, we only need to prove that $\cR_G$ is terminating in order to show that it is complete; indeed, the fact that it is reduced follows by assumption of the fact that $\cR_H$ and $\cR_Q$ are both assumed to be reduced, and that all words $s(u,v)$ are chosen to be irreducible. 

Thus, let us prove that $\cR_G$ is terminating. Recall the definition of \textit{stretch} and \textit{disorder} from \S\ref{Subsec:background-rewriting}. Let $w \in (A \cup B)^\ast$, and let $w'$ be the projection of the word to $B^\ast$, i.e.\ the word obtained by removing all letters from $A$. Let $n = \sigma_{\cR_Q}(w')$ be the stretch of $w'$ with respect to $\cR_Q$. We can then write, factoring from right to left,
\[
w \equiv k_1 q_1'' k_2 \cdots q_n'' k_{n+1}, \quad k_i \in A^\ast, \: q_i'' \in B \cup \{ 1\},
\]
such that if $q_i'' \equiv 1$, then the preceding prefix $k_{1} q_{1}'' \cdots q_{i-1}'' k_{i}$ is empty. We then define the following weight functions $\psi_j \colon (A \cup B)^\ast \to \mathbb{N}$ by:
\begin{equation*}
\begin{cases}
\psi_0(w) := \sigma_{\cR_Q}(w') & \\
\psi_{2i}(w) := \delta_{\cR_H}(k_i), & 
\end{cases}
\quad \text{and} \quad
\begin{cases}
\psi_1(w) := \delta_{\cR_Q}(w') & \\
\psi_{2i+1}(w) := |k_i|, &
\end{cases}
\end{equation*}
for all $1 \leq i \leq n$. By definition of the stretch of $w'$, these functions can be defined also for any descendant of $w$ modulo $\cR_G$, with the same $n$. For any $v \in (A \cup B)^\ast$ with $w \xr{\cR_G} v$, we will now show that
\begin{equation}\label{Eq:decreasing}
(\psi_{2n}(w), \psi_{2n-1}(w), \dots, \psi_{0}(w)) \succ (\psi_{2n}(v), \psi_{2n-1}(v), \dots, \psi_{0}(v)) 
\end{equation}
where $\succeq$ is the \textit{reverse} lexicographical ordering on $\mathbb{N}^{2n+1}$. This will be sufficient to establish the termination of $\cR_G$. Thus suppose $v$ is such that $w \xr{\cR_G} v$. First, if the rule witnessing this rewriting is of type (1), i.e.\ from $\cR_H \subseteq \cR_G$, then it must rewrite some $k_j$ to $k_j' \in A^\ast$, where $1 \leq j \leq n+1$. Then $\delta_{\cR_H}(k_j) > \delta_{\cR_H}(k_j')$, and hence $\psi_{2j}(w) > \psi_{2j}(v)$. Since the succeeding $\psi_i$, i.e.\ those with $i > 2j$, are unaffected by this rewriting, we thus see that \eqref{Eq:decreasing} holds. On the other hand, if the rule witnessing the rewriting $w \xr{\cR_G} v$ is of type (3), i.e.\ of the form $q_1 \to s(q_1, q_2) q_2$ where $(q_1 \to q_2) \in \cR_Q$, then let $q_j''$ be the left-most of the $q_i''$ in $w$ that is affected by the rewriting, and $q_\ell''$ the rightmost. Then necessarily $k_{i}'' \equiv 1$ for all $j+1 \leq i \leq \ell$, and $q_1 \equiv q_i'' q_{i+1}'' \cdots q_{\ell}''$. Thus $\psi_{0}(v) \leq \psi_0(w)$. Furthermore, we have that $\delta_{\cR_Q}(v) < \delta_{\cR_Q}(w)$, and hence $\psi_1(v) < \psi_1(w)$. Thus in this case we also have \eqref{Eq:decreasing}, since the ordering is lexicographical. Finally, if the rule witnessing the rewriting $w \xr{\cR_G} v$ is of type (2), i.e.\ of the form $qh \to \varphi_q(h)q$, with $q \in B$ and $h \in A$, then suppose that the letters affected are a prefix of $q_i'' k_{i+1}$, so that $qh \equiv q_i'' h$, where $k_{i+1} \equiv h k_{i+1}''$ for some $k_{i+1}'' \in A^\ast$, i.e.\ $k_{i+1}''$ is just $k_{i+1}$ with the first letter removed. Thus $\delta_{\cR_H}(k_{i+1}) \geq \delta_{\cR_H}(k_{i+1}'')$ and $|k_{i+1}| > |k_{i+1}''|$. However, factoring $v$, it is easy to see that $\psi_j(w) \equiv \psi_j(v)$ for all $j < 2i$, and the argument just made shows that $\psi_{2i}(w) \geq \psi_{2i}(v)$ and $\psi_{2i+1}(w) > \psi_{2i+1}(v)$. Thus \eqref{Eq:decreasing} holds also in this case. In particular, in all cases, we have \eqref{Eq:decreasing}, and since there is no infinite descending chain in the reverse lexicographical ordering on $\mathbb{N}^{2n+1}$, it follows that $\cR_G$ is terminating, as required. 
\end{proof}

If $\cR_H$ and $\cR_Q$ are both finite systems, then since there are only finitely many rules of type (2), we see that the entire system $\cR_G$ is finite, too. Thus, we have the following result, appearing already in the literature as e.g.\ \cite[Theorem~2.2]{Hermiller1999} and \cite[Proposition~3.4]{Hermiller1999b} (see also \cite[p.~285]{Groves1993}), but with incomplete proofs. 

\begin{theorem}[Hermiller \& Meier \cite{Hermiller1999}]\label{Thm:Hermiller-Meier-G-FCRS}
Let $G$ be a group that is an extension
\[
1 \longrightarrow H \longrightarrow G \longrightarrow Q \longrightarrow 1.
\]
If $H$ and $Q$ both admit finite complete rewriting systems, then so too does $G$. 
\end{theorem}

We will now deduce information about the homology of $G$ by using the complete rewriting system $\cR_G$, and in particular gain explicit control over terms in a spectral sequence for computing $H_2(G, \Z)$.

\subsection{Algorithm for the second homology group}\label{Subsec:algorithm-description}

We retain all notation from the previous sections; in particular, $G$ is a group, $H \trianglelefteq G$ and $G / H = Q$, with $\cR_Q, \cR_H$, and $\cR_G$ are the complete rewriting systems constructed in \S\ref{Subsec:rewriting-for-complete-extensions}. Let $(P_\bullet, \partial_\bullet)$ be the right Kobayashi resolution associated to the complete rewriting system $\cR_G$. We recall the notation $V^{(n)}$ as the $\ZM$-basis of the free right $\ZM$-module $P_n$ from \S\ref{Subsec:background-rewriting}. Let $V^{(n)}_Q$ resp.\ $V^{(n)}_H$ denote the bases of the $\Z Q$-free resp.\ $\Z H$-free modules $P_n^Q$ resp.\ $P_n^H$ in the Kobayashi resolution associated to the complete rewriting systems $\cR_Q$ resp.\ $\cR_H$. Notice that since the only rules of $\cR_G$ which begin with a letter from $A$ come from $\cR_H$, and since the left-hand side of every second type of rule is of the form $ba$, where $a \in A$ is a single letter, it follows that every element of $V^{(n)}$ can be written of the form 
\[
(b_1, \dots, b_p, a_{1}, \dots, a_q) \text{ with $b_i \in B$ and $a_j \in A$}
\]
where $p+q = n$, such that $(b_1, \dots, b_p) \in V^{(p)}_Q$ and $(a_1, \dots, a_q) \in V^{(q)}_H$. We now give a ``small'' filtration on the chain complex $(P_\bullet, \partial_\bullet)$ arising from this. We filter $P_1$ by letting the basis elements $(a)$ with $a \in A$ have degree $0$, and the basis elements $(b)$ with $b \in B$ have degree $1$. We denote these two sets $V^{(1)}_0$ resp.\ $V^{(1)}_1$. For $P_2$, the filtration is defined by letting the basis elements 
\[
(h_1, h_2), \quad (q_1, h_1), \quad \text{resp.} \quad (q_1, q_2)
\]
have degree $0, 1$, resp.\ $2$, where $h_1, h_2 \in A^+$ and $q_1, q_2 \in B^+$. The sets of basis elements of degree $0, 1$, resp.\ $2$ are denoted $V^{(2)}_0, V^{(2)}_1$, resp.\ $V^{(2)}_2$. Analogously, for $n= 3$ the filtration is defined on a basis element $(x_1, x_2, x_3)$ as the number of $B$-type entries in the tuple (either $0, 1, 2$, or $3$); and we use the notation $V^{(3)}_i$ for $0 \leq i \leq 3$ for the four degrees. Finally, for $n \geq 4$ all entries have degree $3$. The filtration is summarized in Table~\ref{Tab:filtration}. As part of our computations of $\partial_2$ and $\partial_3$, we will see that they preserve this filtration, and thus we obtain a spectral sequence $E^\bullet_{p,q}$ converging to the homology of $G$. This spectral sequence can be seen as a minimal ``truncation'' of the Lyndon--Hochschild--Serre spectral sequence specifically for computing $H_2(G,\Z)$. In particular, we have that $E^0_{0,1}$ resp.\ $E^0_{1,0}$ is free abelian on $\{ (a) \mid a \in A \}$ resp.\ $\{ (q) \mid q \in B \}$. Furthermore, $E^0_{0,2}, E^0_{1,1}$, resp.\ $E^0_{2,0}$ is free abelian on 
\[
\mathcal{E} \cap (A \times A^\ast), \quad \mathcal{E} \cap (B \times A), \quad \text{resp.} \quad \mathcal{E} \cap (B \times B^\ast),
\]
being the sets $V^{(2)}_0, V^{(2)}_1$, resp.\ $V^{(2)}_2$. The analogous statement holds also for the bases of $E^0_{0,3}, E^0_{1,2}, E^0_{2,1}$, and $E^0_{3,0}$. Finally, for $p+q = n \geq 4$, we have
\[
E^0_{p,q} = \begin{cases*}
P_n \otimes_{\Z G} \Z & if $p=3$, \\
0  & otherwise.
\end{cases*}
\]
These last entries, in the higher degrees, will not be relevant to our computations. 

\begin{table}[]
\begin{tabular}{l|lllll}
      & $n=0$                        & $n=1$                         & $n=2$                                & $n=3$                                     & $n \geq 4$                    \\ \hline
$p=0$ & \cellcolor[HTML]{FFFFFF}$\Z$ & \cellcolor[HTML]{FFFFFF}$(a)$ & \cellcolor[HTML]{FFFFFF}$(h_1, h_2)$ & \cellcolor[HTML]{FFFFFF}$(h_1, h_2, h_3)$ & \cellcolor[HTML]{FFFFFF} \\
$p=1$ & \cellcolor[HTML]{FFFFFF}     & \cellcolor[HTML]{FFFFFF}$(b)$ & \cellcolor[HTML]{FFFFFF}$(q_1, h_1)$ & \cellcolor[HTML]{FFFFFF}$(q_1, h_1, h_2)$ & \cellcolor[HTML]{FFFFFF}      \\
$p=2$ & \cellcolor[HTML]{FFFFFF}     & \cellcolor[HTML]{FFFFFF}      & \cellcolor[HTML]{FFFFFF}$(q_1, q_2)$ & \cellcolor[HTML]{FFFFFF}$(q_1, q_2, h_1)$ & \cellcolor[HTML]{FFFFFF}      \\
$p=3$ & \cellcolor[HTML]{FFFFFF}     & \cellcolor[HTML]{FFFFFF}      & \cellcolor[HTML]{FFFFFF}             & \cellcolor[HTML]{FFFFFF}$(q_1, q_2, q_3)$ & \cellcolor[HTML]{FFFFFF}     $P_n$
\end{tabular}
\caption{The filtration on the Kobayashi resolution $(P_\bullet, \partial_\bullet)$ of the rewriting system $\cR_G$, which we will use to compute $H_2(G, \Z)$. The row corresponds to the filtration degree.}\label{Tab:filtration}
\end{table}

\renewcommand{\thealgorithm}{for computing $H_2(G, \Z)$, \textnormal{with $1 \to H \to G \to Q \to 1$ an extension, with $H$ having $H_2(H, \Z)=0$, and given a finite, reduced, complete rewriting system $\cR_Q \subseteq B^\ast \times B^\ast$ for $Q$, a section $s \colon \cR_Q \to A^\ast$, and the action $\varphi \colon B \times A \to A^\ast$.}}
\begin{algorithm}
\caption{}\label{alg:cap}
\begin{algorithmic}[1]
\State Write down the finite set of rules of $\cR_G \setminus \cR_H$ as described in \S\ref{Subsec:rewriting-for-complete-extensions}.
\State Compute the finite sets $V^{(2)}_{1}, V^{(2)}_{2}$ and $V^{(3)}_2, V^{(3)}_3$. Let $n_1 = |V_1^{(2)}|$ and $n_2 = |V_2^{(2)}|$. Let $m_1 = |V_2^{(3)}|$, and $m_2 = |V_3^{(3)}|$. Note in particular that $n_1 = |B| \cdot |A|$. 
\State Using the inductive definition \eqref{Eq:partial_n-definition}, compute the value of $\partial_2$ on $V^{(2)}_{1}$ and $V^{(2)}_{2}$. Tensor the map with $\Z$ by forgetting coefficients, giving the map $\widetilde{\partial_2}$ taking values in $\widetilde{P}_1$, which is free abelian on $|A| + |B|$ elements. We record these values of $\widetilde{\partial_2}$ in an integer $(n_1+n_2)\times (|A| + |B|)$-matrix $M_2$. Identify the first $|A|$ columns as a basis for $H^\ab = E^1_{0,1}$, by mapping $(a) \to [a]_\ab$ for all $a \in A$. 
\State Using the inductive definition \eqref{Eq:partial_n-definition}, and the previous computation of $\partial_2$, compute the map $\partial_3 + \fg_0 P_2$ on $V^{(3)}_2$ and $V^{(3)}_3$ in the following manner: we compute $\partial_3$ inductively as usual, but delete any occurrence of terms of the form
\[
\partial_2 ( (h_1, h_2) \circ w) \quad \text{and} \quad i_2((h_1) \circ w) \quad \text{and} \quad i_2 \left( \partial_2  (h_1, h_2)  \circ w\right)
\]
for all $h_1, h_2 \in A^\ast$ and $w \in \Z M$, as these will all land in $\fg_0 P_2$. 
\State Tensor the map $\partial_3$ above with $\Z$, giving a map from landing in $\fg_1 \widetilde{P}_2 \oplus \fg_2 \widetilde{P}_2$. We record this in an integer $(m_1 + m_2) \times (n_1 + n_2)$-matrix $M_3$. Identify the first $n_1$ columns as a basis for $\Z[B] \times H^\ab = E^1_{1,1}$, by identifying $(b, a)$ with $(b, [a]_\ab)$ for all $b \in B$ and $a \in A$.
\State The row space of the top left $(n_1 \times |A|)$-minor in $M_2$ gives the image of the horizontal map $d_{1,1}^{1} \colon E^{1}_{1,1} \to E^1_{0,1} = H^\ab$. 
\State The row space of the bottom right $(n_2 \times |B|)$-minor in $M_2$ gives the image of the horizontal map $d_{2,0}^{1} \colon E^{1}_{2,0} \to E^1_{1,0}$ (indeed, it coincides with the map $\widetilde{\partial}^Q_2$).
\State The row space of the top left $(m_1 \times n_1)$-minor in $M_3$ gives the image of the horizontal map $d_{2,1}^{1} \colon E^{1}_{2,1} \to E^1_{1,1} = B \times H^\ab$. 
\State The row space of the bottom right $(m_2 \times n_2)$-minor in $M_3$ gives the image of the horizontal map $d_{3,0}^{1} \colon E^{1}_{3,0} \to E^1_{2,0}$ (indeed, it coincides with the map $\widetilde{\partial}^Q_3$).
\State Taking homology, we can compute $E^2_{0,1}$ as the quotient of $H^\ab$ by $\im(d_{1,1}^1)$, giving the co-invariants $H^\ab_Q$. We also compute the groups $E^2_{1,1}, E^2_{2,0}$. The latter two are isomorphic to $H_1(Q, H^\ab)$ resp.\ $H_2(Q, \Z)$. 
\State Choose a basis for the kernel of $d_{2,0}^1$ (which becomes a basis for $H_2(Q, \Z)$) using the above information. For each such basis element $\mathbf{v}$, there is a corresponding element in the row space of the bottom $(n_2 \times (|A| + |B|))$-minor of $M_2$, and this element, which by definition is entirely supported in the first $|A|$ columns, projects under the abelianization map to an element of $H^\ab$. That element is the image of $\mathbf{v}$ under the transgression map $d^2_{2,0} \colon H_2(Q, \Z) \to H_0(Q, H^\ab)$. Repeat this process for every basis element of the kernel of $d_{2,0}^1$. This computes the entire transgression map $d^2_{2,0}$. 
\State Choose a basis for the kernel of $d_{3,0}^1$ (which becomes a basis for $H_3(Q, \Z)$) using the above information. For each such basis element $\mathbf{v}$, there is, like in the previous step, a corresponding element in the row space of the bottom $(m_2 \times (n_1+n_2))$-minor of $M_3$, and this element is entirely supported in the first $n_1$ columns. Projecting this row-space element to $B \times H^\ab$, we get the image of (the $H_3$-equivalence class of) $\mathbf{v}$ under the transgression map $d^2_{3,0} \colon H_3(Q, \Z) \to H_1(Q, H^\ab)$. Repeat this process for every basis element, giving the entire transgression map $d^2_{3,0}$. 
\State We compute $E^\infty_{2,0} = E^3_{2,0} = \ker d^2_{2,0}$ and $E^\infty_{1,1} = E^3_{1,1} = E^2_{1,1} / \im(d^2_{3,0})$. The spectral sequence has now converged for computing $H_2(G, \Z)$. Since $H_2(H, \Z) = 0$, the diagonal $p+q=2$ has three terms $0, E^\infty_{1,1}$, and $E^\infty_{2,0}$. Thus we have an exact sequence
\begin{align}\label{Eq:SES-algorithm}
0 \to E^\infty_{1,1} \to H_2(G, \Z) \to E^\infty_{2,0} \to 0.
\end{align}
\State Using a standard diagram chasing argument, the extension class of \eqref{Eq:SES-algorithm} can easily be determined (if one is only checking whether $H_2(G, \Z) = 0$ or not, this step is not needed). The information is entirely contained in the computation of $\widetilde{\partial}_2$ and $\widetilde{\partial}_3$. 
\end{algorithmic}
\end{algorithm}

We now use this filtration to compute several sets, which yield a spectral sequence that computes $H_2(G, \Z)$. The algorithm is described in the algorithm environment below. A proof of the correctness of this algorithm is provided in the Appendix to this article. This proof consists, primarily, of verifying that the filtration on $P_\bullet$ defined above is preserved by $\partial$, that ignoring all terms supported in degree $0$ is valid in Step~4, and that there are canonical identifications 
\[
E^2_{p,q} \cong H_p(Q, H_q(H, \Z)) \quad \text{when } (p,q) \in \{ (0,1), (1,0), (1,1), (2,0), (3,0) \}
\]
as in the Lyndon--Hochschild-Serre spectral sequence. The algorithm computes some data in our spectral sequence, and in particular we compute the homomorphisms
\[
d^1_{1,1} \colon E^1_{1,1} \to E^1_{0,1}, \quad d^1_{2,1} \colon E^1_{2,1} \to E^1_{1,1}, \quad d^1_{2,0} \colon E^1_{2,0} \to E^1_{1,0}, \quad \text{and} \quad d^1_{3,0} \colon E^1_{3,0} \to E^1_{2,0}
\]
in the first page, and the two transgressions 
\[
d^2_{2,0} \colon E^2_{2,0} \to E^2_{0,1}, \quad \text{and} \quad d^2_{3,0} \colon E^2_{3,0} \to E^2_{1,1}
\]
in the second page. This data is all contained in two matrices $M_2$ and $M_3$, which are naturally subdivided into minors, and these minors provide the homomorphisms as follows, with notation as in the algorithm:
\[
{ \renewcommand{\arraystretch}{1.8}
\setlength{\arraycolsep}{10pt}
\raisebox{-1.5ex}{$M_2 =$}
\begin{blockarray}{ccc}
  & \scriptstyle |A| & \scriptstyle |B| \\
  \begin{block}{c(c|c)}
    \scriptstyle n_1 & d^1_{1,1} & \mathbf{0} \\
    \cline{2-3}
    \scriptstyle n_2 & d^2_{2,0} & d^1_{2,0} \\
  \end{block}
\end{blockarray} \quad 
\raisebox{-1.5ex}{$M_3 =$}
\begin{blockarray}{ccc}
  & \scriptstyle n_1 & \scriptstyle n_2 \\
  \begin{block}{c(c|c)}
    \scriptstyle m_1 & d^1_{2,1} & \mathbf{0} \\
    \cline{2-3}
    \scriptstyle m_2 & d^2_{3,0} & d^1_{3,0} \\
  \end{block}
\end{blockarray}
}
\]
These two matrices carry all of the data for computing $H_2(G, \Z)$ in the case that $H_2(H, \Z)=0$. The two matrices are, by the above algorithm, computable if $\cR_Q$ is finite, the section map $s \colon \cR_Q \to A^\ast$ is effectively computable (which it is if there is an effectively computable section $Q \to G$), and if the conjugation action $\varphi$ of $Q$ on $A^\ast$ is effectively computable. Thus, if $H_2(H, \Z)=0$, then by Step~14 of the algorithm the isomorphism type of $H_2(G, \Z)$ can be determined. In other words, we have the main result of this section:

\begin{theorem}\label{Thm:main-thm-h2-G}
Let $H$ be a finitely presented group with $H_2(H, \Z) = 0$. Suppose that 
\[
1 \to H \to G \to Q \to 1
\]
is an extension such that 
\begin{itemize}
\item $Q$ admits a finite complete rewriting system, 
\item a section map $s \colon Q \to G$ is effectively computable, and 
\item the conjugation action of $Q$ on $H^\ab$ is effectively computable.
\end{itemize}
Then there is a (practical) algorithm for computing $H_2(G, \Z)$. 
\end{theorem}

If $M$ is a hyperbolic rational homology $3$-sphere, then by the Mostow--Prasad Rigidity Theorem it follows that $\Isom(M)$ is a finite group acting by outer automorphisms on $\pi_1(M)$. Since $M$ is a rational homology $3$-sphere, we have $H_2(\pi_1(M), \Z) = 0$ by Poincaré duality, and $\pi_1(M)$ is finitely presented. Every finite group admits a finite complete rewriting system (e.g.\ its multiplication table), and the section map and conjugation action are effectively computable from the finite presentation of the extension. Thus Theorem~\ref{Thm:main-thm-h2-G} applies to every extension of $\pi_1(M)$ by a subgroup $Q \leq \Isom(M)$.

\begin{corollary}
Let $M$ be a closed hyperbolic rational homology $3$-sphere. Then for every $Q \leq \Isom(M)$ the group $H_2(\pi^\orb_1(M / Q), \Z)$ is effectively computable. 
\end{corollary}

As an example, we are able to effectively compute the second homology group of the fundamental group of every orbifold quotient of the Weeks manifold $\mathcal{W}$. This is the goal of the next section, and it will lead us to answering the two questions of Bridson \& Reid from the introduction: what is $H_2(\Gamma_{\mathcal{O}}, \Z)$, and can we find infinitely many Grothendieck pairs $P_\lambda \hookrightarrow \Gamma_{\mathcal{O}}^1 \times \Gamma_{\mathcal{O}}^1$?

\begin{remark}
Throughout this section, we have dealt with \textit{trivial} coefficients, i.e.\ computing $H_2(G, \Z)$ taking $\Z$ as a $G$-module with trivial action. But this is not the entire story: owing to the explicit nature of the computation of our chain maps $\partial_2, \partial_3$, we can work in much more generality, and as long as we start with a $G$-module $M$ for which the action is computable, then little else would be required to compute $H_2(G, M)$, or at least deciding whether or not $H_2(G, M)$ is trivial or not, as long as we have $H_2(H, M) = 0$. However, stating the precise conditions for this computability would be outside the scope of this paper. 
\end{remark}

\section{The Weeks manifold and its orbifolds}\label{Sec:Weeks-manifold}

\noindent In this section, we will prove several results about extensions of the fundamental group of the Weeks manifold $\mathcal{W}$, and answer the two questions of Bridson \& Reid from the introduction. By an appropriate choice of extension, we will also construct uncountably many pairwise non-isomorphic groups $P_\lambda$ and embeddings $P_\lambda \hookrightarrow \Gamma_{\mathcal{O}}^1 \times \Gamma_{\mathcal{O}}^1$ that induce an isomorphism of profinite completions, where $\Gamma_{\mathcal{O}}^1$ is the group of units in a maximal order in the quaternion algebra of the Weeks manifold (see below for definitions). As a byproduct of our investigations, we will also show that every lattice $\Gamma < \PSL_2(\mathbb{C})$ with $\pi_1(\mathcal{W}) \leq \Gamma \leq \Gamma_\mathcal{O}$ is (absolutely) profinitely rigid, where $\Gamma_{\mathcal{O}}$ is the normalizer in $\PSL_2(\C)$ of $\pi_1(\mathcal{W})$, 

\subsection{Background}\label{Subsec:Weeks-background}

We briefly recall the necessary results about the Weeks manifold $\mathcal{W}$. Significantly more information can be found in the book by Maclachlan \& Reid \cite[\S4.8.3]{MaclachlanReidBook}. Recall that the set of volumes of closed orientable hyperbolic $3$-manifolds is a well-ordered subset of $\mathbb{R}$ by a result due to Thurston \& Jørgensen \cite{Thurston}, and that furthermore there are only finitely many such manifolds realizing any given volume. For the smallest volume $v_0$ (approximately $0.9427...$), it turns out there is a unique such manifold $\mathcal{W}$ with $\mathrm{Vol}(\mathcal{W}) = v_0$, the \textit{Weeks manifold}. It was constructed by Weeks \cite{Weeks1985} and independently by Matveev \& Fomenko \cite{Matveev1988}, and is obtained by performing $(-5,1)$-surgery on one component of the Whitehead link, and $(5,2)$-surgery on the other. That $\mathcal{W}$ is indeed the unique manifold of volume $v_0$ was proved in 2009 by Gabai, Meyerhoff \& Milley \cite{Gabai2009}. The fundamental group is given by
\begin{align}\label{Eq:weeks-presentation}
\Gamma_\mathcal{W} = \pi_1(\mathcal{W}) = \pres{}{a,b}{a^2b^2a^2 = ba^{-1}b, \: b^2a^2b^2 = ab^{-1}a}
\end{align}
In particular, we have $H_1(\Gamma_\mathcal{W}, \Z) = \Z_5 \oplus \Z_5$. We furthermore have that $\mathcal{W}$ is an arithmetic manifold, with invariant trace field $\Q[\theta]$ where $\theta^3-\theta^2+1=0$, see also \cite{Chinburg2001}. Its associated quaternion algebra will be denoted $B_\mathcal{W}$, and choosing a maximal order $\mathcal{O}$ in $B_\mathcal{W}$ we let the group of units of this order be denoted $\Gamma_\mathcal{O}^1$. By \cite{Mednykh1998}, the orbifold $\mathbb{H}^3 / \Gamma_{\mathcal{O}}^1$ can be obtained by performing $(3,0)$-Dehn surgery on the knot $5_2$. The normalizer of $\Gamma_{\mathcal{O}}^1$ in $\PSL_2(\mathbb{C})$ is denoted $\Gamma_\mathcal{O}$, and arises as the extension of $\Gamma_\mathcal{W}$ by its full isometry group, which we now describe.

It is well-known (see e.g.\ \cite{Mednykh1998, Hodgson1994}; it can also be verified in \texttt{SnapPy}) that $\Isom(\mathcal{W}) \cong D_6$, the dihedral group on six points. We shall as in \cite{Mednykh1998} let it be defined by the Coxeter presentation
\[
\Isom(\mathcal{W}) = \pres{}{r, s, t}{r^2 = s^2 = t^2 = (rs)^2=(rt)^2= (st)^3 = 1}.
\]
By the Mostow--Prasad rigidity theorem, the isometries $r, s, t$ correspond to outer automorphisms $R, S, T$, respectively, acting on the generators of $\Gamma_\mathcal{W}$, in this case by
\begin{align}\label{Eq:action-on-weeks}
\begin{cases*}
RaR^{-1} = a^{-1} & \\
RbR^{-1} = b^{-1} & 
\end{cases*} \quad \begin{cases*}
SaS^{-1} = b & \\
SbS^{-1} = a & 
\end{cases*} \quad \begin{cases*} 
TaT^{-1} = a & \\
TbT^{-1} = a^{-1}b^{-1} & 
\end{cases*}
\end{align}
and in particular that 
\[
R^2 = S^2 = (RS)^2 = (RT)^2 = (ST)^3 = 1, \quad T^2 = a^{-1}
\]
see \cite[p. 52]{Mednykh1998}. Let $Q \leq \Isom(\mathcal{W})$, and let $\mathcal{W}/Q$ be the corresponding orbifold. Then we say that $\mathcal{W}/Q$ is a \textit{Weeks orbifold}, and its orbifold fundamental group will be called a \textit{Weeks orbifold group}. By rigidity, $Q_1, Q_2 \leq \Isom(\mathcal{W})$ are conjugate if and only if the corresponding Weeks orbifolds are isometric if and only if the corresponding Weeks orbifold groups are isomorphic. Thus the lattice of conjugacy classes of subgroups of $D_6$ is the same as the lattice of Weeks orbifold groups; we have drawn this out in Figure~\ref{Fig:Weeks-lattice}.

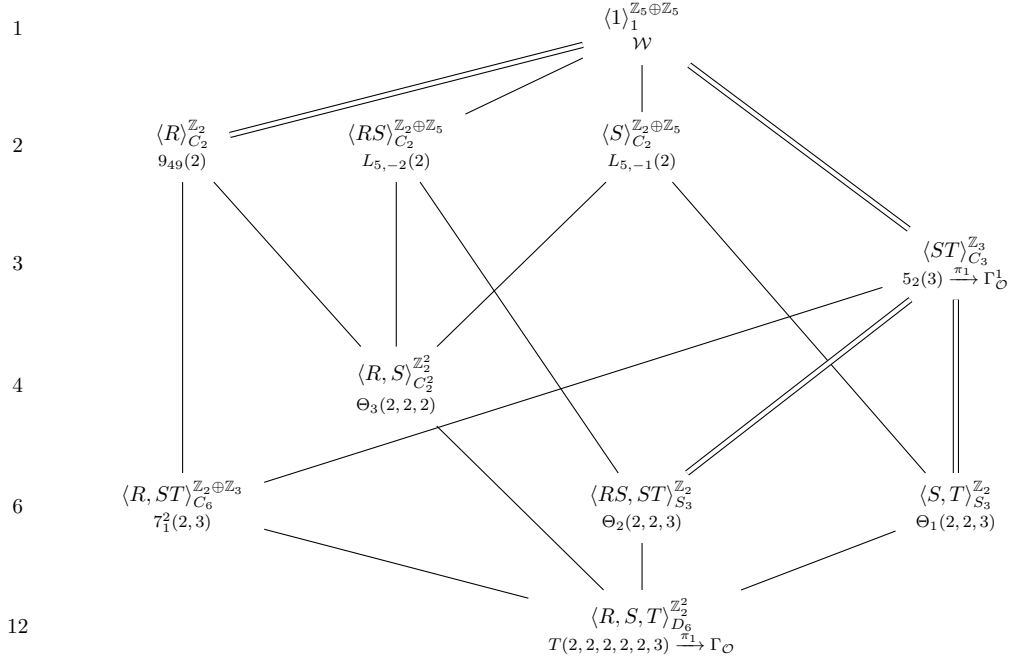
\begin{figure}
\[\begin{tikzcd}[scale cd=0.80]
	1 &&& \begin{array}{c} {\langle 1 \rangle}^{\mathbb{Z}_5\oplus\mathbb{Z}_5}_{1} \\ {\footnotesize \text{$\mathcal{W}$}} \end{array} \\
	2 & \begin{array}{c} \langle R \rangle_{C_2}^{\mathbb{Z}_2}\\ {\footnotesize \text{$9_{49}(2)$}} \end{array} & \begin{array}{c} \langle RS \rangle_{C_2}^{\mathbb{Z}_2\oplus\mathbb{Z}_5}\\ {\footnotesize \text{$L_{5,-2}(2)$}} \end{array} & \begin{array}{c} \langle S \rangle_{C_2}^{\mathbb{Z}_2\oplus\mathbb{Z}_5} \\ {\footnotesize \text{$L_{5,-1}(2)$}} \end{array} \\
	3 &&&&& \begin{array}{c} \langle ST \rangle_{C_3}^{\mathbb{Z}_3} \\ {\footnotesize \text{$5_2(3) \xrightarrow{\pi_1} \Gamma_\mathcal{O}^1$}} \end{array} \\
	4 && \begin{array}{c} \langle R,S \rangle_{C_2^2}^{\mathbb{Z}_2^2}\\ {\footnotesize \text{$\Theta_3(2,2,2)$}} \end{array} \\
	6 & \begin{array}{c} \langle R,ST \rangle_{C_6}^{\mathbb{Z}_2\oplus\mathbb{Z}_3} \\ {\footnotesize \text{$7_1^2(2,3)$}} \end{array} && \begin{array}{c} \langle RS,ST \rangle_{S_3}^{\mathbb{Z}_2} \\ {\footnotesize \text{$\Theta_2(2,2,3)$}} \end{array} && \begin{array}{c} \langle S,T \rangle_{S_3}^{\mathbb{Z}_2} \\ {\footnotesize \text{$\Theta_1(2,2,3)$}} \end{array} \\
	12 &&&  \begin{array}{c} {\langle R,S,T \rangle_{D_6}^{\mathbb{Z}_2^2}} \\ {\footnotesize \text{$T(2,2,2,2,2,3)\xrightarrow{\pi_1}\Gamma_\mathcal{O}$}} \end{array}
	\arrow[equals, from=1-4, to=2-2]
	\arrow[no head, from=1-4, to=2-3]
	\arrow[no head, from=1-4, to=2-4]
	\arrow[equals, from=1-4, to=3-6]
	\arrow[no head, from=2-2, to=4-3]
	\arrow[no head, from=2-2, to=5-2]
	\arrow[no head, from=2-3, to=4-3]
	\arrow[no head, from=2-3, to=5-4]
	\arrow[no head, from=2-4, to=4-3]
	\arrow[no head, from=2-4, to=5-6]
	\arrow[no head, from=3-6, to=5-2]
	\arrow[equals, from=3-6, to=5-4]
	\arrow[equals, from=3-6, to=5-6]
	\arrow[no head, from=4-3, to=6-4]
	\arrow[no head, from=5-2, to=6-4]
	\arrow[no head, from=5-4, to=6-4]
	\arrow[no head, from=5-6, to=6-4]
\end{tikzcd}\]
\caption{The subgroup lattice of the full group $\Isom(\mathcal{W}) \cong D_6$ of isometries of the Weeks manifold, cf.\ \cite[Fig.~2]{Mednykh1998}. The leftmost column is the order of the subgroup. The subscripts indicate the isomorphism class of the subgroup, and the superscripts indicate the first homology of the corresponding orbifold quotient. The name of each orbifold quotient from \cite{Mednykh1998} is listed beneath the corresponding subgroup. Double lines indicate that the induced embedding of fundamental groups is via the derived subgroup; thus if we e.g.\ set $\Theta=\pi_1 (\Theta_2(2,2,3))$, then we can read that $[\Theta, \Theta]=\pi_1(5_2(3)) := \Gamma_\mathcal{O}^1$, see Proposition~\ref{Prop:Derived-subgroups-of-Thetas}.}
\label{Fig:Weeks-lattice}
\end{figure}

\subsection{Schur multipliers of Weeks orbifolds}

In \cite{BridsonReidC}, Bridson \& Reid proved that in the chain $\Gamma_\mathcal{W} <_3 \Gamma^1_\mathcal{O} <_4 \Gamma_{\mathcal{O}}$ we have $[\Gamma_{\mathcal{O}}, \Gamma_{\mathcal{O}}] = \Gamma^1_{\mathcal{O}}$ and $[\Gamma_{\mathcal{O}}^1, \Gamma_{\mathcal{O}}^1] = \Gamma_\mathcal{W}$. Furthermore, it is known that $\Gamma_{\mathcal{O}}$ is the normalizer in $\PSL_2(\mathbb{C})$ of $\Gamma_\mathcal{O}^1$. They also prove $H_2(\Gamma_{\mathcal{O}}^1, \Z) = 0$, and hence can apply their Theorem~\ref{ThmC-BridsonReid} to deduce the existence of Grothendieck pairs in $\Gamma_\mathcal{W}\times \Gamma_\mathcal{W}$. However, they are unable to apply their Theorem~\ref{ThmC-BridsonReid} to obtain Grothendieck pairs in $\Gamma_{\mathcal{O}}^1 \times \Gamma_{\mathcal{O}}^1$, since they ``do not know that $H_2(\Gamma_\mathcal{O}, \Z) =0$''. We first prove that, unfortunately, this cannot be arranged; namely, the second homology group of $\Gamma_{\mathcal{O}}$ is cyclic of order $2$. 

\begin{proposition}\label{Prop:sad-news}
With notation as above, we have $H_2(\Gamma_\mathcal{O},\Z) \cong \Z_2$.
\end{proposition}
\begin{proof}
Let $Q = \Isom(\mathcal{W})$. Then the group $\Gamma_\mathcal{O}$ is the extension $1 \to \Gamma_\mathcal{W} \rightarrow \Gamma_\mathcal{O} \rightarrow Q \to 1$ with action as in \S\ref{Subsec:Weeks-background}. However, $H_1(\Gamma_\mathcal{W}, \Z) = \Z_5 \oplus \Z_5$, with the two generators being images of $a$ and $b$ from \eqref{Eq:weeks-presentation}. Since $\gcd(12,25)=1$, we can now apply Lemma~\ref{Lem:easy-SES}, which yields the exact sequence 
\begin{equation}\label{Eq:weeks-maximal-h2-exact-sequence}
0 \longrightarrow H_2(\Gamma_{\mathcal{O}}, \Z) \longrightarrow \Z_2 \xlongrightarrow{\tau} \pi_1(\mathcal{W})^\ab_Q
\end{equation}
as $H_2(Q,\Z) = \Z_2$ since $Q$ is the dihedral group of order $12$. Now the transgression $\tau$ is necessarily $0$, being a map from a $2$-group to a $5$-group; alternatively, using \eqref{Eq:action-on-weeks} it is not difficult to see that $\pi_1(\mathcal{W})^\ab_Q = 0$ (see also \S\ref{Subsec:computations-for-Weeks}). Thus \eqref{Eq:weeks-maximal-h2-exact-sequence} yields the isomorphism $H_2(\Gamma_{\mathcal{O}}, \Z) \cong \Z_2$. 
\end{proof}

In \S\ref{Subsec:computations-for-Weeks}, we will show how our algorithm from \S\ref{Subsec:algorithm-description} can be used to mechanically deduce the same result. 

Before this, we turn towards Bridson \& Reid's goal of applying Theorem~\ref{ThmC-BridsonReid} to produce infinitely many Grothendieck pairs in $\Gamma_{\mathcal{O}}^1 \times \Gamma_{\mathcal{O}}^1$. As we shall prove, this can still be realized even though $H_2(\Gamma_{\mathcal{O}}, \Z)\neq 0$. Indeed, we must simply use another extension than $\Gamma_{\mathcal{O}}$. To do this, we begin by showing that all other, except one, lattices between $\Gamma_\mathcal{W}$ and its normalizer have trivial second homology. We use the names of the lattices as in Figure~\ref{Fig:Weeks-lattice}, and these names are due to Mednykh \& Vesnin \cite{Mednykh1998}.

\begin{proposition}\label{Prop:all-smaller-Weeks-orbifolds-trivial-H2}
Let $\Delta$ be any lattice with $\Gamma_\mathcal{W} \leq \Delta \leq \Gamma_{\mathcal{O}}$. Then 
\[
H_2(\Delta, \Z) = 
\begin{cases*}
\Z / 2 \Z, & $\Delta \in \{ \Gamma_{\mathcal{O}}, \Theta_3(2,2,2) \}$ \\ 
0, & otherwise.
\end{cases*}
\]
\end{proposition}
\begin{proof}
Note that for every subgroup $Q \leq D_6 \cong \Isom(\mathcal{W})$ of the dihedral group of order $12$, we have $H_2(Q, \Z) = 0$ when $Q$ is one of the groups $C_2, C_3, C_6$, and $S_3$; and $H_2(Q, \Z) = \Z / 2\Z$ if $Q = D_6$ or $Q = C_2 \times C_2$ (see Figure~\ref{Fig:Weeks-lattice} for the lattice of subgroup). Furthermore, all such subgroups have order coprime with $25 = |\pi_1(\mathcal{W})^\ab|$. Thus by Lemma~\ref{Lem:easy-SES}, the claim follows. 
\end{proof}

We shall now use this result as described above.

\subsection{Infinitely many Grothendieck pairs in $\Gamma_{\mathcal{O}}^1 \times \Gamma_{\mathcal{O}}^1$}\label{Subsec:Groth-pairs-Weeks}

As we mentioned above, Bridson \& Reid proved that while the derived subgroup of $\Gamma_\mathcal{O}$ is $\Gamma_\mathcal{O}^1$, we discovered by Proposition~\ref{Prop:sad-news} that their Theorem~\ref{ThmC-BridsonReid} cannot be applied to construct Grothendieck pairs in $\Gamma_\mathcal{O}^1 \times \Gamma_\mathcal{O}^1$, since the Schur multiplier is non-trivial. However, we will now use other subgroups in the subgroup lattice between $\Gamma_\mathcal{W}$ and its normalizer in $\PSL_2(\C)$ to successfully apply Theorem~\ref{ThmC-BridsonReid}.

Let $\Gamma_1^\Theta = \pi_1(\Theta_1(2,2,3))$ and let $\Gamma_2^\Theta = \pi_1(\Theta_2(2,2,3))$ as in Figure~\ref{Fig:Weeks-lattice}. These groups correspond, respectively, to the extensions of $\Gamma_\mathcal{W}$ by the two non-conjugate embeddings of the symmetric group $S_3$ in $D_6$, generated by  $\langle S, T \rangle$ resp.\ $\langle RS, ST \rangle$. We now prove that either of these groups can be used to construct the desired Grothendieck pairs. We first prove that their derived subgroups are as desired. 

\begin{proposition}\label{Prop:Derived-subgroups-of-Thetas}
We have $[\Gamma_1^\Theta,\Gamma_1^\Theta] = [\Gamma_2^\Theta,\Gamma_2^\Theta] = \Gamma_\mathcal{O}^1$. 
\end{proposition}
\begin{proof}
Note that (as in Figure~\ref{Fig:Weeks-lattice}) the group $\Theta_1$ is the extension of $\Gamma_\mathcal{W}$ by the two isometries $S$ and $T$ acting by conjugation on the fundamental group, and by \eqref{Eq:action-on-weeks} we have $S^2 = (ST)^3 = 1$ and $T^2 = a^{-1}$. Since $\Gamma_\mathcal{W}^\ab \cong \Z_5 \oplus \Z_5$, any homomorphism $\varrho \colon \Theta_1 \to \Z_2$ necessarily kills $\Gamma_\mathcal{W}$. Hence since $(ST)^3 = 1$, the only possible non-trivial such homomorphism $\varrho$ is the map sending $S$ and $T$ to the generator of $\Z_2$. Hence $\Theta_1$ has only one subgroup of index $2$, and since $[\Theta_1 : \Gamma_\mathcal{O}^1] = 2$ and $\Theta_1^\ab \cong \Z_2$ (as is easily verified from the relations \eqref{Eq:action-on-weeks}), we thus conclude the statement of the proposition for $\Theta_1$. The claim for $\Theta_2$ is proved entirely analogously.
\end{proof}

Note that we can also conclude by Proposition~\ref{Prop:all-smaller-Weeks-orbifolds-trivial-H2} that
\[
H_2(\Gamma_1^\Theta, \Z) = H_2(\Gamma_2^\Theta, \Z) = 0.
\]
These facts ultimately stem from the fact that $H_2(S_3, \Z) = 0$. We are now in a position to apply Theorem~\ref{ThmC-BridsonReid}, and thus we can conclude: 

\begin{corollary}\label{Cor:Groth-pairs-inGamma1}
Let $\Gamma_\mathcal{O}^1$ be the group of units in a maximal order of the quaternion algebra $B_\mathcal{W}$ of the Weeks manifold. Then there exist uncountably many non-isomorphic groups $P_\lambda$ with embeddings $P_\lambda \hookrightarrow \Gamma_{\mathcal{O}}^1 \times \Gamma_{\mathcal{O}}^1$ that induce an isomorphism of profinite completions, and infinitely many of these $P_\lambda$ are finitely generated. 
\end{corollary}

Thus the question of Bridson \& Reid from the introduction is answered affirmatively. 

\subsection{Profinite rigidity}\label{Subsec:Profinite-rigidity-of-Weeks}

In \cite[\S9.4, Theorem 9.5]{BMRS}, it is proved that both $\Gamma_\mathcal{O}^1$ and $\Gamma_{\mathcal{O}}$ are profinitely rigid. In this section, we will extend their analysis to show that for all lattices $\pi_1(\mathcal{W}) \leq \Delta \leq \Gamma_{\mathcal{O}}$, we have that $\Delta$ is profinitely rigid. We note that there are $10$ such lattices in total (including $\pi_1(\mathcal{W}), \Gamma_\mathcal{O}^1$, and $\Gamma_{\mathcal{O}}$), and thus there are seven remaining cases to consider. 

\begin{theorem}\label{Thm:Weeks-are-profinitely-rigid}
Every Weeks orbifold group is (absolutely) profinitely rigid. 
\end{theorem}
\begin{proof}
We deal with orbifold groups with notation as in Figure~\ref{Fig:Weeks-lattice} in turn. By Theorem~\ref{Thm:genus-is-same-level}, it suffices to distinguish the groups at each specific index from one another by their finite quotients. Thus, it immediately follows that the lattices $\Gamma_\mathcal{O}^1$, $\pi_1^\orb(\Theta_3(2,2,2))$, and $\Gamma_\mathcal{O}$, in which $\pi_1(\mathcal{W})$ has index $3, 4$, resp.\ $12$ are profinitely rigid, a fact already noted in \cite{BridsonReidC}. We thus turn to the groups in which $\Gamma_\mathcal{W}$ has index $2$ and $6$, respectively.

At index $2$, there are three groups, being the orbifold fundamental groups of
\[
9_{49}(2), \quad L_{5,-2}(2), \quad \text{resp.} \quad L_{5,-1}(2).
\]
The abelianization of these three groups are, respectively, $\Z_2, \Z_2 \oplus \Z_5$, and $\Z_2 \oplus \Z_5$, and hence $\pi_1^\orb(9_{49}(2))$ is profinitely rigid. On the other hand, in the notation of \eqref{Eq:action-on-weeks}, if we let $x$ denote the isometry $RS$, then $xa x^{-1} = b^{-1}$ and $xbx^{-1} = a^{-1}$, so we have the presentations
\begin{align*}
\pi_1^\orb(L_{5,-2}(2)) &= \pres{}{\Gamma_\mathcal{W}, x}{xax^{-1}b = xbx^{-1}a = x^2 = 1}, \\
\pi_1^\orb(L_{5,-1}(2)) &= \pres{}{\Gamma_\mathcal{W}, s}{sas^{-1}b^{-1} = sbs^{-1}a^{-1} = s^2 = 1}.
\end{align*}
From these presentations, it can be verified with e.g.\ \texttt{GAP} that $\pi_1^\orb(L_{5,-2}(2))$ surjects $\PSL_2(7)$. Indeed, the map 
\begin{equation}
a \mapsto \begin{pmatrix} 1 & 0 \\ 1 & 1\end{pmatrix}, \quad b \mapsto \begin{pmatrix} 1 & 4 \\ 0 & 1\end{pmatrix}, \quad x \mapsto \begin{pmatrix} 0 & 2 \\ 3 & 0\end{pmatrix}
\end{equation}
can easily be seen to define a surjective homomorphism $\pi_1^\orb(L_{5,-2}(2)) \to \PSL_2(7)$. However, a direct check shows that $\pi_1^\orb(L_{5,-1}(2))$ does not surject $\PSL_2(7)$.\footnote{Indeed, an exhaustive search of all groups of order $\leq 168 = |\PSL_2(7)|$ shows that $\PSL_2(7)$ is the smallest finite quotient that distinguishes the two groups $\pi_1^\orb(L_{5,-2}(2))$ and $\pi_1^\orb(L_{5,-1}(2))$.} Hence $\pi_1^\orb(L_{5,-2}(2))$ and $\pi_1^\orb(L_{5,-1}(2))$ have non-isomorphic profinite completions and hence are profinitely rigid. 

At index $6$, there are also three groups, being the orbifold fundamental groups of 
\[
7_{1}^2(2,3), \quad \Theta_1(2,2,3), \quad \text{resp.} \quad \Theta_2(2,2,3).
\]
Recall that we denote by $\Gamma_1^\Theta$ resp.\ $\Gamma_2^\Theta$ the orbifold fundamental groups of the last two of these orbifolds. Now, the abelianization of $\pi_1^\orb(7_{1}^2(2,3))$,being $\Z_2 \oplus \Z_3$, shows that it is profinitely rigid, since $\Gamma_1^\Theta$ and $\Gamma_2^\Theta$ both have abelianization $\Z_2$. On the other hand, we may again use $\PSL_2(7)$ to distinguish these two groups; indeed, using the presentations of the groups from \eqref{Eq:action-on-weeks}, we can check that for primes $p < 100$ we have that $\Gamma_1^\Theta = \langle \Gamma_\mathcal{W}, rs, st \rangle$ surjects $\PSL_2(p)$ only when $p = 2, 11, 17, 23, 43, 59, 83$, whereas $\Gamma_2^\Theta$ surjects only when $p = 2, 7, 37, 43, 59, 67, 97$. Thus these two groups are also profinitely rigid, and we are done. 
\end{proof}

Thus, again recalling the names of the orbifold quotients of the Weeks manifold due to Mednykh \& Vesnin \cite{Mednykh1998}, we can restate the above theorem in the following form. 
 
\begin{corollary}\label{Cor:Weeks-orbifold-groups}
Let $\Gamma$ be the orbifold fundamental group of any of the hyperbolic $3$-orbifolds $9_{49}(2)$, $L_{5,-2}(2)$, $L_{5,-1}(2)$, $5_2(3)$, $\Theta_3(2,2,2)$, $7_1^2(2,3)$, $\Theta_1(2,2,3)$, $\Theta_2(2,2,3)$, and $T(2,2,2,2,2,3)$. Then $\Gamma$ is absolutely profinitely rigid. 
\end{corollary}

We now turn towards verifying the fact that $H_2(\Gamma_{\mathcal{O}}, \Z) = \Z / 2\Z$ in a purely mechanical way, namely by using our algorithm from \S\ref{Subsec:algorithm-description} with the quotient $Q \cong D_6$. 

\subsection{Additional calculations}\label{Subsec:computations-for-Weeks}

Let $\Gamma_{\mathcal{W}} = \pi_1(\mathcal{W})$ be the group with the presentation \eqref{Eq:weeks-presentation}. We will let $H = \pres{Mon}{\Sigma \cup \Sigma^{-1}}{\cR_H}$ be a monoid presentation for $\Gamma_{\mathcal{W}}$ on the generating set $\Sigma \cup \Sigma^{-1}$, where $\Sigma = \{ a, b\}$ is the standard group generating set for $\Gamma_{\mathcal{W}}$, subject to the relations \eqref{Eq:weeks-presentation}. We will denote the elements of $\Sigma^{-1}$ as $\{ A, B \}$, so that $aA = Aa = 1$ in $H$. We will assume that $\cR_H$ is a complete and reduced rewriting system for $H$, such that the word $ABab$ is irreducible modulo $\cR_H$, and such that the words $ba$ and $AB$ are irreducible modulo $\cR_H$. This is possible to arrange for simultaneously, since projecting to the abelianization shows that $[a,b] \neq u$ for any word $u$ with $|u| \leq 3$. Furthermore, we have that $ab \neq AB$, as is easily seen by projecting to the abelianization, and hence these words $a, b, A, B, AB, ABab$ can all be assumed to be irreducible modulo $\cR_H$. Beyond this, the rewriting system $\cR_H$ for the fundamental group of the Weeks manifold will be taken entirely as a black box, and no further information is needed from it. 

The group $Q = \Isom(\mathcal{W})$, isomorphic to the dihedral group with $12$ elements, is generated by the isometries $\{ r, s, t \}$, as in \S\ref{Subsec:Weeks-background}. It is also generated by the two elements $rt, s$, and if we call these $x, y$, then a monoid presentation for $Q$ is 
\[
Q = \pres{Mon}{x,y}{x^2 = y^2 = (xy)^6 = 1} \cong D_6.
\]
The action of $x$ and $y$ on the generators $a, b$ (and consequently also $A, B$) can be read off from \eqref{Eq:action-on-weeks}. In particular, $x$ acts by mapping $(a,b) \to (a^{-1}, ab)$, while $y$ acts by mapping $(a,b) \to (b, a)$. A complete and reduced rewriting system $\cR_Q$ for $Q$ on the generating set $\{ x, y \}$ is given by 
\[
x^2 \to 1, \quad y^2 \to 1, \quad (yx)^3 \to (xy)^3.
\]
To compute the section map $s \colon \cR_Q \to \Irr(\cR_H)$, we simply note that $s(x^2 \to 1) = s(y^2 \to 1) = 1$, since (see \eqref{Eq:action-on-weeks}) we have $(RT)^2 = S^2 = 1$ when acting by inner automorphisms. On the other hand, we have 
\begin{align*}
(S \cdot RT)^3 (RT \cdot S)^{-3}(a) &= (a^{-1}b^{-1}ab)a(b^{-1}a^{-1}ba) \\
(S \cdot RT)^3 (RT \cdot S)^{-3}(b) &= a^{-1}b^{-1}aba^{-1}ba = (a^{-1}b^{-1}ab)b(b^{-1}a^{-1}ba)
\end{align*}
and hence $(s \cdot rt)^3 (rt \cdot s)^{-3}$ acts via conjugation by $a^{-1}b^{-1}ab$. Thus $s((yx)^3 \to (xy)^3) = ABab$. This word is irreducible, since $\Gamma_{\mathcal{W}}$ is not commutative, and both its generators are non-trivial (as can be seen by passing to the abelianization). 

Thus, assembling this information, we can by Theorem~\ref{Thm:Hermiller-Meier-G-FCRS} construct a complete and reduced rewriting system $\cR_G$ for the extension 
\[
1 \to \pi_1(\mathcal{W}) \to \Gamma_{\mathcal{O}} \to \Isom(\mathcal{W}) \to 1
\]
where $H = \pi_1(\mathcal{W})$, and $Q = \Isom(\mathcal{W})$, and $G = \Gamma_{\mathcal{O}}$. The system $\cR_G$ has the alphabet $\{ a, b, A, B, x, y \}$, with all rules from $\cR_H$ together with 
\begin{align*}
&xa \to Ax, \quad xb \to abx, \quad xA \to ax, \quad xB \to BAx, \\
&ya \to by, \quad yb \to ay, \quad yA \to By, \quad yB \to Ay \\
&x^2 \to 1, \quad y^2 \to 1, \quad (yx)^3 \to ABab(xy)^3.
\end{align*}

We now compute the data that we need to find $H_2(\Gamma_{\mathcal{O}}, \Z)$. For  ease of notation we will often deliberately conflate equivalence classes and representatives for them.

\begin{table}[h]
\begin{tabular}{|r|cccc|cc|}\hline
                & $(a)$ & $(A)$ & $(b)$ & $(B)$ & $(x)$ & $(y)$ \\ \hline
$(x,a)$         & $1$   & $-1$  &       &       &       &       \\
$(x,b)$         & $-1$  &       &       &       &       &       \\
$(x,A)$         & $-1$  & $1$   &       &       &       &       \\
$(x,B)$         &       & $-1$  &       &       &       &       \\
$(y,a)$         & $1$   &       & $-1$  &       &       &       \\
$(y,b)$         & $-1$  &       & $1$   &       &       &       \\
$(y,A)$         &       & $1$   &       & $-1$  &       &       \\
$(y,B)$         &       & $-1$  &       & $1$   &       &       \\
\hline
$(x,x)$         &       &       &       &       & $2$   &       \\
$(y,y)$         &       &       &       &       &       & $2$   \\
$(y,x(yx)^2)$ & $-1$  & $-1$  & $-1$  & $-1$  &       &      \\ \hline
\end{tabular}

\caption{The matrix $M_2$ computing $\widetilde{\partial_2}$ for the Weeks manifold. The first eight rows are computations in filtration degree $1$ (i.e.\ they are the basis of $E^0_{1,1}$), and the last three rows are computations in filtration degree $2$ (i.e.\ they are the basis of $E^0_{2,0}$). }
\label{Tab:M2}
\end{table}
\begin{table}[h]
\footnotesize
\setlength{\tabcolsep}{3pt}
\begin{tabular}{|r|cccccccc|ccc|}\hline
                 & $(x,a)$ & $(x,b)$ & $(x,A)$ & $(x,B)$ & $(y,a)$ & $(y,b)$ & $(y,A)$ & $(y,B)$ & $(x,x)$ & $(y,y)$ & $(y,x(yx)^2)$ \\ \hline
$(x,x,a)$        &    $-1$ &         &    $-1$ &         &         &         &         &         &         &         &                 \\
$(x,x,b)$        &    $-1$ &   $-2$  &         &         &         &         &         &         &         &         &                 \\
$(x,x,A)$        &   $-1$  &         &   $-1$  &         &         &         &         &         &         &         &                 \\
$(x,x,B)$        &         &         &   $-1$  &  $-2$   &         &         &         &         &         &         &                 \\
$(y,y,a)$        &         &         &         &         &  $-1$   &  $-1$   &         &         &         &         &                 \\
$(y,y,b)$        &         &         &         &         &  $-1$   &  $-1$   &         &         &         &         &                 \\
$(y,y,A)$        &         &         &         &         &         &         &   $-1$  &  $-1$   &         &         &                 \\
$(y,y,B)$        &         &         &         &         &         &         &   $-1$  &  $-1$   &         &         &                 \\
$(y,x(yx)^2,a)$  &   $-1$  &  $-3$   &   $1$   &   $1$   &  $-2$   &  $-2$   &    $2$  &   $2$   &         &         &                 \\
$(y,x(yx)^2,b)$  &    $1$  &  $3$    &  $-1$   &  $-1$   &   $2$   &  $1$    &    $2$  &   $1$   &         &         &                 \\
$(y,x(yx)^2,A)$  &    $1$  &  $1$    &  $-1$   &  $-3$   &   $2$   &  $2$    &   $-2$  &  $-2$   &         &         &                 \\
$(y,x(yx)^2,B)$  &  $-1$   &  $-1$   &   $1$   &   $3$   &   $2$   &  $1$    &    $2$  &   $1$   &         &         &                 \\
        \hline
$(x,x,x)$        &         &         &         &         &         &         &         &         &         &         &                 \\
$(y,y,y)$        &         &         &         &         &         &         &         &         &         &         &                 \\
$(y,y,(xy)^2x)$  &         &         &         &         & $-1$    &  $-1$   &  $-1$   &  $-1$   &         &         &   $-2$          \\ 
$(y,x(yx)^2,x)$  &   $-1$  &  $-1$   &  $-1$   &  $-1$   &         &         &         &         &         &         &    $2$          \\ 
$(y,x(yx)^2,yx)$ &   $-1$  &  $-1$   &  $-1$   &  $-1$   & $-1$    &  $-1$   &  $-1$   &  $-1$   &         &         &       \\ \hline           
\end{tabular}
\caption{The matrix $M_3$ computing $\widetilde{\partial_3}$ for the Weeks manifold. The first twelve rows are computations in filtration degree $2$ (i.e.\ they are the basis of $E^0_{2,1}$), and the last five rows are computations in filtration degree $3$ (i.e.\ they are the basis of $E^0_{3,0}$).}
\label{Tab:M3}
\end{table}

We compute $\widetilde{\partial_2}$ and $\widetilde{\partial_3}$ modulo the degree $0$ terms, and get the matrices $M_2$ resp.\ $M_3$ in Table~\ref{Tab:M2} resp.\ Table~\ref{Tab:M3}. First, we compute $H_0(Q, H^\ab)$. This is contained in the computations of the top-left corner of $M_2$. Let $X_1$ be this top left corner, which is a $(8 \times 4)$-matrix. Now $H^\ab$ is generated as an abelian group by the four basis elements $(a), (A), (b), (B)$ (which are \textit{not} a basis; indeed we have e.g.\ $(a) + (A) = 0$ in $H^\ab$). Now the module of co-invariants is just the quotient of this group by the image of $X_1^t$, which is generated by the columns of $X_1^t$, amounting to
\[
(a)-(A), (a), (A), (a) - (b), \: \text{and} \: (A) - (B).
\] 
Thus $(a) = 0$ in the quotient, and since $(a)-(b) = 0$, we also have $(b) = 0$ in the quotient. Thus $H^\ab = 0$, which we found already using Lemma~\ref{Lem:easy-SES}. Note, however, that we have here not made any direct use of the fact that $H^\ab$ has order $25$, or indeed that it is a finite group; the co-invariants will be $0$ regardless of the order, and this is deduced only from the specific action of $Q$ on the base group. 

Next, we compute $H_1(Q, \Z)$ and $H_2(Q, \Z)$; of course, since $Q = D_6$, this can easily be looked up, but the additional information contained in the chain maps coming from the section $s \colon Q \to G$ will be crucial in the case of $H_2$. Let $Y_2$ be the bottom right $(3 \times 2)$-minor in $M_2$, and let $Y_3$ be the bottom right $(5 \times 3)$-minor in $M_3$, i.e.\
\[
Y_2 = \begin{pmatrix}
2 & 0 \\ 
0 & 2 \\ 
0 & 0
\end{pmatrix}, \quad \text{and} \quad 
Y_3 = \begin{pmatrix}
0 & 0 & 0 \\
0 & 0 & 0 \\
0 & 0 & -2 \\ 
0 & 0 & 2 \\
0 & 0 & 0
\end{pmatrix} 
\]
corresponding to the filtration degree $2$ resp.\ $3$. Then $H_2(Q, \Z)$ is just isomorphic to the kernel of $Y_2^t$ modulo the image of $Y_3^t$. The image of $Y_2^t$ is spanned by $2(x)$ and $2(y)$, and hence $H_1(Q, \Z)$ is generated by the two order $2$ elements $[(x)]$ and $[(y)]$, i.e.\ $H_1(Q, \Z) \cong \Z_2^2$, as is to be expected from the dihedral group of order $12$. Next, notice that the kernel of $Y_2^t$ is spanned by the element $(y, x(yx)^2)$, and the image of $Y_3^t$ hits this basis element with coefficient $2$; thus $H_2(Q, \Z) \cong \Z / 2\Z$. Although $H_0(Q, H^\ab) = 0$, showing that the transgression $\tau \colon H_2(Q, \Z) \to H_0(Q, H^\ab)$ is the zero map, in general the information about $\tau$ is contained in the bottom left corner of $M_2$. Indeed, $H_2(Q, \Z)$ is generated by $(y, x(yx)^2)$, whose image in filtration degree $1$ is seen to be $(a) - (A) - (b) - (B)$. Thus $\im(\tau)$ is generated by precisely that element; of course, this image represents $0$ when projected into $H^\ab$, and of course also remains $0$ when projected into $H_0(Q, H^\ab)$.  

We now compute $H_1(Q, H^\ab)$, a basis for $H_3(Q, \Z)$, and the image of the transgression $\tau_3 \colon H_3(Q, \Z) \to H_1(Q, H^\ab)$. The information to compute $H_1(Q, H^\ab)$ is contained in the \textit{top left corner} of $M_3$ and $M_2$. First, the generators of $H_1(Q, H^\ab)$ come from the kernel of the transpose of the $(8 \times 4)$-minor in the top left corner of $M_2$. Let us write the element $(x,a)$ as $a_x$, and analogously for $b_x, A_x, B_x, a_y, b_y, A_y$, and $B_y$. Then reading those top left rows of $M_2$, we find that $H_1(Q, H^\ab)$ is generated by 
\[
\{ a_x + A_x, \: -a_x - b_x + B_x, \: a_y + b_y, \: A_y + B_y \}
\]
Notice, however, that in $\Z[(x), (y)] \otimes H^\ab$ we have the equalities $a_x + A_x = 0$ and $b_x + B_x = 0$. Thus the above generators give the generating set
\[
S = \{ -a_x - 2b_x, \: a_y + b_y \}
\]
for $H_1(Q, H^\ab)$. To find their images in $H_1(Q, H^\ab)$, we now quotient out by the image of $\widetilde{\partial_3}$ in this filtration degree, i.e.\ we take the rows of the top left $(12\times 8)$-minor of $M_3$, and quotient out $\Z[(x), (y)] \otimes H^\ab$ by these rows; thus, the first row yields the (trivial) relation $-a_x - A_x = 0$, the second row yields $-a_x - 2b_x = 0$, and the fifth row yields $- a_y - b_y = 0$. Thus both of the generators of $S$ lie in the image of our map; and hence $H_1(Q, H^\ab) = 0$. 

Finally, for completeness, we also show how to compute the transgression $E^2_{3,0} \to E^2_{1,1}$, even though we know its image to be zero. Note that $H_3(Q, \Z)$ is spanned by the kernel of $Y_3^t$. We rename the basis elements of our matrix as follows:
\[
e_1 = (x,x,x), \: e_2 = (y,y,y), \: e_3 = (y,y,(xy)^2x), \: e_4 = (y,x(yx)^2x), \: \text{ and } \: e_5 = (y,x(yx)^2yx).
\]
Then the kernel is easily seen to be generated by the elements $e_1, e_2, e_3 + e_4, \: \text{and} \: e_5$. In fact, were we to compute $\partial_4^Q$, which is easily doable, we would see that in $H_3(Q, \Z)$ we have $2[e_1] = 2[e_2] = 6[e_5] = 0$, and that $[e_3] + [e_4] = [e_5]$, giving that $H_3(Q, \Z) \cong \Z_2^2 \oplus \Z_6$, as expected from $Q = D_6$. However, we will not need this additional information for computing $H_2(G, \Z)$. Returning to the basis above, we can look in the \textit{bottom left corner} of $M_3$ to see that neither the image of $e_1$ nor of $e_2$ contributes any terms to the lower filtration degree (this corresponds to the fact that $s(x,x) = s(y,y) = 1$), whereas the contribution of $e_3+ e_4$ and $e_5$ both contribute to the term 
\[
- (x,a) - (x,b) - (x,A) - (x,B) - (y,a) - (y,b) - (y,A) - (y,B). 
\]
However, as in Step~4 of the algorithm, the element $(x,a)$ corresponds to $(x) \otimes [a]_\ab$ in $\Z[B] \otimes_{\Z} H^\ab$, and hence the above term becomes the equivalence class in $H_1(Q, H^\ab)$ of 
\[
- (x) \otimes ([a]+[b]+[A]+[B]) - (y) \otimes ([a]+[b]+[A]+[B]) = 0 - 0 = 0,
\]
since $[a]_\ab + [A]_\ab = 0$, and likewise for $b$. Thus the transgression $H_3(Q, \Z) \to H_1(Q, H^\ab)$ is the zero map (as expected, since, by the above computation, its range is $0$). We summarize this data into the spectral sequence for the homology of $G$ in Figure~\ref{Fig:ss-ex-Weeks} below. 

\begin{sseqdata}[name = weeks, homological Serre grading, xscale = 3, y axis gap = 25pt, x axis gap = 15pt, left clip padding = 15pt, right clip padding= 25pt, classes = {draw = none },differentials=blue, title={$E^\page_{p,q}$} ]
\class["\Z"](0,0)
\class["{\Z_5 \oplus \Z_5}"](0,1)
\class["{(\Z_5 \oplus \Z_5)^2}"](1,1)
\class["{(\Z_5 \oplus \Z_5)^3}"](2,1)
\class["{\Z^2}"](1,0)
\class["{\Z^3}"](2,0)
\class["{\Z^5}"](3,0)
\class["0"](0,2)

\d[""]1(1,1)
\d[""]1(2,1)
\d[""]1(1,0)
\d["{\widetilde{\partial}_{2}^Q}"]1(2,0)
\d["{\widetilde{\partial}_{3}^Q}"]1(3,0)

\replaceclass["{\Z}"](0,0)
\replaceclass["{0}"](0,1)
\replaceclass["{0}"](1,1)
\replaceclass["\Z_2\oplus \Z_2"](1,0)
\replaceclass["\Z_2"](2,0)
\replaceclass["\Z_6\oplus \Z_2 \oplus \Z_2"](3,0)

\d["{\tau_2}"]2(2,0)
\d["{\tau_3}"]2(3,0)
\end{sseqdata}

\begin{figure}[h]
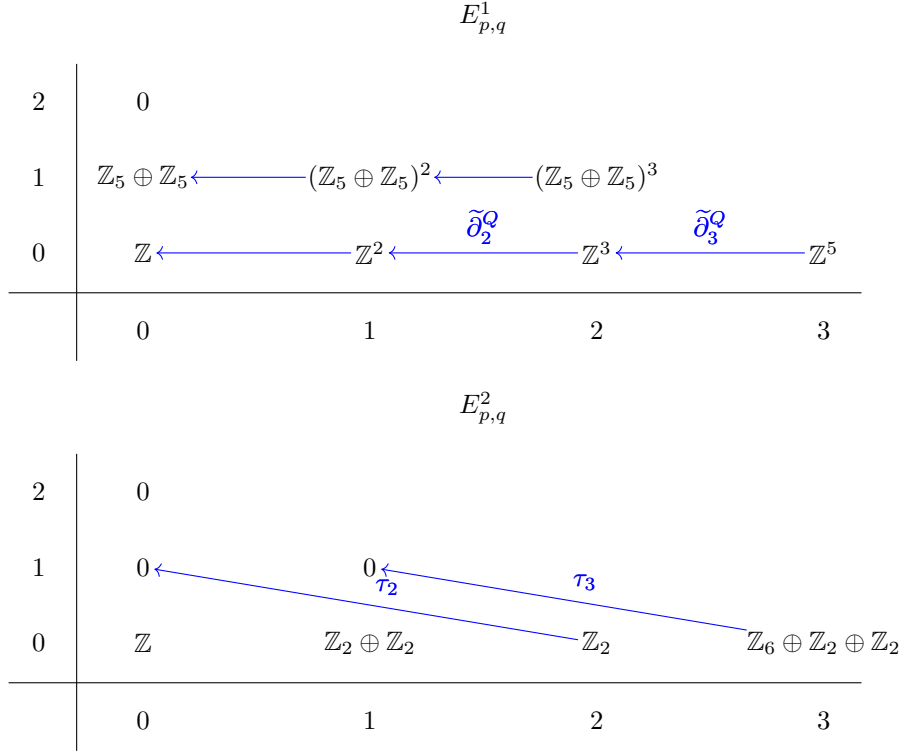

\begin{align*}
\printpage[ name = weeks, page = 1 ] \\
\printpage[ name = weeks, page = 2 ]
\end{align*}
\caption{The first and second page of the homological spectral sequence for the extension $G = \Gamma_{\mathcal{O}}$ of the Weeks manifold $\mathcal{W}$ by its full isometry group $\Isom(\mathcal{W}) \cong D_6$, showing that $H_1(\Gamma_{\mathcal{O}}, \Z) = \Z_2 \oplus \Z_2$, and $H_2(\Gamma_{\mathcal{O}}, \Z) = \Z_2$. The data for all morphisms drawn out is contained in the matrices $M_2$ and $M_3$ of Tables~\ref{Tab:M2} and \ref{Tab:M3}.}
\label{Fig:ss-ex-Weeks}
\end{figure}

Thus the diagonal $p+q=2$ has converged on the $E^2$-page, and so reading off the diagonal in Figure~\ref{Fig:ss-ex-Weeks}, we see that we have $H_2(G, \Z) \cong \Z / 2 \Z$ when $G = \Gamma_{\mathcal{O}}$, which is precisely what we wanted to show, and recovers our result from Proposition~\ref{Prop:sad-news}.

\section{Fibonacci Manifolds}\label{Sec:Fibonacci}

\noindent In this section, we will compute the second homology group of all orbifolds arising as quotients of the \textit{Fibonacci manifolds} $M_n$ by orientation-preserving isometries $Q \leq \Isom^+(M_n)$. We shall see that for $n = p > 3$ prime, these second homology groups coincide with that of $H_2(Q, \Z)$. The same will also be proved to hold for $n=4$.

We now give definitions and fix some notation. Let $n \geq 4$, let $K \subset S^3$ be the figure-eight knot, and let $Q_n$ be the orbifold obtained by $(n,0)$-Dehn filling on $K$. Let $M_n$ denote the $n$-fold cyclic covering of $S^3$ branched over the figure-eight knot. Then $M_n$ is a hyperbolic manifold, and $Q_n$ is a hyperbolic orbifold. These manifolds were studied in \cite{Helling1998}, see also \cite{Vesnin1995, Howie2017}. Furthermore, $M_n$ is known to be arithmetic if and only if $n=4, 5, 6, 8$ or $12$ \cite{Hilden1992}. We will consider the groups
\[
\Gamma_n := \pi_1(M_n) \quad \text{and} \quad \Delta_n := \pi_1^\orb(Q_n).
\]
Note that $\Gamma_n \leq_n \Delta_n$. It is known \cite{Helling1998} that $\Gamma_n$ is isomorphic to the Fibonacci group $\Fib(2,2n)$ introduced by Conway \cite{Conway1965}, given by the presentation 
\begin{equation}
\Fib(2,2n) = \pres{}{x_1, \dots, x_{2n}}{x_{i}x_{i+1} = x_{i+2} \quad (1 \leq i \leq 2n, \text{ taken $\operatorname{mod} 2n$)}}
\end{equation}
It is not difficult to use this presentation to find that 
\begin{equation}\label{Eq:abelianization-of-Fibonacci}
\Gamma_n^\ab \cong \begin{cases*} \Z_N \oplus \Z_{5N}, \: \: N = f_{n'}L_{n'} & when $n = 2n'$ \\ \Z_N \oplus \Z_{N}, \: \: \: N = L_{n} & when $n = 2n'+1$ \end{cases*}
\end{equation}
where $f_n$ and $L_n$ are, respectively, the Fibonacci--Lucas numbers defined by the linear recurrence
\[
s_{n+2} = s_{n+1} + s_n
\]
and the initial conditions 
\[
f_0 = 0, \: f_1 = 1, \quad \text{and} \quad L_0 = 2, \: L_1 = 1.
\]
In particular, $M_n$ is a rational homology $3$-sphere for all $n \geq 4$. 

It is known (see \cite{Vesnin1999}) that the full isometry group of $M_n$ is isomorphic to
\[
\Isom(M_n) \cong \pres{}{x, y}{x^{2n} = y^4 = (yx)^2 = (y^{-1}x)^2=1},
\]
which is a group of order $8n$ isomorphic to a semidirect product $C_n \rtimes D_8$. The action of these isometries is described explicitly in \cite{Vesnin1999}. The group of orientation-preserving isometries is generated by $x^2, y$ and is hence isomorphic to
\[
\Isom^+(M_n) \cong \pres{}{x, y}{x^n = y^4 = 1, yxy^{-1} = x^{-1}} \cong C_n \rtimes C_4
\]
of order $4n$, as proved by Maclachlan \& Reid \cite[p. 174]{Maclachlan1997}. The action of these orientation-preserving isometries can be described explicitly. 

We will now consider certain extensions of $\Gamma_n$ by subgroups of $\Isom^+(M_n)$, and compute their second homology groups. In the case that $n=p$ is prime, the number of extensions is uniformly bounded, and we can compute all second homology groups. Indeed, it is easy to see that for prime $p$, the group $\Isom^+(M_p)$ is isomorphic to $C_2 \times D_p$, which has exactly 10 conjugacy classes of subgroups $Q \leq C_2 \times D_p$. For non-prime $n$, the situation is more difficult owing to the large number of subgroups of $C_n \rtimes C_4$, but we will give a general necessary condition for the vanishing of the second homology group. 

\subsection{Prime case}\label{Subsec:Prime-Fibonacci}

The prime case is straightforward and requires no deep machinery. Throughout this section, we will fix a prime $p>3$ and consider the fundamental group $\Gamma_p = \pi_1(M_p) \cong \Fib(2,2p)$. Let $Q \leq \Isom^+(M_p)$, and let $G$ denote the corresponding extension of $\Gamma_p$ by $Q$. As usual, we let $E_{\bullet,\bullet}^2$ denote the terms of the homological Lyndon--Hochschild--Serre spectral sequence. We will describe the relevant homology groups $H_\bullet(Q, \Gamma_p^\ab)$, and use this to deduce our main result that $H_2(G, \Z) \cong H_2(Q, \Z)$ (see Theorem~\ref{Thm:prime-fibonacci}). We divide the cases into whether $Q$ is large or small, in the following sense:

\begin{itemize}
\item \textbf{$p$ does not divide $|Q|$.} Since $\Isom^+(M_p)$ has order $4p$, it follows that $Q$ is a $2$-group. Recall that $|\Gamma^\ab_p| = L_p^2$ by \eqref{Eq:abelianization-of-Fibonacci}. Note that since $L_0 = 2$ and $L_1 = 1$, we obviously have that $L_n \equiv 1 \pmod{2}$ for all odd $n \geq 1$. In particular, $L_p$ is odd. Thus $|Q|$ and $L_p$ are coprime, so by Lemma~\ref{Lem:easy-SES} we have an exact sequence
\[
0 \to H_2(G, \Z) \to H_2(Q,\Z) \xrightarrow{\tau} (\Gamma_p^\ab)_Q
\]
where $\tau \colon H_2(Q,\Z) \to \Gamma_p^\ab$ is the transgression. However, since $Q$ is a $2$-group, it follows that $H_2(Q, \Z)$ is a finite $2$-group, but $(\Gamma_p^\ab)_Q$ is necessarily of odd order since $\Gamma_p^\ab$ has odd order. Thus $\tau =0$ and $H_2(G, \Z) \cong H_2(Q, \Z)$, as desired. 
\item \textbf{$p$ divides $|Q|$.} First, we note that the module of co-invariants for the action of $Q$ on $\Gamma_p^\ab$ is trivial: indeed, the group $\Isom^+(M_p)$ contains an isometry $\rho \colon M_p \to M_p$ of order $p$, which acts on the generators $x_1, \dots, x_{2p}$ of the fundamental group $\Gamma_p$ by 
\[
\rho(x_i) = x_{i+2} \quad (i = 1, \dots, 2p)
\]
with indices taken modulo $2p$. Since $p$ divides $|Q|$ it follows by Cauchy's theorem that $Q$ contains an element of order $p$. But all elements of order $p$ of $\Isom^+(M_p)$ are conjugate to $\rho$, and hence by Mostow--Prasad rigidity we may assume without loss of generality that $\langle \rho \rangle \leq Q$. However, we also have $x_i x_{i+1} = x_{i+2}$ in $\Gamma_p$, and hence taking co-invariants we have $[x_{i+1}] = 0$ in $H_0(Q, \Gamma_p^\ab)$ for all $1 \leq i \leq 2p$, as desired. Thus $E_{0,1}^2 = 0$. 

Next, we show that we have $E_{1,1}^2=0$. It is well-known that the sequence of Lucas numbers $L_p$ satisfies a Gauss congruence, and so in particular by M\"obius inversion we have that $L_p \equiv 1 \pmod{p}$. Furthermore, it is obvious that for $p$ prime, we have that $L_p \equiv 1 \pmod{2}$, and hence $L_p \equiv 1 \pmod{2p}$. In particular, $\gcd(L_p^2, 2^kp)=1$, so by the same reasoning via the transfer map as in the proof of Lemma~\ref{Lem:easy-SES} we have $E_{1,1}^2=0$. Thus if $p$ divides $|Q|$, we have that $E_{1,1}^2=0$ and so $H_2(G, \Z) \cong H_2(Q,\Z)$ also in this case. 
\end{itemize}

This completes the proof of the following theorem.

\begin{theorem}\label{Thm:prime-fibonacci}
Let $p>3$ be prime, let $M_p$ be the Fibonacci manifold, and let $\Gamma_p = \pi_1(M_p)$. Let $Q \leq \Isom^+(M_p)$, and let $G = \pi_1^\orb(M_p / Q)$. Then $H_2(G, \Z) \cong H_2(Q, \Z)$. In particular, if $Q$ is a proper subgroup of $\Isom^+(M_p)$ which is not isomorphic to $C_2\times C_2$, then $H_2(G, \Z) = 0$. 
\end{theorem}

The last claim follows directly from the first together with the simple observation that $\Isom^+(M_p) \cong C_2 \times D_p$ has only two conjugacy classes of subgroups with non-trivial second integral homology group: $C_2\times C_2$, and the full group $\Isom^+(M_p)$, both of which have Schur multiplier $\Z_2$.

\subsection{The case of $\Gamma_4$}

In this section, we will deal with the special case of $n=4$, which has received some special attention in the literature. The group $\Gamma_4$ is easily seen to be given by the presentation
\[
\Gamma_4 = \pres{}{a,b}{ba^{-2}ba^{-1}b^2ab^2a^{-1} = a^2bab^2aba^2b^{-1} = 1}, \quad \text{so that} \quad \Gamma_4^\ab \cong \Z_3^2 \oplus \Z_5.
\]
It is arithmetic, and profinitely rigid \cite[Theorem~6.2]{Bridson2022}. The orientation preserving isometry group is $\Isom^+(M_4) \cong C_2 \times D_4$ which has order $16$. Let $Q \leq \Isom^+(M_4)$. Then $H_2(Q, \Z)$ is a finite $2$-group, whereas $\Gamma_4^\ab \cong \Z_3^2 \oplus \Z_5$. In particular, the transgression map
\[
H_2(Q, \Z) \xrightarrow{\tau} (\Gamma_4^\ab)_Q
\]
is necessarily trivial; that is, Lemma~\ref{Lem:easy-SES} immediately implies the following result.

\begin{theorem}\label{Thm:Fib-also-works-for-n=4}
Let $Q \leq \Isom^+(M_4)$, and let $G$ be the extension of $\Gamma_4$ by $Q$. Then $H_2(G, \Z) \cong H_2(Q, \Z)$. 
\end{theorem}

Recall that $\Delta_4$ is the cyclic extension of $\Gamma_4$ by the obvious outer automorphism of order $4$, permuting the generators of $\Gamma_4$ cyclically. As a corollary to Theorem~\ref{Thm:Fib-also-works-for-n=4}, we thus recover the following (easy) result due to Bridson \& Reid \cite[Lemma~7.5]{BridsonReidC}.

\begin{corollary}
With $\Delta_4 = \pi_1^\orb(Q_4)$ as above, we have $H_2(\Delta_4,\Z) = 0$.
\end{corollary}

Since Theorem~\ref{Thm:Fib-also-works-for-n=4} gives the second homology group for any lattice $\Gamma$ between $\Gamma_4$ and its extension by $\Isom^+(M_4)$, it is natural to compute the derived subgroup of those $\Gamma$ which have $H_2(\Gamma, \Z)=0$. If $G = [\Gamma, \Gamma]$, then Theorem~\ref{ThmC-BridsonReid} yields infinitely many Grothendieck pairs $P_\lambda \hookrightarrow G \times G$. Unfortunately, for all non-cyclic subgroups of $\Isom^+(M_4) \cong C_2 \times D_4$, we have that $H_2(Q, \Z)$ is a non-trivial $2$-group, and hence the corresponding $\Gamma$ has $H_2(\Gamma, \Z) \neq 0$. On the other hand, for all cyclic subgroups, we tend to have that $[\Gamma, \Gamma]$ is a proper subgroup of $\Gamma_4$, which still yields non-trivial Grothendieck pairs, but not in the direct product of lattices between $\Gamma_4$ and its normalizer in $\PSL_2(\mathbb{C})$.

\bibliographystyle{amsalpha}
\bibliography{ProfiniteRigidityH2}

@article {Agol2013,
    AUTHOR = {Agol, Ian},
     TITLE = {The virtual {H}aken conjecture},
      NOTE = {With an appendix by Agol, Daniel Groves, and Jason Manning},
   JOURNAL = {Doc. Math.},
  FJOURNAL = {Documenta Mathematica},
    VOLUME = {18},
      YEAR = {2013},
     PAGES = {1045--1087},
      ISSN = {1431-0635,1431-0643},
   MRCLASS = {20F67 (57Mxx)},
  MRNUMBER = {3104553},
MRREVIEWER = {Thomas\ Koberda},
       URL = {https://elibm.org/article/10000267},
}

@article {Anick1986,
    AUTHOR = {Anick, David J.},
     TITLE = {On the homology of associative algebras},
   JOURNAL = {Trans. Amer. Math. Soc.},
  FJOURNAL = {Transactions of the American Mathematical Society},
    VOLUME = {296},
      YEAR = {1986},
    NUMBER = {2},
     PAGES = {641--659},
      ISSN = {0002-9947,1088-6850},
   MRCLASS = {16A62 (13D03 55S10)},
  MRNUMBER = {846601},
MRREVIEWER = {V.\ P.\ Snaith},
       DOI = {10.2307/2000383},
       URL = {https://doi.org/10.2307/2000383},
}

@article {Baumslag1974,
    AUTHOR = {Baumslag, Gilbert},
     TITLE = {Residually finite groups with the same finite images},
   JOURNAL = {Compositio Math.},
  FJOURNAL = {Compositio Mathematica},
    VOLUME = {29},
      YEAR = {1974},
     PAGES = {249--252},
      ISSN = {0010-437X,1570-5846},
   MRCLASS = {20E25},
  MRNUMBER = {357615},
MRREVIEWER = {A.\ H.\ Rhemtulla},
}

@article {BridsonReidC,
    AUTHOR = {Bridson, Martin and Reid, Alan},
     TITLE = {Relatively hyperbolic groups, {G}rothendieck pairs, and uncountable profinite ambiguity among fibre products},
      NOTE = {arXiv:2507.15009},
   JOURNAL = {arXiv preprint},
  FJOURNAL = {arXiv preprint},
  YEAR = {2025}
}

@article {BMRS,
    AUTHOR = {Bridson, M. R. and McReynolds, D. B. and Reid, A. W. and
              Spitler, R.},
     TITLE = {Absolute profinite rigidity and hyperbolic geometry},
   JOURNAL = {Ann. of Math. (2)},
  FJOURNAL = {Annals of Mathematics. Second Series},
    VOLUME = {192},
      YEAR = {2020},
    NUMBER = {3},
     PAGES = {679--719},
      ISSN = {0003-486X,1939-8980},
   MRCLASS = {57M50 (11F06 20E18 20H10)},
  MRNUMBER = {4172619},
       DOI = {10.4007/annals.2020.192.3.1},
       URL = {https://doi.org/10.4007/annals.2020.192.3.1},
}

@article {Bridson2004,
    AUTHOR = {Bridson, Martin R. and Grunewald, Fritz J.},
     TITLE = {Grothendieck's problems concerning profinite completions and
              representations of groups},
   JOURNAL = {Ann. of Math. (2)},
  FJOURNAL = {Annals of Mathematics. Second Series},
    VOLUME = {160},
      YEAR = {2004},
    NUMBER = {1},
     PAGES = {359--373},
      ISSN = {0003-486X,1939-8980},
   MRCLASS = {20E18 (20E26)},
  MRNUMBER = {2119723},
MRREVIEWER = {Jan-Christoph\ Schlage-Puchta},
       DOI = {10.4007/annals.2004.160.359},
       URL = {https://doi.org/10.4007/annals.2004.160.359},
}

@book {Brown1982,
    AUTHOR = {Brown, Kenneth S.},
     TITLE = {Cohomology of groups},
    SERIES = {Graduate Texts in Mathematics},
    VOLUME = {87},
 PUBLISHER = {Springer-Verlag, New York-Berlin},
      YEAR = {1982},
     PAGES = {x+306},
      ISBN = {0-387-90688-6},
   MRCLASS = {20-02 (18-01 20F32 20J05 55-01)},
  MRNUMBER = {672956},
MRREVIEWER = {Ross\ Staffeldt},
}

@book {Book1993,
    AUTHOR = {Book, Ronald V. and Otto, Friedrich},
     TITLE = {String-rewriting systems},
    SERIES = {Texts and Monographs in Computer Science},
 PUBLISHER = {Springer-Verlag, New York},
      YEAR = {1993},
     PAGES = {viii+189},
      ISBN = {0-387-97965-4},
   MRCLASS = {68Q42 (03D05 68-02 68R15)},
  MRNUMBER = {1215932},
MRREVIEWER = {Matthias\ Jantzen},
       DOI = {10.1007/978-1-4613-9771-7},
       URL = {https://doi.org/10.1007/978-1-4613-9771-7},
}

@article {Borel1964,
    AUTHOR = {Borel, A. and Serre, J.-P.},
     TITLE = {Th\'{e}or\`emes de finitude en cohomologie galoisienne},
   JOURNAL = {Comment. Math. Helv.},
  FJOURNAL = {Commentarii Mathematici Helvetici},
    VOLUME = {39},
      YEAR = {1964},
     PAGES = {111--164},
      ISSN = {0010-2571,1420-8946},
   MRCLASS = {14.50},
  MRNUMBER = {181643},
MRREVIEWER = {T.\ Ono},
       DOI = {10.1007/BF02566948},
       URL = {https://doi.org/10.1007/BF02566948},
}

@article {Bridson2022,
    AUTHOR = {Bridson, M. R. and Reid, A. W.},
     TITLE = {Profinite rigidity, {K}leinian groups, and the cofinite {H}opf
              property},
   JOURNAL = {Michigan Math. J.},
  FJOURNAL = {Michigan Mathematical Journal},
    VOLUME = {72},
      YEAR = {2022},
     PAGES = {25--49},
      ISSN = {0026-2285,1945-2365},
   MRCLASS = {20H10 (20E18 22E40 30F40 57M50)},
  MRNUMBER = {4460248},
MRREVIEWER = {Alexander\ W.\ Mason},
       DOI = {10.1307/mmj/20217218},
       URL = {https://doi.org/10.1307/mmj/20217218},
}

@article {Button2005,
    AUTHOR = {Button, J. O.},
     TITLE = {Fibred and virtually fibred hyperbolic 3-manifolds in the
              censuses},
   JOURNAL = {Experiment. Math.},
  FJOURNAL = {Experimental Mathematics},
    VOLUME = {14},
      YEAR = {2005},
    NUMBER = {2},
     PAGES = {231--255},
      ISSN = {1058-6458,1944-950X},
   MRCLASS = {57N10 (57M50)},
  MRNUMBER = {2169525},
MRREVIEWER = {Colin\ C.\ Adams},
       URL = {http://projecteuclid.org/euclid.em/1128100134},
}

@book {Chenadec1986,
    AUTHOR = {Le Chenadec, Philippe},
     TITLE = {Canonical forms in finitely presented algebras},
    SERIES = {Research Notes in Theoretical Computer Science},
 PUBLISHER = {Pitman Publishing, Ltd., London; John Wiley \& Sons, Inc., New
              York},
      YEAR = {1986},
     PAGES = {xii+201},
      ISBN = {0-273-08721-5},
   MRCLASS = {68Q50 (03B35 03D03 08B05 20F05 68Q40)},
  MRNUMBER = {840218},
MRREVIEWER = {Friedrich\ Otto},
}

@article {Chinburg2001,
    AUTHOR = {Chinburg, Ted and Friedman, Eduardo and Jones, Kerry N. and
              Reid, Alan W.},
     TITLE = {The arithmetic hyperbolic 3-manifold of smallest volume},
   JOURNAL = {Ann. Scuola Norm. Sup. Pisa Cl. Sci. (4)},
  FJOURNAL = {Annali della Scuola Normale Superiore di Pisa. Classe di
              Scienze. Serie IV},
    VOLUME = {30},
      YEAR = {2001},
    NUMBER = {1},
     PAGES = {1--40},
      ISSN = {0391-173X,2036-2145},
   MRCLASS = {57M50 (11F06)},
  MRNUMBER = {1882023},
MRREVIEWER = {Colin\ C.\ Adams},
       URL = {http://www.numdam.org/item?id=ASNSP_2001_4_30_1_1_0},
}

@article {Conder2021,
    AUTHOR = {Conder, Marston},
     TITLE = {Two new proofs of the fact that triangle groups are
              distinguished by their finite quotients},
   JOURNAL = {New Zealand J. Math.},
  FJOURNAL = {New Zealand Journal of Mathematics},
    VOLUME = {52},
      YEAR = {2021 [2021--2022]},
     PAGES = {827--844},
      ISSN = {1171-6096,1179-4984},
   MRCLASS = {20E26 (20E18 20F36 20F67 20H10 20J06 22E40 57M07)},
  MRNUMBER = {4387996},
MRREVIEWER = {Jack\ O.\ Button},
       DOI = {10.53733/193},
       URL = {https://doi.org/10.53733/193},
}

@article {Conway1965,
    AUTHOR = {Conway, John},
     TITLE = {Advanced problem 5327},
   JOURNAL = {Amer. Math. Monthly},
  FJOURNAL = {American Mathematical Monthly},
    VOLUME = {72},
    ISSUE = {8},
      YEAR = {1965},
     PAGES = {915},
}

@article {Dixon1982,
    AUTHOR = {Dixon, John D. and Formanek, Edward W. and Poland, John C. and
              Ribes, Luis},
     TITLE = {Profinite completions and isomorphic finite quotients},
   JOURNAL = {J. Pure Appl. Algebra},
  FJOURNAL = {Journal of Pure and Applied Algebra},
    VOLUME = {23},
      YEAR = {1982},
    NUMBER = {3},
     PAGES = {227--231},
      ISSN = {0022-4049,1873-1376},
   MRCLASS = {20E18},
  MRNUMBER = {644274},
MRREVIEWER = {S.\ P.\ Demushkin},
       DOI = {10.1016/0022-4049(82)90098-6},
       URL = {https://doi.org/10.1016/0022-4049(82)90098-6},
}

@article {Gabai2009,
    AUTHOR = {Gabai, David and Meyerhoff, Robert and Milley, Peter},
     TITLE = {Minimum volume cusped hyperbolic three-manifolds},
   JOURNAL = {J. Amer. Math. Soc.},
  FJOURNAL = {Journal of the American Mathematical Society},
    VOLUME = {22},
      YEAR = {2009},
    NUMBER = {4},
     PAGES = {1157--1215},
      ISSN = {0894-0347,1088-6834},
   MRCLASS = {57M50 (51M10 51M25)},
  MRNUMBER = {2525782},
MRREVIEWER = {Stephan\ Tillmann},
       DOI = {10.1090/S0894-0347-09-00639-0},
       URL = {https://doi.org/10.1090/S0894-0347-09-00639-0},
}

@article {Groves1993,
    AUTHOR = {Groves, J. R. J. and Smith, G. C.},
     TITLE = {Soluble groups with a finite rewriting system},
   JOURNAL = {Proc. Edinburgh Math. Soc. (2)},
  FJOURNAL = {Proceedings of the Edinburgh Mathematical Society. Series II},
    VOLUME = {36},
      YEAR = {1993},
    NUMBER = {2},
     PAGES = {283--288},
      ISSN = {0013-0915,1464-3839},
   MRCLASS = {20F16 (68Q42)},
  MRNUMBER = {1221049},
MRREVIEWER = {Friedrich\ Otto},
       DOI = {10.1017/S0013091500018381},
       URL = {https://doi.org/10.1017/S0013091500018381},
}

@article {Grothendieck1970,
    AUTHOR = {Grothendieck, Alexander},
     TITLE = {Repr\'{e}sentations lin\'{e}aires et compactification profinie
              des groupes discrets},
   JOURNAL = {Manuscripta Math.},
  FJOURNAL = {Manuscripta Mathematica},
    VOLUME = {2},
      YEAR = {1970},
     PAGES = {375--396},
      ISSN = {0025-2611,1432-1785},
   MRCLASS = {20.80 (14.00)},
  MRNUMBER = {262386},
MRREVIEWER = {F.\ Oort},
       DOI = {10.1007/BF01719593},
       URL = {https://doi.org/10.1007/BF01719593},
}

@incollection {Groves1990,
    AUTHOR = {Groves, J. R. J.},
     TITLE = {Rewriting systems and homology of groups},
 BOOKTITLE = {Groups---{C}anberra 1989},
    SERIES = {Lecture Notes in Math.},
    VOLUME = {1456},
     PAGES = {114--141},
 PUBLISHER = {Springer, Berlin},
      YEAR = {1990},
      ISBN = {3-540-53475-X},
   MRCLASS = {20F10 (20J05)},
  MRNUMBER = {1092227},
MRREVIEWER = {Susan\ Hermiller},
       DOI = {10.1007/BFb0100735},
       URL = {https://doi.org/10.1007/BFb0100735},
}

@article {Guba1996,
    AUTHOR = {Guba, V. S. and Pride, S. J.},
     TITLE = {Low-dimensional (co)homology of free {B}urnside monoids},
   JOURNAL = {J. Pure Appl. Algebra},
  FJOURNAL = {Journal of Pure and Applied Algebra},
    VOLUME = {108},
      YEAR = {1996},
    NUMBER = {1},
     PAGES = {61--79},
      ISSN = {0022-4049,1873-1376},
   MRCLASS = {20M05 (20J10 20M50 68Q42)},
  MRNUMBER = {1382243},
MRREVIEWER = {Grigori\ I.\ \v{Z}itomirski\u{\i}},
       DOI = {10.1016/0022-4049(95)00038-0},
       URL = {https://doi.org/10.1016/0022-4049(95)00038-0},
}

@article {Guba1998,
    AUTHOR = {Guba, V. S. and Pride, S. J.},
     TITLE = {On the left and right cohomological dimension of monoids},
   JOURNAL = {Bull. London Math. Soc.},
  FJOURNAL = {The Bulletin of the London Mathematical Society},
    VOLUME = {30},
      YEAR = {1998},
    NUMBER = {4},
     PAGES = {391--396},
      ISSN = {0024-6093,1469-2120},
   MRCLASS = {20M50 (16E10 20M25)},
  MRNUMBER = {1620825},
MRREVIEWER = {J.\ K.\ Luedeman},
       DOI = {10.1112/S0024609398004676},
       URL = {https://doi.org/10.1112/S0024609398004676},
}

@article {Helling1998,
    AUTHOR = {Helling, H. and Kim, A. C. and Mennicke, J. L.},
     TITLE = {A geometric study of {F}ibonacci groups},
   JOURNAL = {J. Lie Theory},
  FJOURNAL = {Journal of Lie Theory},
    VOLUME = {8},
      YEAR = {1998},
    NUMBER = {1},
     PAGES = {1--23},
      ISSN = {0949-5932},
   MRCLASS = {57M07 (20F34 57M05 57N10)},
  MRNUMBER = {1616794},
MRREVIEWER = {Ronnie\ Lee},
}

@article {Hermiller1994,
    AUTHOR = {Hermiller, Susan M.},
     TITLE = {Rewriting systems for {C}oxeter groups},
   JOURNAL = {J. Pure Appl. Algebra},
  FJOURNAL = {Journal of Pure and Applied Algebra},
    VOLUME = {92},
      YEAR = {1994},
    NUMBER = {2},
     PAGES = {137--148},
      ISSN = {0022-4049,1873-1376},
   MRCLASS = {20F05 (20F55)},
  MRNUMBER = {1261121},
MRREVIEWER = {J.\ R. J. Groves},
       DOI = {10.1016/0022-4049(94)90019-1},
       URL = {https://doi.org/10.1016/0022-4049(94)90019-1},
}

@article {Hermiller1995,
    AUTHOR = {Hermiller, Susan and Meier, John},
     TITLE = {Algorithms and geometry for graph products of groups},
   JOURNAL = {J. Algebra},
  FJOURNAL = {Journal of Algebra},
    VOLUME = {171},
      YEAR = {1995},
    NUMBER = {1},
     PAGES = {230--257},
      ISSN = {0021-8693,1090-266X},
   MRCLASS = {20F32 (20F10)},
  MRNUMBER = {1314099},
MRREVIEWER = {Ian\ M.\ Chiswell},
       DOI = {10.1006/jabr.1995.1010},
       URL = {https://doi.org/10.1006/jabr.1995.1010},
}

@article {Hermiller1999,
    AUTHOR = {Hermiller, Susan M. and Meier, John},
     TITLE = {Artin groups, rewriting systems and three-manifolds},
   JOURNAL = {J. Pure Appl. Algebra},
  FJOURNAL = {Journal of Pure and Applied Algebra},
    VOLUME = {136},
      YEAR = {1999},
    NUMBER = {2},
     PAGES = {141--156},
      ISSN = {0022-4049,1873-1376},
   MRCLASS = {20F36 (57M07 68Q42)},
  MRNUMBER = {1674774},
       DOI = {10.1016/S0022-4049(97)00171-0},
       URL = {https://doi.org/10.1016/S0022-4049(97)00171-0},
}

@article {Hermiller1999b,
    AUTHOR = {Hermiller, Susan and Shapiro, Michael},
     TITLE = {Rewriting systems and geometric three-manifolds},
   JOURNAL = {Geom. Dedicata},
  FJOURNAL = {Geometriae Dedicata},
    VOLUME = {76},
      YEAR = {1999},
    NUMBER = {2},
     PAGES = {211--228},
      ISSN = {0046-5755,1572-9168},
   MRCLASS = {20F10 (20F65 57M07 57N10)},
  MRNUMBER = {1703216},
MRREVIEWER = {Michael\ L.\ Mihalik},
       DOI = {10.1023/A:1005064309732},
       URL = {https://doi.org/10.1023/A:1005064309732},
}

@incollection {Hilden1992,
    AUTHOR = {Hilden, Hugh M. and Lozano, Mar\'{\i}a Teresa and
              Montesinos-Amilibia, Jos\'{e} Mar\'{\i}a},
     TITLE = {The arithmeticity of the figure eight knot orbifolds},
 BOOKTITLE = {Topology '90 ({C}olumbus, {OH}, 1990)},
    SERIES = {Ohio State Univ. Math. Res. Inst. Publ.},
    VOLUME = {1},
     PAGES = {169--183},
 PUBLISHER = {de Gruyter, Berlin},
      YEAR = {1992},
      ISBN = {3-11-012598-6},
   MRCLASS = {57M25 (57M50)},
  MRNUMBER = {1184409},
MRREVIEWER = {Boris\ N.\ Apanasov},
}

@article {Hodgson1994,
    AUTHOR = {Hodgson, Craig D. and Weeks, Jeffrey R.},
     TITLE = {Symmetries, isometries and length spectra of closed hyperbolic
              three-manifolds},
   JOURNAL = {Experiment. Math.},
  FJOURNAL = {Experimental Mathematics},
    VOLUME = {3},
      YEAR = {1994},
    NUMBER = {4},
     PAGES = {261--274},
      ISSN = {1058-6458,1944-950X},
   MRCLASS = {57N10 (57M50)},
  MRNUMBER = {1341719},
       URL = {http://projecteuclid.org/euclid.em/1048515809},
}

@article {Howie2017,
    AUTHOR = {Howie, James and Williams, Gerald},
     TITLE = {Fibonacci type presentations and 3-manifolds},
   JOURNAL = {Topology Appl.},
  FJOURNAL = {Topology and its Applications},
    VOLUME = {215},
      YEAR = {2017},
     PAGES = {24--34},
      ISSN = {0166-8641,1879-3207},
   MRCLASS = {20F05 (57M05 57M50)},
  MRNUMBER = {3576437},
MRREVIEWER = {Daniel\ Matei},
       DOI = {10.1016/j.topol.2016.10.012},
       URL = {https://doi.org/10.1016/j.topol.2016.10.012},
}

@article {Kobayashi1990,
    AUTHOR = {Kobayashi, Yuji},
     TITLE = {Complete rewriting systems and homology of monoid algebras},
   JOURNAL = {J. Pure Appl. Algebra},
  FJOURNAL = {Journal of Pure and Applied Algebra},
    VOLUME = {65},
      YEAR = {1990},
    NUMBER = {3},
     PAGES = {263--275},
      ISSN = {0022-4049,1873-1376},
   MRCLASS = {18B99 (16S99 20M25 68Q42)},
  MRNUMBER = {1072284},
MRREVIEWER = {H.-J.\ Hoehnke},
       DOI = {10.1016/0022-4049(90)90106-R},
       URL = {https://doi.org/10.1016/0022-4049(90)90106-R},
}

@article {Kobayashi2007,
    AUTHOR = {Kobayashi, Yuji},
     TITLE = {The homological finiteness property {${\operatorname{FP}}_1$} and finite
              generation of monoids},
   JOURNAL = {Internat. J. Algebra Comput.},
  FJOURNAL = {International Journal of Algebra and Computation},
    VOLUME = {17},
      YEAR = {2007},
    NUMBER = {3},
     PAGES = {593--605},
      ISSN = {0218-1967,1793-6500},
   MRCLASS = {20M50},
  MRNUMBER = {2333373},
MRREVIEWER = {A.\ R.\ Rajan},
       DOI = {10.1142/S0218196707003743},
       URL = {https://doi.org/10.1142/S0218196707003743},
}

@article {Liu2023,
    AUTHOR = {Liu, Yi},
     TITLE = {Finite-volume hyperbolic 3-manifolds are almost determined by
              their finite quotient groups},
   JOURNAL = {Invent. Math.},
  FJOURNAL = {Inventiones Mathematicae},
    VOLUME = {231},
      YEAR = {2023},
    NUMBER = {2},
     PAGES = {741--804},
      ISSN = {0020-9910,1432-1297},
   MRCLASS = {57K32 (30F40 57M10)},
  MRNUMBER = {4542705},
MRREVIEWER = {Ken'ichi\ Yoshida},
       DOI = {10.1007/s00222-022-01155-4},
       URL = {https://doi.org/10.1007/s00222-022-01155-4},
}

@article {Maclachlan1997,
    AUTHOR = {Maclachlan, C. and Reid, A. W.},
     TITLE = {Generalised {F}ibonacci manifolds},
   JOURNAL = {Transform. Groups},
  FJOURNAL = {Transformation Groups},
    VOLUME = {2},
      YEAR = {1997},
    NUMBER = {2},
     PAGES = {165--182},
      ISSN = {1083-4362,1531-586X},
   MRCLASS = {57M50 (11B39)},
  MRNUMBER = {1451362},
MRREVIEWER = {Mark\ Brittenham},
       DOI = {10.1007/BF01235939},
       URL = {https://doi.org/10.1007/BF01235939},
}

@book {MaclachlanReidBook,
    AUTHOR = {Maclachlan, Colin and Reid, Alan W.},
     TITLE = {The arithmetic of hyperbolic 3-manifolds},
    SERIES = {Graduate Texts in Mathematics},
    VOLUME = {219},
 PUBLISHER = {Springer-Verlag, New York},
      YEAR = {2003},
     PAGES = {xiv+463},
      ISBN = {0-387-98386-4},
   MRCLASS = {57M50 (11R52)},
  MRNUMBER = {1937957},
MRREVIEWER = {Kerry\ N.\ Jones},
       DOI = {10.1007/978-1-4757-6720-9},
       URL = {https://doi.org/10.1007/978-1-4757-6720-9},
}

@article {Matveev1988,
    AUTHOR = {Matveev, S. V. and Fomenko, A. T.},
     TITLE = {Isoenergetic surfaces of {H}amiltonian systems, the
              enumeration of three-dimensional manifolds in order of growth
              of their complexity, and the calculation of the volumes of
              closed hyperbolic manifolds},
   JOURNAL = {Uspekhi Mat. Nauk},
  FJOURNAL = {Akademiya Nauk SSSR i Moskovskoe Matematicheskoe Obshchestvo.
              Uspekhi Matematicheskikh Nauk},
    VOLUME = {43},
      YEAR = {1988},
    NUMBER = {1(259)},
     PAGES = {5--22, 247},
      ISSN = {0042-1316},
   MRCLASS = {58F05 (57M99)},
  MRNUMBER = {937017},
MRREVIEWER = {A.\ Morimoto},
       DOI = {10.1070/RM1988v043n01ABEH001554},
       URL = {https://doi.org/10.1070/RM1988v043n01ABEH001554},
}

@book {MKS,
    AUTHOR = {Magnus, Wilhelm and Karrass, Abraham and Solitar, Donald},
     TITLE = {Combinatorial group theory: {P}resentations of groups in terms
              of generators and relations},
 PUBLISHER = {Interscience Publishers [John Wiley \& Sons], New
              York-London-Sydney},
      YEAR = {1966},
     PAGES = {xii+444},
   MRCLASS = {20.10},
  MRNUMBER = {207802},
MRREVIEWER = {Graham\ Higman},
}

@article {Mednykh1998,
    AUTHOR = {Mednykh, Alexander and Vesnin, Andrei},
     TITLE = {Visualization of the isometry group action on the
              {F}omenko-{M}atveev-{W}eeks manifold},
   JOURNAL = {J. Lie Theory},
  FJOURNAL = {Journal of Lie Theory},
    VOLUME = {8},
      YEAR = {1998},
    NUMBER = {1},
     PAGES = {51--66},
      ISSN = {0949-5932},
   MRCLASS = {57M60 (57M12 57N10)},
  MRNUMBER = {1616790},
}

@book {Morgan2014,
    AUTHOR = {Morgan, John and Tian, Gang},
     TITLE = {The geometrization conjecture},
    SERIES = {Clay Mathematics Monographs},
    VOLUME = {5},
 PUBLISHER = {American Mathematical Society, Providence, RI; Clay
              Mathematics Institute, Cambridge, MA},
      YEAR = {2014},
     PAGES = {x+291},
      ISBN = {978-0-8218-5201-9},
   MRCLASS = {53C44 (53C23)},
  MRNUMBER = {3186136},
MRREVIEWER = {Ryan\ D.\ Budney},
}

@article {Mostow1968,
    AUTHOR = {Mostow, G. D.},
     TITLE = {Quasi-conformal mappings in {$n$}-space and the rigidity of
              hyperbolic space forms},
   JOURNAL = {Inst. Hautes \'{E}tudes Sci. Publ. Math.},
  FJOURNAL = {Institut des Hautes \'{E}tudes Scientifiques. Publications
              Math\'{e}matiques},
    NUMBER = {34},
      YEAR = {1968},
     PAGES = {53--104},
      ISSN = {0073-8301,1618-1913},
   MRCLASS = {30.47},
  MRNUMBER = {236383},
MRREVIEWER = {J.\ R.\ Cannon},
       URL = {http://www.numdam.org/item?id=PMIHES_1968__34__53_0},
}

@article {Newman1942,
    AUTHOR = {Newman, M. H. A.},
     TITLE = {On theories with a combinatorial definition of
              ``equivalence.''},
   JOURNAL = {Ann. of Math. (2)},
  FJOURNAL = {Annals of Mathematics. Second Series},
    VOLUME = {43},
      YEAR = {1942},
     PAGES = {223--243},
      ISSN = {0003-486X},
   MRCLASS = {02.0X},
  MRNUMBER = {7372},
MRREVIEWER = {O.\ Ore},
       DOI = {10.2307/1968867},
       URL = {https://doi.org/10.2307/1968867},
}

@article {Nikolov2007,
    AUTHOR = {Nikolov, Nikolay and Segal, Dan},
     TITLE = {On finitely generated profinite groups. {I}. {S}trong
              completeness and uniform bounds},
   JOURNAL = {Ann. of Math. (2)},
  FJOURNAL = {Annals of Mathematics. Second Series},
    VOLUME = {165},
      YEAR = {2007},
    NUMBER = {1},
     PAGES = {171--238},
      ISSN = {0003-486X,1939-8980},
   MRCLASS = {20E18 (20E32 20F12)},
  MRNUMBER = {2276769},
MRREVIEWER = {Benjamin\ Klopsch},
       DOI = {10.4007/annals.2007.165.171},
       URL = {https://doi.org/10.4007/annals.2007.165.171},
}

@article {Nikolov2021,
    AUTHOR = {Nikolov, Nikolay and Segal, Dan},
     TITLE = {Constructing uncountably many groups with the same profinite
              completion},
   JOURNAL = {New Zealand J. Math.},
  FJOURNAL = {New Zealand Journal of Mathematics},
    VOLUME = {52},
      YEAR = {2021 [2021--2022]},
     PAGES = {765--771},
      ISSN = {1171-6096,1179-4984},
   MRCLASS = {20E18 (20E26 20F16)},
  MRNUMBER = {4387993},
MRREVIEWER = {Egle\ Bettio},
       DOI = {10.53733/89},
       URL = {https://doi.org/10.53733/89},
}

@incollection {Noskov1979,
    AUTHOR = {Noskov, G. A. and Remeslennikov, V. N. and Roman'kov,
              V. A.},
     TITLE = {Infinite groups},
 BOOKTITLE = {Algebra. {T}opology. {G}eometry, {V}ol. 17 ({R}ussian)},
    SERIES = {Itogi Nauki i Tekhniki},
     PAGES = {65--157, 308},
 PUBLISHER = {Akad. Nauk SSSR, Vsesoyuz. Inst. Nauchn. i Tekhn. Inform.,
              Moscow},
      YEAR = {1979},
   MRCLASS = {20-02},
  MRNUMBER = {584569},
MRREVIEWER = {Yu.\ I.\ Merzlyakov},
}

@article {Pickel1971,
    AUTHOR = {Pickel, P. F.},
     TITLE = {Finitely generated nilpotent groups with isomorphic finite
              quotients},
   JOURNAL = {Trans. Amer. Math. Soc.},
  FJOURNAL = {Transactions of the American Mathematical Society},
    VOLUME = {160},
      YEAR = {1971},
     PAGES = {327--341},
      ISSN = {0002-9947,1088-6850},
   MRCLASS = {20F05},
  MRNUMBER = {291287},
MRREVIEWER = {K.\ W.\ Gruenberg},
       DOI = {10.2307/1995809},
       URL = {https://doi.org/10.2307/1995809},
}

@article {Pickel1973,
    AUTHOR = {Pickel, P. F.},
     TITLE = {Nilpotent-by-finite groups with isomorphic finite quotients},
   JOURNAL = {Trans. Amer. Math. Soc.},
  FJOURNAL = {Transactions of the American Mathematical Society},
    VOLUME = {183},
      YEAR = {1973},
     PAGES = {313--325},
      ISSN = {0002-9947,1088-6850},
   MRCLASS = {20E99},
  MRNUMBER = {384940},
MRREVIEWER = {John\ D. P. Meldrum},
       DOI = {10.2307/1996471},
       URL = {https://doi.org/10.2307/1996471},
}

@article {Pickel1974,
    AUTHOR = {Pickel, P. F.},
     TITLE = {Metabelian groups with the same finite quotients},
   JOURNAL = {Bull. Austral. Math. Soc.},
  FJOURNAL = {Bulletin of the Australian Mathematical Society},
    VOLUME = {11},
      YEAR = {1974},
     PAGES = {115--120},
      ISSN = {0004-9727},
   MRCLASS = {20E15},
  MRNUMBER = {364455},
MRREVIEWER = {J.\ Wiegold},
       DOI = {10.1017/S0004972700043689},
       URL = {https://doi.org/10.1017/S0004972700043689},
}

@article {Platonov1986,
    AUTHOR = {Platonov, V. P. and Tavgen' , O. I.},
     TITLE = {On the {G}rothendieck problem of profinite completions of
              groups},
   JOURNAL = {Dokl. Akad. Nauk SSSR},
  FJOURNAL = {Doklady Akademii Nauk SSSR},
    VOLUME = {288},
      YEAR = {1986},
    NUMBER = {5},
     PAGES = {1054--1058},
      ISSN = {0002-3264},
   MRCLASS = {20E18 (11R34 11R70 19B37 22E40)},
  MRNUMBER = {852649},
MRREVIEWER = {L.\ N.\ Vaserstein},
}

@article {Prasad1973,
    AUTHOR = {Prasad, Gopal},
     TITLE = {Strong rigidity of {${\bf Q}$}-rank {$1$} lattices},
   JOURNAL = {Invent. Math.},
  FJOURNAL = {Inventiones Mathematicae},
    VOLUME = {21},
      YEAR = {1973},
     PAGES = {255--286},
      ISSN = {0020-9910,1432-1297},
   MRCLASS = {22E40 (53C35)},
  MRNUMBER = {385005},
MRREVIEWER = {M.\ S.\ Raghunathan},
       DOI = {10.1007/BF01418789},
       URL = {https://doi.org/10.1007/BF01418789},
}

@book {Rotman1979,
    AUTHOR = {Rotman, Joseph J.},
     TITLE = {An introduction to homological algebra},
    SERIES = {Pure and Applied Mathematics},
    VOLUME = {85},
 PUBLISHER = {Academic Press, Inc. [Harcourt Brace Jovanovich, Publishers],
              New York-London},
      YEAR = {1979},
     PAGES = {xi+376},
      ISBN = {0-12-599250-5},
   MRCLASS = {18-01 (16A62 18Gxx 20J05)},
  MRNUMBER = {538169},
}

@article {Squier1987,
    AUTHOR = {Squier, Craig C.},
     TITLE = {Word problems and a homological finiteness condition for
              monoids},
   JOURNAL = {J. Pure Appl. Algebra},
  FJOURNAL = {Journal of Pure and Applied Algebra},
    VOLUME = {49},
      YEAR = {1987},
    NUMBER = {1-2},
     PAGES = {201--217},
      ISSN = {0022-4049,1873-1376},
   MRCLASS = {20J05 (03B25 03D40 20F10 20M05 20M50)},
  MRNUMBER = {920522},
MRREVIEWER = {Kenneth\ S.\ Brown},
       DOI = {10.1016/0022-4049(87)90129-0},
       URL = {https://doi.org/10.1016/0022-4049(87)90129-0},
}

@article {Thurston,
    AUTHOR = {Thurston, William P.},
     TITLE = {The geometry and topology of three-manifolds},
      NOTE = {},
   JOURNAL = {Princeton lecture notes},
  FJOURNAL = {Princeton lecture notes},
  YEAR = {1978}
}

@article {Vesnin1995,
    AUTHOR = {Vesnin, A. Yu. and Mednykh, A. D.},
     TITLE = {Hyperbolic volumes of {F}ibonacci manifolds},
   JOURNAL = {Sibirsk. Mat. Zh.},
  FJOURNAL = {Rossi\u{\i}skaya Akademiya Nauk. Sibirskoe Otdelenie. Institut
              Matematiki im. S. L. Soboleva. Sibirski\u{\i} Matematicheski\u{\i} Zhurnal},
    VOLUME = {36},
      YEAR = {1995},
    NUMBER = {2},
     PAGES = {266--277, i},
      ISSN = {0037-4474},
   MRCLASS = {57N10 (20F05 57M12 57M50)},
  MRNUMBER = {1340395},
       DOI = {10.1007/BF02110146},
       URL = {https://doi.org/10.1007/BF02110146},
}

@article {Vesnin1999,
    AUTHOR = {Vesnin, A. Yu. and Rasskazov, A. A.},
     TITLE = {Isometries of hyperbolic {F}ibonacci manifolds},
   JOURNAL = {Sibirsk. Mat. Zh.},
  FJOURNAL = {Rossi\u{\i}skaya Akademiya Nauk. Sibirskoe Otdelenie. Institut
              Matematiki im. S. L. Soboleva. Sibirski\u{\i} Matematicheski\u{\i} Zhurnal},
    VOLUME = {40},
      YEAR = {1999},
    NUMBER = {1},
     PAGES = {14--29, i},
      ISSN = {0037-4474},
   MRCLASS = {57M50 (30F40)},
  MRNUMBER = {1686986},
MRREVIEWER = {Igor\ Belegradek},
       DOI = {10.1007/BF02674286},
       URL = {https://doi.org/10.1007/BF02674286},
}

@article {BMRSTriangle,
    AUTHOR = {Bridson, Martin R. and McReynolds, D. B. and Reid, Alan W. and
              Spitler, Ryan},
     TITLE = {On the profinite rigidity of triangle groups},
   JOURNAL = {Bull. Lond. Math. Soc.},
  FJOURNAL = {Bulletin of the London Mathematical Society},
    VOLUME = {53},
      YEAR = {2021},
    NUMBER = {6},
     PAGES = {1849--1862},
      ISSN = {0024-6093,1469-2120},
   MRCLASS = {20H10 (11F06)},
  MRNUMBER = {4386043},
MRREVIEWER = {Alexander\ W.\ Mason},
       DOI = {10.1112/blms.12546},
       URL = {https://doi.org/10.1112/blms.12546},
}

@phdthesis{Weeks1985,
  author  = {Weeks, J},
  title   = {Hyperbolic structures on $3$-manifolds},
  school  = {Princeton University},
  note =    {Ph.D thesis, Princeton University},
  year    = {1985}
}

@book {Wise2021,
    AUTHOR = {Wise, Daniel T.},
     TITLE = {The structure of groups with a quasiconvex hierarchy},
    SERIES = {Annals of Mathematics Studies},
    VOLUME = {209},
 PUBLISHER = {Princeton University Press, Princeton, NJ},
      YEAR = {[2021] \copyright 2021},
     PAGES = {x+357},
      ISBN = {[9780691170442]; [9780691170459]; [9780691213507]},
   MRCLASS = {20F65 (20F67)},
  MRNUMBER = {4298722},
MRREVIEWER = {Anthony\ Genevois},
}

@article {Xu2025,
    AUTHOR = {Xu, Xiaoyu},
     TITLE = {Profinite almost rigidity in 3-manifolds},
   JOURNAL = {Adv. Math.},
  FJOURNAL = {Advances in Mathematics},
    VOLUME = {480},
      YEAR = {2025},
    NUMBER = {part B},
     PAGES = {Paper No. 110505, 79},
      ISSN = {0001-8708,1090-2082},
   MRCLASS = {57K30 (20F34 20F65 57M05 57M10 57M50)},
  MRNUMBER = {4953004},
MRREVIEWER = {Anthony\ Genevois},
       DOI = {10.1016/j.aim.2025.110505},
       URL = {https://doi.org/10.1016/j.aim.2025.110505},
}

\section*{Appendix: verifying Algorithm 1}\label{Appendix:computations-for-general-appendix}

\noindent In this appendix, we verify the correctness of the Algorithm in \S\ref{Subsec:algorithm-description}. Specifically, we prove that the filtration we placed on the terms $P_\bullet$ of the Kobayashi resolution $(P_\bullet, \partial_\bullet)$ for the rewriting system $\cR_G$ defined in \S\ref{Subsec:rewriting-for-complete-extensions} is preserved by the chain maps, that the resulting spectral sequence agrees with the Lyndon--Hochschild-Serre spectral sequence in the relevant degrees, and that the transgression maps $d_{2,0}^2$ and $d_{3,0}^2$ are computed correctly by the algorithm. The computations are explicit and elementary, although rather intricate, relying only on the inductive definition of the chain maps of the Kobayashi resolution (defined in \S\ref{Subsec:Kobayashi-back}) and Lemma~\ref{Lem:Kobayashi-c}. We retain all notation from \S\ref{Subsec:algorithm-description}, including the fact that $G$ is an extension of $H$ by $Q$, and that $H_2(H, \Z) = 0$. 

\

\noindent\textbf{Computing $\partial_2$.} We begin with the computations of $\partial_2$. First, for $p=0$, let $(h_1, h_2) \in V^{(2)}$ be such that $h_1, h_2 \in A^\ast$. Then there is a corresponding rule $(h_1 h_2 \longrightarrow h_3) \in \cR_H \subseteq \cR_G$, with $h_3 \in A^\ast$. Since $\cR_G$ rewrites words over $A^\ast$ to words over $A^\ast$, it follows from Lemma~\ref{Lem:Kobayashi-c} that $\partial_2$ agrees with $\partial_2^H \otimes_{\Z H} \Z G$ on $(h_1, h_2)$, i.e.\ that $\partial_2(h_1, h_2) = \partial_2(h_1, h_2) \otimes_{\Z H} \Z G$. In particular, the filtration specified above is preserved by $\partial_2$. Furthermore, notice, that upon applying $- \otimes_{\Z G} \Z$ this becomes precisely the sum of the generators of $H$ in the defining relation $h_1 h_2 = h_3$, i.e.\ taking the quotient of the free abelian group on $\{ (a) \mid a \in A \}$ by the image of $\widetilde{\partial_2}$ on the set of all degree $0$ pairs $(h_1, h_2) \in V^{(2)}$ we obtain precisely $H_1(H, \Z)$. 

Next, for $p=1$, let $(q, h) \in V^{(2)}$ such that $q \in B^\ast$ and $h \in A^\ast$. Then $q \in B$ and $h \in A$, and
\begin{align*}
\partial_2(q, h) &= (q)\circ h - i_1 \partial_1 ( (q) \circ h) = (q) \circ h - i_1 \left( \overline{qh} - h \right) \\ 
&= (q) \circ h - i_1 ( \varphi_q(h)q - q) - (h)  = (q) \circ [h - 1] - \sum_{a \in A} (a) \circ \left( \frac{\partial \varphi_q(h)}{\partial a} - \frac{\partial h}{\partial a}\right). 
\end{align*}

Finally, for $p=2$, let $(q_1, q_2) \in V^{(2)}$ such that $q_1, q_2 \in B^\ast$. Then there is some corresponding rule $(q_1 q_2  \longrightarrow \alpha q_3) \in \cR_G$, where $\alpha \in A^\ast$ and $q_3 \in B^\ast$. We then have 
\begin{align}
\partial_2( q_1, q_2 ) &= (q_1)\circ q_2 - i_1 \partial_1 \left( (q_1) \circ q_2 \right) = 
(q_1) \circ q_2 - i_1 \left( \alpha q_3 - q_2 \right) \nonumber \\ 
&= \sum_{b \in B} (b) \circ \left(\frac{\partial q_1 q_2 }{\partial b} - \frac{\partial q_3 }{\partial b}  \right)  - \sum_{a \in A} (a) \circ \frac{\partial \alpha }{\partial a}. \label{Eq:d2-on-q1q2}
\end{align}
Notice that the first sum in this term is, syntactically, the same as in the computation of $\partial_2^Q (q_1, q_2)$, with the only difference being the interpretation of the coefficients (as elements of $G$ resp.\ elements of $Q$). In particular, when applying $- \otimes_{\Z G} \Z$, this sum becomes exactly $\widetilde{\partial_2^Q}(q_1, q_2)$. The second sum of course becomes $[\alpha]_{\ab}$, the image of $\alpha \in A^\ast$ in $H_1(H, \Z)$. 

Note that all of the above maps are effectively computable from a finite presentation of $H$ (as the map $H \twoheadrightarrow H^\ab$ is then effectively computable) along with a finite system $\cR_Q$ and section $s \colon \cR_Q \to A^\ast$. Summarizing our computations for $\partial_2$, and in the notation of the above, when applying $- \otimes_{\Z G} \Z$ we find that 
\begin{align}
\widetilde{\partial_2}(h_1, h_2) &= \widetilde{\partial_2^H}(h_1, h_2) \label{Eq:d2-0} \\ 
\widetilde{\partial_2}(q, h) &=  [h]_{\ab} - [\varphi_q(h)]_{\ab} =  [h]_{\ab} - q . [h]_{\ab}\label{Eq:d2-1} \\
\widetilde{\partial_2}(q_1, q_2) &= \widetilde{\partial_2^Q}(q_1, q_2) - [\alpha]_{\ab} \label{Eq:d2-2} 
\end{align}

\

\noindent \textbf{Computing $\partial_3$.} We will now compute $\partial_3$, where needed (and where possible). There are four types of basis elements of $P_3$, as in the filtration in Table~\ref{Tab:filtration}. For $p=0$, we note that since $\cR_H$ is not assumed to be finite or even computable, we cannot compute $\partial_3$ on these terms. However, just as in the case of $\partial_2$, we can see immediately from Lemma~\ref{Lem:Kobayashi-c} that $\partial_3 = \partial_3^H \otimes_{\Z H } \Z G$, and in particular that $\widetilde{\partial_3} = \widetilde{\partial_3^H}$. 

For $p=1$, i.e.\ on basis elements of the form $(q, h_1, h_2)$ with $q \in B^\ast$ and $h_1, h_2 \in A^\ast$, we see that $q \in B$ and $h_1 \in A$ and $h_2 \in A^\ast$, and that there is some rule of the form $(h_1 h_2 \longrightarrow h_3) \in \cR_H \subseteq \cR_G$, and of course a rule $(q h_1 \longrightarrow \varphi_q(h_1) q) \in \cR_G$. We compute: 
\begin{align*}
\partial_3(q, h_1, h_2) &= (q, h_1) \circ h_2 - i_2 \partial_2 \left( (q, h_1) \circ h_2 \right) \\ 
&= (q, h_1) \circ h_2 - i_2 \left( \underbrace{(q) \circ [h_1 - 1]h_2}_{:= X_1} - \underbrace{\sum_{a \in A} (a) \circ \left( \frac{\partial \varphi_q(h_1)}{\partial a} - \frac{\partial h_1}{\partial a}\right) h_2}_{:= X_2} \right).
\end{align*}
Notice that since $\varphi_q(h_1), h_1, h_2 \in A^\ast$, we have by Lemma~\ref{Lem:Kobayashi-c} that $i_2(X_2)$ is entirely supported in degree $p=0$, i.e.\ $i_2(X_2) \in \fg_0 P_2$, and we will not require this part for our purposes. We now compute $i_2(X_1)$. Let $h \equiv a_1 a_2 \cdots a_k$, where $k \geq 1$. Then  
\begin{align*}
i_2( (q) \circ h) &= i_2 ((q) \circ [a_1 a_2 \cdots a_k] ) = (q, a_1) \circ [a_2 \cdots a_k] + i_2 \left( i_1 \partial_1 \left( (q) \circ a_1 \right) \cdots [a_2 \cdots a_k] \right) \\ 
&= (q, a_1) \circ [a_2 \cdots a_k] + i_2 \left( i_1 \left( \varphi_q(a_1)q - a_1 \right) \circ [a_2 \cdots a_k] \right).
\end{align*}
Since there is no rule that begins with letters from $A$ and ends with a $q$, we see immediately that 
\[
i_2 \left( i_1 (\varphi_q(a_1)q - a_1) \circ [a_2 \cdots a_k] \right) + \fg_0 P_2 = i_2 \left( (q) \circ [a_2 \cdots a_k] \right) + \fg_0 P_2.
\]
Hence, by induction on $k \geq 0$, we have that 
\[
i_2( (q) \circ h) + \fg_0 P_2 = \sum_{i = 1}^k (q, a_i) \circ [a_{i+1} \cdots a_k]  + \fg_0 P_2 = \sum_{a \in A} (q, a) \frac{\partial h}{\partial a} + \fg_0 P_2.
\]
Thus we can compute $i_2(X_1) + \fg_0 P_2$ as 
\[
i_2(X_1) + \fg_0 P_2 = i_2( (q) \circ [h_3 - h_2] ) = \sum_{a \in A} (q,a) \left( \frac{\partial h_3}{\partial a} - \frac{\partial h_2}{\partial a }\right) + \fg_0 P_2.
\]
Since $h_1 \in A$, we can thus combine this with the above expression for $\partial_3(q, h_1, h_2)$ to obtain 
\[
\partial_3(q, h_1, h_2) + \fg_0 P_2 = \sum_{a \in A} (q, a)\left( \frac{\partial h_3}{\partial a} - \frac{\partial h_1 h_2}{\partial a }\right) + \fg_0 P_2.
\]
Hence $\partial_3$ respects the filtration, i.e.\ there are no elements of degree $2$ in the image of $\partial_3$ applied to elements of degree $1$ or $0$. Note also that we know that, as in the case of $\partial_2$, the expression $\frac{\partial h_3}{\partial a} - \frac{\partial h_1 h_2}{\partial a }$, when applying $- \otimes_{\Z G} \Z$, is the same as the image under $\widetilde{\partial_2^H}$ of the rule $(h_1 h_2 \to h_3)$, and thus taking the quotient of the free abelian group on $\{ (q, a) \mid q \in B, a \in A \}$ by the image of $\widetilde{\partial_2}$ on the set of all degree $0$ triples $(q, h_1, h_2) \in V^{(3)}$, we obtain the group $\bigoplus_{q \in B} H_1(H, \Z)$, where $[(q,a)]$ corresponds to the element that is $[a]_{\ab}$ in index $q$ and $0$ for all other indices.  

For $p=2$, i.e.\ on the basis elements of the form $(q_1, q_2, a)$ with $q_1, q_2 \in B^\ast$ and $a \in A$, we see as before that there is a rule $(q_1 q_2 \to \alpha q_3) \in \cR_G$. We see that 
\begin{align}
\partial_3 (q_1, q_2, a) &= (q_1, q_2) \circ a - i_2 \partial_2 \left( (q_1, q_2) \circ a \right)  \nonumber \\ 
&= (q_1, q_2) \circ a - i_2 \left( \left(\sum_{b \in B} (b) \circ \left(\frac{\partial q_1 q_2 }{\partial b} - \frac{\partial q_3 }{\partial b}  \right)  - \sum_{a \in A} (a) \circ \frac{\partial \alpha }{\partial a} \right) \circ a \right) \label{Eq:3-intermediate}
\end{align}
where for the second equality we have used the computation of $\partial_2$ from above. Notice first that since $\alpha \in A^\ast$, we have that
\[
i_2 \left( \sum_{a \in A} (a) \circ \frac{\partial \alpha }{\partial a} a \right) \in \fg_0 P_2.
\]
Now suppose that $q_1 q_2 \equiv b_1 b_2 \cdots b_k$ and $q_3 \equiv c_1 c_2 \cdots c_\ell$ for $b_i, c_j \in B$. Then 
\begin{equation}\label{Eq:d2-q1q2-expanded}
\sum_{b \in B} (b) \circ \left( \frac{\partial q_1 q_2 }{\partial b} - \frac{\partial q_3 }{\partial b}  \right) = \sum_{i=1}^k (b_i) \circ b_{i+1} \cdots b_k - \sum_{j=1}^\ell (c_j) \circ c_{j+1} \cdots c_\ell.
\end{equation}
For $1 \leq i \leq k$ and $1 \leq j \leq \ell$, write the conjugation action as
\begin{align*}
\varphi_{b[i]}(a) := \varphi_{b_{i+1} b_{i+2} \cdots b_k}(a), \quad \text{and} \quad
\varphi_{c[j]}(a) := \varphi_{c_{j+1} c_{j+2} \cdots c_\ell}(a),
\end{align*}
letting $\varphi_{b[k]}(a) = \varphi_{b[\ell]}(a) = \varphi_{\varepsilon}(a) = a$. These are, of course, effectively computable if $\varphi$ is effectively computable. Furthermore, we have 
\begin{align*}
b_{i+1} \cdots b_k a \xra{\cR_G} \varphi_{b[i]}(a) b_{i+1} \cdots b_k \equiv \overline{b_{i+1} \cdots b_k a}, \quad \text{and} \\
c_{j+1} \cdots c_\ell a \xra{\cR_G} \varphi_{c[j]}(a)c_{j+1} \cdots c_\ell \equiv \overline{c_{j+1} \cdots c_\ell a}.
\end{align*}
We thus see that 
\begin{align*}
i_2 \left( \sum_{b \in B} (b) \circ \left( \frac{\partial q_1 q_2 }{\partial b} - \frac{\partial q_3 }{\partial b}  \right) a \right) &= i_2 \left( \sum_{i=1}^k (b_i) \circ \varphi_{b[i]}(a) b_{i+1} \cdots b_k \right) \\ & - i_2 \left( \sum_{j=1}^\ell (c_j) \circ \varphi_{c[j]}(a) c_{j+1} \cdots c_\ell \right)
\end{align*}
Since $\varphi_{b[i]}(a), \varphi_{c[j]}(a) \in A^\ast$, we can write them as 
\begin{align*}
\varphi_{b[i]}(a) &\equiv \varphi^{(1)}_{b[i]}(a) \varphi^{(2)}_{b[i]}(a) \cdots \varphi^{(s)}_{b[i]}(a), \quad \text{and} \\
\varphi_{c[j]}(a) &\equiv \varphi^{(1)}_{c[j]}(a) \varphi^{(2)}_{c[j]}(a) \cdots \varphi^{(t)}_{c[j]}(a),
\end{align*}
for some $s, t \geq 0$, where all $\varphi^{(\mu)}_{b[i]}(a), \varphi^{(\nu)}_{c[j]}(a) \in A$ for $1 \leq \mu \leq s$ and $1 \leq \nu \leq t$. Note that $(b_i, \varphi^{(\mu)}_{b[i]}(a)) \in V^{(2)}$ for all indices, and analogously for $c$. Furthermore, no left-hand side of a rule which begins with a letter in $A$ will end in a letter from $B$, and for $j$ the word $c_{j} c_{j+1} \cdots c_\ell$ is irreducible mod $\cR_G$ since $\cR_G$ is reduced; and since $b_{i} b_{i+1} \cdots b_k$ is reducible if and only if $i = 1$ since $\cR_G$ is reduced; we thus easily compute that
\begin{align*}
i_2 \left( \sum_{i=1}^k (b_i) \circ \varphi_{b[i]}(a) b_{i+1} \cdots b_k \right) &= \sum_{i=1}^k \sum_{\mu = 1}^s (b_i, \varphi^{(\mu)}_{b[i]}(a) ) \circ b_{i+1} \cdots b_k + (q_1, q_2) \\
i_2 \left( \sum_{j=1}^\ell (c_j) \circ \varphi_{c[j]}(a) c_{j+1} \cdots c_\ell \right) &= \sum_{j=1}^\ell \sum_{\nu = 1}^t (c_j, \varphi^{(\nu)}_{c[j]}(a) ) \circ c_{j+1} \cdots c_\ell.
\end{align*}
Hence, we find that our original sum \eqref{Eq:3-intermediate} simplifies to 
\begin{align}
\partial_3(q_1, q_2, a) + \fg_0 P_2 = (q_1, q_2) \circ [a-1] &- \sum_{i=1}^k \sum_{\mu = 1}^s (b_i, \varphi^{(\mu)}_{b[i]}(a) ) \circ b_{i+1} \cdots b_k\nonumber \\ &+ \sum_{j=1}^\ell \sum_{\nu = 1}^t (c_j, \varphi^{(\nu)}_{c[j]}(a) ) \circ c_{j+1} \cdots c_\ell + \fg_0 P_2 \label{Eq:H1(Q,H1Z)-map}
\end{align}
But as is easily verified, this is just precisely the action of the relator $q_1q_2=q_3$ on the letter $a$, with respect to the usual resolution coming from the presentation complex of $Q$ tensored with $H_1(H, \Z)$, plus a correction term in $\fg_0 P_2$. In particular, taking the quotient of $E^1_{1,1} = \{ (q) \mid q \in B \} \otimes H_1(H, \Z)$ by the image of all such maps, we obtain precisely $H_1(Q, H_1(H, \Z))$. Finally, since the above action is certainly computable, since the basis element $(q_1, q_2, a)$ was arbitrary, and since there are only finitely many such triples if the rewriting system $\cR_Q$ for $Q$ is finite, it follows that $\im \partial_3 + \fg_0 P_2$ is computable when $\cR_Q$ is finite and $\varphi$ is computable. 

Finally, for $p=3$, we will compute $\partial_3$, again modulo $\fg_0 P_2$. Let $(q_1, q_2, q_3) \in V^{(3)}$ with $q_1, q_2, q_3 \in B^\ast$. We have, using our computation of \eqref{Eq:d2-on-q1q2}, that
\begin{align}
\partial_3( q_1, q_2, q_3) &= (q_1, q_2) \circ q_3 - i_2 \partial_2 ( (q_1, q_2) \circ q_3 ) \nonumber \\ 
&= (q_1, q_2) \circ q_3 - i_2 \left( \left( \sum_{b \in B} (b) \circ \left(\frac{\partial q_1 q_2 }{\partial b} - \frac{\partial q_3 }{\partial b}  \right)  - \sum_{a \in A} (a) \circ \frac{\partial \alpha }{\partial a} \right) q_3. \right) \label{Eq:11}
\end{align}
Note first that $ \frac{\partial \alpha }{\partial a}$ is entirely supported on words from $A^\ast$, it follows that for all $a \in A$, rewriting $\frac{\partial \alpha }{\partial a} q_3$ will give $q_3 \varphi_{q_3}\left( \frac{\partial \alpha }{\partial a}\right)$. But no rule of the form $(a, q)$ exists in $\cR_G$ for any $a \in A$ or $q \in B^+$, and hence it follows that 
\[
i_2 \left( \left( \sum_{a \in A} (a) \circ \frac{\partial \alpha }{\partial a} \right) q_3 \right) = i_2 \left( \sum_{a \in A} (a) \circ \varphi_{q_3} \left( \frac{\partial \alpha }{\partial a}\right) \right)  \in \fg_0 P_2,
\]
with the inclusion in $\fg_0 P_2$ following from Lemma~\ref{Lem:Kobayashi-c}, since words over $A^\ast$ are rewritten to words over $A^\ast$ by $\cR_G$. Thus, we focus on the first sum inside the $i_2$ above. As before, we expand as \eqref{Eq:d2-q1q2-expanded}, and we must therefore only compute 
\begin{equation}\label{Eq:i2-intermediate-0}
i_2 \left( \sum_{i=1}^k (b_i) \circ \overline{b_{i+1} \cdots b_k q_3} - \sum_{j=1}^\ell (c_j) \circ \overline{c_{j+1} \cdots c_\ell q_3} \right) 
\end{equation}
Now for every word $w \in B^\ast$, we have that the irreducible descendant modulo $\cR_G$ of $w$ is a word in $A^\ast B^\ast$, and furthermore if $w \xra{\cR_G} h v \in \Irr(\cR_G)$ with $h \in A^\ast$ and $v \in B^\ast$, then $w \xra{\cR_Q} v$. This is easily shown by induction on the number of rules applied; the only key idea is, of course, that an application of a conjugation rule $qa \longrightarrow \varphi_q(a)q$ does not affect the projection of a word to $B^\ast$. Thus, for all $1 \leq i \leq k$ and $1 \leq j \leq \ell$, let $h_i, h_i' \in A^\ast$ and $v_i, v_i' \in B^\ast$ be such that 
\[
\overline{b_{i+1} \cdots b_k q_3} \equiv h_i v_i \quad \text{and} \quad \overline{c_{j+1} \cdots c_\ell q_3} \equiv h'_j v'_j
\]
so that $h_iv_i, h_i' v_i' \in \Irr(\cR_G)$. Note that if $\cR_Q$ is finite, and the section map $s \colon \cR_Q \to A^\ast$ and $\varphi$ is computable, then $[h_i]_\ab$ and $[h_j']_\ab$ can all be effectively computed here, since $H$ is finitely presented. Now our expression \eqref{Eq:i2-intermediate-0} becomes
\begin{equation}\label{Eq:i2-intermediate-1}
i_2 \left( \sum_{i=1}^k (b_i) \circ h_i v_i - \sum_{j=1}^\ell (c_j) \circ h_j' v_j' \right).
\end{equation}
We write
\[
h_i \equiv h_i^{(1)} h_i^{(2)} \cdots h_i^{(s)} \quad \text{and} \quad h_j' \equiv {h_j^{'}}^{(1)} {h_j^{'}}^{(2)} \cdots {h_j^{'}}^{(t)}
\]
for $s, t\geq 0$ with $h_i^{(\mu)}, {h_j^{'}}^{(\nu)} \in A$ for all $1 \leq \mu \leq s$ and $1 \leq \nu \leq t$. Then by induction on $s$ and $t$ it is easy to see that \eqref{Eq:i2-intermediate-1} becomes
\begin{equation}\label{Eq:i2-intermediate-2}
\sum_{i=1}^k \sum_{\mu=1}^s (b_i, h_i^{(\mu)}) \circ v_i - \sum_{j=1}^\ell \sum_{\nu=1}^t (c_j, {h_j^{'}}^{(\nu)}) \circ v_j'  + i_2 \left( \sum_{i=1}^k (b_i) \circ v_i - \sum_{j=1}^\ell (c_j) \circ v_j' \right).
\end{equation}
Now notice that the first two sums in this expression are such that each $[h_i]_\ab = \sum_{\mu} [h_i^{(\mu)}]_{\ab}$ and $[h_j']_\ab = \sum_{\nu} [{h'_j}^{(\nu)}]_{\ab}$ are computable if $\cR_Q$ is finite, $\varphi$ is computable, and $s \colon \cR_Q \to A^\ast$ is computable. 

We rewrite the last term, i.e.\ 
\begin{equation}\label{Eq:i2-intermediate-3}
 i_2 \left( \sum_{i=1}^k (b_i) \circ v_i - \sum_{j=1}^\ell (c_j) \circ v_j' \right).
\end{equation}
We will prove that this is equal to applying $i_2^Q$ to the same terms, and inflating to $\Z G$, up to a correction term in $\fg_0 P_2$. We focus on a single entry $(b_i) \circ v_i$, and extend this argument by $\Z$-linearity. First, note that for any element $(a, b) \in V^{(2)}$ in the support of $i_2 ( (b_i) \circ v_i)$ (resp.\ $i_2( (c_j) \circ v_j')$) we have that $b_i v_i \xra{\cR_G} ab$ by Lemma~\ref{Lem:Kobayashi-c}.If $b_i v_i$ is irreducible, then $i_2 ( (b_i) \circ v_i) = 0$, so suppose that $b_i v_i$ is not irreducible. Let $w_i \in (A \cup B)^\ast$ be the result of applying a single rule of $\cR_G$ to $b_i v_i$ mod $\cR_G$. Since $v_i$ is irreducible, it follows that the rewriting must be left-most, and hence that there exists some (unique) prefix of $b_i v_i$ that is a left-hand side of a rule in $\cR_G$; hence 
\[
b_i v_i \equiv \ell_i V, \quad \text{and} \quad w_i \equiv r_i V,
\]
where $V \in B^\ast$ is some word and $(\ell_i \to r_i) \in \cR_G$. Let $\ell_i \equiv b_i \ell_i'$, where $\ell_i'$ is non-trivial. Since $\ell_i \in B^\ast$ as a subword of $b_i v_i \in B^\ast$, it follows that $r_i \in A^\ast B^\ast$. Write $r_i \equiv r_i' r_i''$ with $r_i' \in A^\ast$ and $r_i'' \in B^\ast$. Then $(\ell_i \to r_i'') \in \cR_Q$, by definition, and $s(\ell_i \to r_i'') \equiv r_i'$. Then 
\begin{align*}
i_2 ( (b_i) \circ v_i ) + \fg_0 P_2 &= (b_i, \ell_i') \circ V + i_2 \left( i_1 \partial_1 ( (b_i) \circ \ell_i') \circ V \right) + \fg_0 P_2 \\
&= (b_i, \ell_i') \circ V + i_2 \left( i_1 ( r_i' r_i''  - \ell_i' )  \circ V \right)  + \fg_0 P_2 \\
&= (b_i, \ell_i') \circ V + i_2 \left( (i_1 (r_i') r_i'') \circ V \right)  + i_2 \left( i_1(r_i'' - \ell_i') ) \circ V \right) + \fg_0 P_2\\  
&= (b_i, \ell_i') \circ V + i_2 \left( i_1(r_i'' - \ell_i') ) \circ V \right) + \fg_0 P_2 \\ 
&=  (b_i, \ell_i') \circ V + i_2^Q \left( i_1(r_i'' - \ell_i') ) \circ V \right) + \fg_0 P_2, \quad \text{(inductive hypothesis)} \\
&= i_2^Q \left( (b_i) \circ v_i \right) + \fg_0 P_2, 
\end{align*}
as desired. In the middle of this reasoning, we used the easy fact that
\[
i_2 \left( (i_1 (r_i') r_i'') \circ V \right) \in \fg_0 P_2.
\]
This is just the inductive hypothesis, since the word $\overline{r_i'' V}$ is a word in $A^\ast B^\ast$, and since this word is irreducible it follows that any element $(a,b) \in V^{(2)}$ in the support of $i_2 \left( i_1 (r_i') \overline{r_i'' V} \right)$ must have its first entry $a \in A$, and hence also its second entry $b \in A^+$, i.e.\ $(a,b) \in \fg_0 P_2$. Also note that the application of the inductive hypothesis is valid since $V$ is smaller than $v_i$ in the well-ordering described above (and the base case $V \equiv 1$ is trivial); we also used the fact that $i_1^Q$ agrees with $i_1$ once we inflate the coefficients from $\Z Q$ to $\Z G$. 

Thus the term \eqref{Eq:i2-intermediate-3} just becomes equal to 
\begin{equation}\label{Eq:i2-intermediate-4}
i_2^Q \left( \sum_{i=1}^k (b_i) \circ v_i - \sum_{j=1}^\ell (c_j) \circ v_j' \right) = i_2^Q ( \partial_2^Q (q_1, q_2) \circ q_3)
\end{equation}
with coefficients inflated to $\Z G$. However, by definition we have
\begin{equation}\label{Eq:d3-Q-computed}
\partial_3^Q(q_1, q_2, q_3) = (q_1, q_2) \circ q_3  - i_2^Q \partial_2^Q ( (q_1, q_2) \circ q_3 ).
\end{equation}
And hence, combining \eqref{Eq:i2-intermediate-2}, \eqref{Eq:i2-intermediate-4}, \eqref{Eq:d3-Q-computed}, and our initial expression \eqref{Eq:11}, we find that 
\[
\partial_3(q_1, q_2, q_3) + \fg_0 P_2 = \underbrace{\partial_3^Q(q_1, q_2, q_3)}_{\in \fg_2 P_2} - \underbrace{\sum_{i=1}^k \sum_{\mu=1}^s (b_i, h_i^{(\mu)}) \circ v_i + \sum_{j=1}^\ell \sum_{\nu=1}^t (c_j, {h_j^{'}}^{(\nu)}) \circ v_j'}_{\in \fg_1 P_2} + \fg_0 P_2,
\]
with coefficients inflated to $\Z G$. The first part $\partial_3^Q(q_1, q_2, q_3)$ with inflated coefficients lives entirely in $\fg_2 P_2$, and the second intermediate entries lie entirely in $\fg_1 P_2$. This completes the computation of $\partial_3$, as far as we shall need it. 

We noted already earlier that from the above computations in degree $p=0, 1$ that $\partial_3 (\fg_p P_3) \subseteq \fg_p P_{2}$ for all $0 \leq p \leq 3$, and hence our filtration is preserved for all $n \geq 0$. In particular, the spectral sequence $E^\bullet_{p,q}$ is well-defined, and in fact we have essentially computed the relevant terms of it for $H_2(G, \Z)$ already. We now summarize our computations of $\partial_3$ above, together with the result of applying $- \otimes_{\Z G} \Z$. Note that $\fg_0 P_2 \otimes_{\Z G} \Z = E^0_{0,2}$. With this in mind, we get:
\begin{align}
\widetilde{\partial_3} (h_1, h_2, h_3) &= \widetilde{\partial_3^H} (h_1, h_2, h_3) \label{Eq:d(hhh)}\\
\widetilde{\partial_3} (q, h_1, h_2) + E_{0,2}^0 &=  \sum_{a \in A} (q, a) \left( \frac{\partial h_3}{\partial a} - \frac{\partial h_1 h_2}{\partial a}\right) + E_{0,2}^0 \label{Eq:d(qhh)}\\
\widetilde{\partial_3} (q_1, q_2, a) + E_{0,2}^0 &= \sum_{i=1}^k \sum_{\mu = 1}^s (b_i, \varphi^{(\mu)}_{b[i]}(a) ) + \sum_{j=1}^\ell \sum_{\nu = 1}^t (c_j, \varphi^{(\nu)}_{c[j]}(a) ) + E_{0,2}^0 \label{Eq:d(qqh)} \\
\widetilde{\partial_3} (q_1, q_2, q_3) + E_{0,2}^0 &= \widetilde{\partial}_3^Q(q_1, q_2, q_3) - \sum_{i=1}^k \sum_{\mu=1}^s (b_i, h_i^{(\mu)}) + \sum_{j=1}^\ell \sum_{\nu=1}^t (c_j, {h_j^{'}}^{(\nu)}) + E_{0,2}^0 \nonumber \\ 
&= \underbrace{\widetilde{\partial}_3^Q(q_1, q_2, q_3)}_{\in E_{2,0}^0} - \underbrace{\sum_{i=1}^k [b_i, h_i]_\ab + \sum_{j=1}^\ell [c_j, h_j^{'}]_\ab}_{\in E_{1,1}^0} + E_{0,2}^0 \label{Eq:d(qqq)} 
\end{align}
We now use this to compute the relevant entries of our spectral sequence $E^\bullet_{p,q}$. 

\

\noindent\textbf{Computing the pages of the spectral sequence.} We begin with the bottom left corner. From our computations \eqref{Eq:d2-0}--\eqref{Eq:d2-2} of $\partial_2$ we have that
\begin{align*}
E^1_{0,1} &= E^0_{0,1} / \im (d_{0,2}^0) = E^0_{0,1} / \widetilde{\partial_2}(E^0_{0,2}) = H_1(H, \Z), \tag{by \eqref{Eq:d2-0}} \\ 
E^1_{1,0} &= E^0_{1,0} / \im (d_{1,1}^0) = E^0_{1,0} / \frac{\widetilde{\partial_2}(E^0_{1,1})}{E^0_{0,1}} = E^0_{1,0} \tag{by \eqref{Eq:d2-1}}
\end{align*}
Furthermore, passing to the next page i.e.\ considering the induced horizontal maps $d_{p,q}^1 \colon E^1_{p,q} \to E^1_{p-1,q}$, we find (since $\widetilde{\partial_1} = 0$) that 
\begin{align*}
E^2_{0,1} &= E^1_{0,1} / \im (d_{1,1}^1) = E^1_{0,1} / \widetilde{\partial_2}(E^1_{1,1}) = E^1_{0,1} / (q.[h]_\ab - [h]_\ab \mid q \in B, h \in A) \\
&= H_0(Q, H_1(H, \Z)), \\
E^2_{1,0} &= E^1_{1,0} / \im (d_{2,0}^1) = E^1_{1,0} / \frac{\widetilde{\partial_2}(E^1_{2,0})}{E^1_{0,1}} = \{ (b) \mid b \in Q\} / \im(\widetilde{\partial_2^Q}) = H_1(Q, \Z),
\end{align*}
where for the penultimate equality we used \eqref{Eq:d2-2}, and for the last equality we used the fact that $\cR_Q$ is a complete rewriting system defining $Q$. Note that our spectral sequence agrees with the Lyndon--Hochschild--Serre spectral sequence for these entries. 

Next, we compute the diagonal $p+q=2$. First, we have (just as for $E^1_{0,1}$ above) from \eqref{Eq:d(hhh)} and \eqref{Eq:d2-0} that 
\begin{align*}
E^1_{0,2} &\cong H_2(H, \Z).
\end{align*}
In all our applications, this group will be a black box, and furthermore we will need to assume it is $H_2(H, \Z) = 0$ in order to find $H_2(G, \Z)$. Indeed, even if we knew $H_2(H, \Z)$ abstractly, we would not know the transgression maps
\[
d^2_{2,1} \colon H_2(Q, H_1(H, \Z)) \to H_2(H, \Z) \quad \text{and} \quad d^3_{3,0} \colon H_3(Q, \Z) \to H_2(H, \Z)
\]
which are necessary to compute the quotient $E^4_{0,2} = E^\infty_{0,2}$, as this information vanishes into the black box $\fg_0 P_2$-terms in the chain map computations above. This explains why the assumption $H_2(H, \Z)=0$ will be crucial.

Next, for the term $E^2_{1,1}$, we note that $E_{1,1}^0$ is free abelian on $\{ (q, a) \mid q \in B, a \in A \}$. The map $d^0_{1,2}$ is induced by $\widetilde{\partial}_3$ on degree-$1$ triples $(q, h_1, h_2)$, and by \eqref{Eq:d(qhh)} its image in $E_{1,1}^0$ is $\sum_{a \in A}(q,a) \left( \frac{\partial h_3}{\partial a} - \frac{\partial h_1 h_2}{\partial a}\right)$, which records precisely the relation $[h_1]_\ab + [h_2]_\ab = [h_3]_\ab$ in the $(q)$-component upon applying $- \otimes_{\Z G} \Z$. Hence 
\[
E^1_{1,1} = E^0_{1,1} / \im(d_{1,2}^0) \cong \Z[B] \otimes_{\Z} H_1(H, \Z)
\]
where $(q,a)$ is identified with $q \otimes [a]_\ab$. By \eqref{Eq:d2-1}, the induced map $d_{1,1}^1 \colon E^1_{1,1} \to E_{0,1}^1$ sends $(q) \otimes [h]_\ab$ to $q . [h]_\ab - [h]_\ab$, which is the action of $Q$ on $H_1(H, \Z)$. By \eqref{Eq:d(qqh)}, the map $d_{2,1}^1 \colon E_{2,1}^1 \to E_{1,1}^1$ records the action of the relators of $Q$ on $H_1(H, \Z)$, and hence it is not difficult to verify that taking homology we obtain $E^2_{1,1} \cong H_1(Q, H_1(H, \Z))$, as was to be shown. 

Finally, from \eqref{Eq:d(qqq)} and \eqref{Eq:d2-2} we find that the chain complex $(E^1_{\bullet,0}, d^1_{\bullet,0})$ in degrees $1$ and $2$ agrees with the resolution $(P_\bullet^Q\otimes_{\Z Q} \Z, \widetilde{\partial}^Q_\bullet)$, and hence in particular that 
\begin{align*}
E^2_{2,0} = H_2(Q,\Z).
\end{align*}
Furthermore, we can also compute the transgression map. Namely, if we have an element $z = \sum_i a_i (q_1^{(i)}, q_2^{(i)})$ where each $(q_1^{(i)}, q_2^{(i)}) \in V^{(2)}$ and $a_i \in \Z$, such that $\widetilde{\partial_2}^Q(z) = 0$, then $\widetilde{\partial_2}(z)$ lands entirely in $E^1_{0,1}$, and projecting this down into $E^2_{0,1}$ yields precisely the map 
\[
d_{2,0}^2 \colon H_2(Q, \Z) \to H_0(Q, H_1(H, \Z)).
\]
Explicitly, we just map the sum $z = \sum_i a_i (q_1^{(i)}, q_2^{(i)})$ where $r_i = (q_1^{(i)}q_2^{(i)} \to s_i q_3^{(i)})$ to the (equivalence class modulo the action of $Q$ of the) element $\sum_i a_i [s_i]_{\ab}$. This is effectively computable from a presentation of $H$ and assuming $s$ and $\cR_Q$ are computable.

Finally, we have, by considering \eqref{Eq:d(qqq)}, a basis for the group $H_3(Q, \Z)$ can be computed (indeed, by computing $\widetilde{\partial}_4^Q$, which can be done since $\cR_Q$ is finite, it can also easily be shown that we can compute a presentation for $H_3(Q, \Z)$ with respect to this basis; however, this is not necessary to compute $H_2(G, \Z)$). For each such basis element, we can compute its image under $\partial_{3}$ modulo terms in $\fg_0 P_2$, and this image is precisely the transgression map
\[
d_{3,0}^2 \colon E_{3,0}^2 \to E_{1,1}^2.
\]
Thus we have computed the second, and final, transgression in our spectral sequence, and we are done with our computations.

Summarizing, if $\cR_Q$ is a finite complete rewriting system, the section map $s \colon \cR_Q \to A^\ast$ and the conjugation action $\varphi \colon B \times A \to A^\ast$ is explicitly computable, then we can explicitly compute all of the following terms of our spectral sequence
\[
E^2_{1,0}, E^2_{0,1}, E^2_{1,1}, E^2_{2,0}
\]
along with the two transgression maps
\[
d_{2,0}^2 \colon E^2_{2,0} \to E^2_{0,1}, \quad \text{and} \quad d_{3,0}^2 \colon E^2_{3,0} \to E^2_{1,1}
\]
where $d_{2,0}^2$ is a computable map of finitely generated abelian groups, and where the image of $d_{3,0}^2$ is computable.  The above information thus lets us compute 
\begin{align*} 
E^\infty_{1,1} &= E^3_{1,1} = E^2_{1,1} / \im(d_{3,0}^2), \quad \text{and} \\
E^\infty_{2,0} &= E^3_{2,0} = \ker(d_{2,0}^2) \leq H_2(Q, \Z). 
\end{align*}
In particular, if we know that $H_2(H, \Z) = 0$, then $E^\infty_{0,2} = 0$. This being the final piece of the puzzle of our spectral sequence on the diagonal $p+q=2$, and since $E^\infty_{p,q}$ converges to $H_{p+q}(G, \Z)$, we thus only need to solve the extension problem, having obtained an exact sequence 
\begin{equation}\label{Eq:exact-sequence-for-H2-final}
0 \to E^\infty_{1,1} \to H_2(G, \Z) \to E^\infty_{2,0} \to 0
\end{equation}
for $H_2(G, \Z)$. In particular, we can already decide the order of $H_2(G, \Z)$, and thus whether or not it is trivial. Furthermore, had we taken coefficients instead in a finite field, then \eqref{Eq:exact-sequence-for-H2-final} would split, so we can already compute $H_2(G, \mathbb{F}_q)$ for any prime power $q = p^k$. Solving the extension problem in general is not difficult from the data above, but it is notationally cumbersome; we omit most details. We summarize it instead as follows: to determine which class in $\Ext^1_\Z(E^\infty_{2,0}, E^\infty_{1,1})$ the extension \eqref{Eq:exact-sequence-for-H2-final} corresponds to, we must simply consider the image of the $\widetilde{\partial}_3$-map on the basis elements in filtration degree $3$ in $P_3 \otimes_{\Z G} \Z$. Each such image restricted to filtration degree $2$ will be a defining relation of $H_2(Q, \Z)$, and the corresponding term in filtration degree $1$ will project, upon taking homology until the $E^\infty_{2,0}$-page, to being precisely the section map from $E^\infty_{2,0}$ to $H_2(G, \Z)$, i.e.\ the required section map required to assemble $H_2(G, \Z)$ from the data in \eqref{Eq:exact-sequence-for-H2-final}. We leave the details to the interested reader, since we will primarily be interested in conditions for $H_2(G, \Z)$ being trivial or not, rather than its precise isomorphism type. We have, in any case, thus solved the extension problem with data we have already computed, thus completing the last step of the algorithm.

\end{document}